\documentclass[10pt,a4paper]{amsart}
\usepackage[toc,page]{appendix}
\usepackage{a4}
\usepackage[active]{srcltx}
\usepackage{amsmath,bbm,esint}
\usepackage{amssymb}
\usepackage{amsfonts}
\usepackage{mathrsfs} %% package of Lebesgue es Hausdorff measures
\usepackage{graphicx}
\usepackage{tikz}
\usetikzlibrary{decorations.pathreplacing,angles,quotes} % need to draw brackets in tikz

\usepackage{color}
\usepackage[colorlinks=true,urlcolor=blue,
citecolor=red,linkcolor=blue,linktocpage,pdfpagelabels,bookmarksnumbered,bookmarksopen]{hyperref}

\def\a{\alpha}

\def\d{\delta}

\def\k{\kappa}
\def\l{\lambda}

\def\vphi{\varphi}

\def\s{\sigma}

\def\t{\tau}
\def\e{\varepsilon}
\def\Om{\Omega}

\newcommand{\cB}{{\mathcal B}}

\newcommand{\cF}{{\mathcal F}}
\newcommand{\cG}{{\mathcal G}}

\newcommand{\cK}{{\mathcal K}}
\newcommand{\cJ}{{\mathcal J}}

\newcommand{\cU}{{\mathcal U}}

\newcommand{\fF}{{\mathfrak F}}

\newcommand{\R}{{\mathbb R}}

\newcommand{\N}{\mathbb{N}}

\newcommand{\E}{{\mathbb E}}
\newcommand{\ov}{\overline}

\newcommand{\weakly}{\ensuremath{\rightharpoonup}}
\newcommand{\weaklys}{\stackrel{\star}{\rightharpoonup}}
\newcommand{\da}{\ensuremath{\downarrow}}

\newcommand{\mres}{\mathbin{\vrule height 1.6ex depth 0pt width
0.13ex\vrule height 0.13ex depth 0pt width 1.3ex}}

\newcommand{\sB}{\mathscr{B}}
\newcommand{\sL}{\mathscr{L}}

\newcommand{\sP}{{\mathscr P}}
\newcommand{\sM}{{\mathscr M}}

\newcommand{\ds}{\displaystyle}

\newcommand{\spt}{{\rm{spt}}}
\newcommand{\Leb}{{\rm{Leb}}}

\newcommand{\diver}{\nabla\cdot}

\newcommand{\iif}{{\rm{if}}}
\newcommand{\iin}{{\rm{in}}}

\newcommand{\ac}{{\rm{ac}}}
\renewcommand{\ae}{{\rm{a.e.}}}
\newcommand{\dd}{\hspace{0.7pt}{\rm d}}
\newcommand{\id}{{\rm id}}

\renewcommand{\E}{{\textbf{E}}}
\newcommand{\tE}{\widetilde{\textbf E}}

\renewcommand{\v}{{\textbf{v}}}

\newcommand{\bW}{{\textbf{W}}}

\newcommand{\bm}{{\boldsymbol \mu}}
\newcommand{\bn}{{\boldsymbol\nu}}
\newcommand{\br}{{\boldsymbol{\rho}}}

\newcommand{\be}{\begin{equation}}
\newcommand{\ee}{\end{equation}}

\newcommand{\ba}{\begin{array}}
\newcommand{\ea}{\end{array}}

\newtheorem{remark}{\textbf{Remark}}[section]
\newtheorem{theorem}{\textbf{Theorem}}[section]
\newtheorem{lemma}[theorem]{\textbf{Lemma}}
\newtheorem{corollary}[theorem]{\textbf{Corollary}}
\newtheorem{proposition}[theorem]{\textbf{Proposition}}

\newtheorem{definition}[remark]{\textbf{Definition}}

\newtheorem{innercustomgeneric}{\customgenericname}
\providecommand{\customgenericname}{}
\newcommand{\newcustomtheorem}[2]{%
  \newenvironment{#1}[1]
  {%
   \renewcommand\customgenericname{#2}%
   \renewcommand\theinnercustomgeneric{##1}%
   \innercustomgeneric
  }
  {\endinnercustomgeneric}
}

\newcustomtheorem{customthm}{Framework}

\numberwithin{equation}{section}

\title[]{On nonlinear cross-diffusion systems: an optimal transport approach}  

\author[I. Kim]{Inwon Kim} 
\address{Department of Mathematics, UCLA, California, USA}
\email{ikim@math.ucla.edu}

\author[A.R. M\'esz\'aros]{Alp\'ar Rich\'ard M\'esz\'aros}  
\date{\today}
\address{Department of Mathematics, UCLA, California, USA}
\email{alpar@math.ucla.edu} 
\thanks{Inwon Kim is supported by the NSF grant DMS-1566578}

\begin{document}
\maketitle
\begin{abstract}
We study a nonlinear, degenerate cross-diffusion model which involves two densities with two different drift velocities.  A general framework is introduced based on its gradient flow structure in Wasserstein space to derive a notion of discrete-time solutions. Its continuum limit, due to the possible mixing of the densities, only solves a weaker version of the original system. In one space dimension, we find a stable initial configuration which allows the densities to be segregated. This leads to the evolution of a stable interface  between the two densities, and to a stronger convergence result to the continuum limit. In particular derivation of a standard weak solution to the system is available.  We also study the incompressible limit of the system, which addresses transport under a height constraint on the total density. In one space dimension we show that the problem leads to a two-phase Hele-Shaw type flow.
\end{abstract}

\section{Introduction}

%Let us consider a periodic box $\Om\subset\R^d.$ Typically $\Om$ is set to be  $\T^d:=\R^d/\Z^d,$  the $d$-dimensional flat torus. Let $m>1$ and $T>0.$  We study the following system 
Let $\Om$ be a bounded domain in $\R^d$ with $C^1$ boundary, and let  $T>0$ and $m>1$ be given constants. In this paper we study a gradient flow formulation of the following system in $[0,T]\times \Om$:
\be\label{eq:PME_m}
\left\{
\ba{l}
\partial_t\rho^1-\diver\left((\nabla p+\nabla\Phi_1)\rho^1\right)=0;\\ [5pt]
\partial_t\rho^2-\diver\left((\nabla p+\nabla\Phi_2)\rho^2\right)=0,  \\ [5pt]
\ea
\right.\tag{PME$_m$}
\ee
 where $\Phi_1,\Phi_2: \Om\to \R$ are given and the common diffusion term is generated by the pressure variable
\be\label{pressure_m}
p:=\frac{m}{m-1}(\rho^1+\rho^2)^{m-1}.
\ee
 In this article the system is subject to no flux condition on $[0,T]\times \partial\Om$ and is equipped with initial nonnegative densities $\rho^1_0,\rho^2_0\in L^1(\Om)$.

 %We denote the masses of the initial densities by 
%$$M_i:=\int_\Om\rho_0^i(x)\dd x,\; i=1,2$$

% Since the above system is a conservative, for any solution $(\rho^1,\rho^2)$ one has in particular $\int_\Om\rho^i_t(x)\dd x=M_i,\ \forall t\in[0,T].$

\medskip

Formally \eqref{eq:PME_m} can be seen as the gradient flow in Wasserstein (product) space of the free energy
\begin{equation}\label{free}
(\rho^1,\rho^2)\mapsto\int_{\Om} \frac{1}{m-1}(\rho^1+ \rho^2)^m\dd x + \int_\Om \Phi_1\rho^1\dd x + \int_\Om \Phi_2\rho^2\dd x.
\end{equation}

\medskip

Staying at the formal level, in the incompressible limit as $m\to+\infty$ where the first term in free energy turns into the constraint $\rho^1+\rho^2 \leq 1$, the corresponding system for the limiting density pair $(\rho^{1,\infty}, \rho^{2,\infty})$ is  
\be\label{eq:PME_infty}
\left\{
\ba{l}
\partial_t\rho^{1,\infty}-\diver\left((\nabla p^\infty+\nabla\Phi_1)\rho^{1,\infty}\right)=0;\\ [5pt]
\partial_t\rho^{2,\infty}-\diver\left((\nabla p^\infty+\nabla\Phi_2)\rho^{2,\infty}\right)=0,
\ea
\right.\tag{PME$_\infty$}
\ee
where the pressure $p^{\infty}$ is supported in the region $\{\rho^{1,\infty} + \rho^{2,\infty}=1\}$. When the densities $\rho^{i,\infty}$'s are characteristic functions with separate supports, the problem corresponds to a two-phase Hele-Shaw type flow with drifts.

\medskip

Our goal in this paper is to study the problems \eqref{eq:PME_m} and \eqref{eq:PME_infty} in the context of the aforementioned gradient flow, and verify the above heuristics. More precisely we will formulate the problem in terms of the discrete-time gradient flow (i.e. JKO or minimizing movement scheme) of the aforementioned free energy~\eqref{free}, posed in the product space equipped with the $2$-Wasserstein metric. Then we will study the solutions of this discrete scheme as the time step goes to zero. We will show that the limiting pair of densities $(\rho^1, \rho^2)$ satisfies a set of transport equation that will reduce to \eqref{eq:PME_m} under a stronger convergence assumption (for the precise statement we refer to Theorem~\ref{thm:precise}). To strengthen this result, it seems necessary to consider ``stable" initial configurations which avoids mixing: see the discussion below.  It turns out that in one dimension, in the setting of stable initial configurations which avoids mixing, a stronger convergence result holds and as a consequence the continuum limit densities satisfy \eqref{eq:PME_m} in the standard weak sense. Below is a  summary of our main results: precise statements are contained in the quoted theorems.

%\begin{theorem}[Theorem~\ref{thm:precise}]\label{general_D}
%For given $m\in (1,\infty)$, let $(\rho^{1,\tau}, \rho^{2,\tau})$ be the time-discrete solutions given by the minimizing movement scheme with time step size $\tau>0$, as given in \eqref{gf:tau}. Then as $\tau\da 0$ each density $\rho^{i,\tau}$ converges weakly in $L^{2m-1}([0,T]\times \Om)$, along a subsequence, to $\rho^i$ with $i=1,2$. Moreover $(\rho^1,\rho^2)$ satisfies the transport equations
%$$
%\partial_t\rho^i + \nabla\cdot(\v^i\rho^i)=0, \ i=1,2
%$$
%in the weak sense, where the vector fields $\v^i\in L^2(\dd\rho^i_t\otimes\dd t)$ satisfy 
%$$
%\v^1\rho^1+\v^2\rho^2=-\nabla (\rho^1+\rho^2)^m-\nabla\Phi_1\rho^1-\nabla\Phi_2\rho^2.
%$$
%If we know in addition that the convergence of the densities $\rho^{i,{\tau}}$ is pointwise a.e., then the pair $(\rho^1,\rho^2)$ satisfies \eqref{eq:PME_m} in the weak sense.
%\end{theorem}

The main results of the paper are obtained in one space dimension. Here we assume that the density with stronger drift in $x$ direction sits on the right side on the $x$-axis, i.e.,
\begin{equation}\label{stable0}
-\partial_x \Phi_1 \geq -\partial_x\Phi_2\quad  \hbox{ and }\quad x_1\geq x_2 \hbox{ for  } x_i \in \{\rho_0^i>0\},i=1,2.
\end{equation}

Under the assumption ~\eqref{stable0} the following theorems hold:
\begin{theorem}[Segregation of solutions: Proposition~\ref{prop:separation_1D}, Theorem~\ref{thm:sep_limit}]\label{one_dimension_1}
For given $m\in (1,\infty]$ and $\tau>0$, let $(\rho^{1,\tau}, \rho^{2,\tau})$ be the time-discrete solutions given by the minimizing movement scheme with time step size $\tau>0$, as given in \eqref{gf:tau}. Then, the pair stays ordered for all times $t>0$, i.e. $\rho^{1,\t}_t$ is supported to the right of the support of $\rho^{2,\t}_t$ for all $t>0$. Moreover, as $\tau\da 0$ each density $\rho^{i,\tau}$ converges weakly in $L^{2m-1}([0,T]\times \Om)$, along a subsequence, to $\rho^{i,m}$ for $i=1,2$. Also, along a subsequence $\rho^{i,\tau}$ converges pointwise a.e. to $\rho^{i,m}$ for $i=1,2$.  The pair of limiting densities $(\rho^{1,m},\rho^{2,m})$ solves  \eqref{eq:PME_m} in the weak sense and it stays ordered for all times.
\end{theorem}

\begin{theorem}[Convergence of weak solutions as $m\to\infty$: Theorem~\ref{thm:limit_m}]\label{one_dimension_2}
Let $(\rho^{1,m}, \rho^{2,m})$ be as given above for given $m>1$. Then as $m\to\infty$ and along a subsequence, the density pairs converge weakly in $L^p([0,T]\times \Om)$ for any $1<p<\infty$ to $(\rho^{1,\infty}, \rho^{2,\infty})$, which is a weak solution of \eqref{eq:PME_infty}.
\end{theorem}

\begin{theorem}[Patch solutions for \eqref{eq:PME_infty}: Proposition~\ref{prop:patch}]\label{one_dimension_3}
Let $m=\infty$ and suppose, in addition to ~\eqref{stable0}, that $\partial_{xx}^2\Phi_i \geq 0$ and $\rho_0^i = \chi_{(a^i(0), b^i(0))}$. Then $\rho^{i,\infty}$ remains a patch for all $t>0$, i.e. $\rho^{i,\infty}_t = \chi_{(a^i(t), b^i(t))}$, $i=1,2$. The density pair in this case is a solution to a two-phase Hele-Shaw type flow with drifts.
\end{theorem}

\begin{remark} 
%\begin{itemize}
%\hspace{-.3cm}%\item[(a)] 
For the linear diffusion $m=1$ the logarithmic entropy $\int \rho\ln\rho \dd x $ replaces $\frac{1}{m-1}\int \rho^m\dd x$ in the free energy, resulting in slightly different, however mostly parallel, analysis. Let us point out that in this case the sum of the densities is always positive, and thus they will always form an interface between each other in the event of segregation.\\
%\item[(b) ]Our results extend to the cross-diffusion system with different diffusion constants $\k_1,\k_2>0$:
%\be\label{eq:k_1-k_2:intro}
%\left\{
%\ba{l}
%\partial_t\rho^1-\diver\left(\k_1\nabla p\rho^1+\nabla\Phi_1\rho^1\right)=0;\\ [5pt]
%\partial_t\rho^2-\diver\left(\k_2\nabla p\rho^2+\nabla\Phi_2\rho^2\right)=0,
%\ea
%\right.
%\ee
%where the pressure is as given in \eqref{pressure_m}. In particular Theorem~\ref{one_dimension} holds under the generalized condition on the potentials 
%\be\label{new_cond}
%-\partial_x\Phi_1/\k_1\ge -\partial_x\Phi_2/\k_2.
%\ee
%This condition is in line with stability conditions imposed in other types of cross diffusion systems, such as in \cite{LorLorPer} and the references therein. As a consequence we positively answer a conjecture in \cite{BerGurHilPel}, which deals with \eqref{eq:k_1-k_2:intro} with $m=2$ and zero drifts (see the Remark after \cite[Theorem 1]{BerGurHilPel}). This conjecture asserts that no mixing of the densities can occur, if the initial densities were segregated. We show that this phenomenon is indeed true, and we refer to Section \ref{sec:existence-1D} for more detailed discussions.
%\item[(c)] The single species version of Theorem~\ref{one_dimension}(2)-(3) has been studied in \cite{MauRouSan1} and \cite{AleKimYao} in general dimensions.
%\end{itemize}
\end{remark}

%\medskip

%{\bf Literature}

\medskip
Let us discuss now the existing results from the literature that are relevant to our work. 
The single density version of the system \eqref{eq:PME_infty} has been introduced in \cite{MauRouSan1} in the gradient flow setting, the free boundary characterization and its links to \eqref{eq:PME_m} has been studied in \cite{AleKimYao}. These and similar systems received a lot of attention in the past a few years (see for instance \cite{MauRouSan2,MauRouSanVen, MauVen} and the references therein). These models are strongly related to the so-called Hele-Shaw models, as we can see in \cite{AleKimYao} (for other references we direct the reader to \cite{Kim,KimPoz,MPQ,PerQuiVaz} and to the references therein).

\medskip

Cross diffusion systems arise naturally from mathematical biology. 
These appear either as systems of reaction-diffusion equations (as in \cite{LorLorPer, JunZam} for instance) or systems of advection-diffusion equations (as in \cite{BerGurHil, BerGurHilPel, DamMeuMauRou, ZamJun} for instance). These systems appear also in fluid mechanics, such as the thin film approximation of the Muskat problem studied in \cite{EscLauMat,LauMat}. Beside the PDE approach for degenerate parabolic systems (used in most of the above references),  more recently the optimal transport and gradient flow theories have been adopted to study these systems.  For a non-exhaustive list of the fast-growing references in this direction we refer to \cite{BerBurPie, BurDiFPieSch, BurDiFFagSte, BruMurRanWol, CanGalMon1, CanGalMon2, CarLab1, CarLab2, DamMeuMauRou, DiFFag, LabPhD, Lab, LauMat}.

\medskip

Most of the aforementioned papers concern systems including separate diffusion terms for the evolution of each densities.  Such feature enables tracking of separate densities in the evolution. This is in contrast to our case where recovering separate densities out of the dynamic system appears to be out of reach unless the densities are guaranteed to be segregated. There are very few results available for systems without separate diffusions, we mention here a few particular papers in this direction.  In \cite{BerGurHilPel} the authors study the well-posedness of the system \eqref{eq:PME_m} in one space dimension when $m=2$, $\Phi_1=\Phi_2\equiv 0$ within the class of segregated solutions. This is a particular case of our model where the total density satisfies a degenerate parabolic equation, from which one can construct a segregated solution. In \cite{Ott1} Otto studies in one dimension the case  $m=\infty$ with gravity potentials $\Phi_i(x) = C_i x$ and with full saturation, i.e. with the condition $\rho_1+\rho_2=1$. There the mixing profile of one density correponds to an entropy solution of a Burger's type equation. This interesting description of mixing phenomena remains open to be extended beyond the specific setting given in the paper.
Lastly in the recent paper \cite{BurDiFFagSte} the authors study existence and segregation properties of one dimensional stationary solutions for systems of similar form to ours, when $m=2$ and the drift is generated by interaction energies. 

\bigskip

{\it Main difficulties and ingredients}

\medskip

\medskip

As mentioned above, our main challenge lies in the fact that the densities may mix into each other during the evolution, which indeed happens with the ``unstable" initial configurations where the densities are initially positioned in the opposite order to the equilibrium solution (see the discussion in Section~\ref{sec:eq_sol}). Such situation indicates low regularity of each densities, hindering the system from being well-posed. Indeed, in general we are only able to obtain strong convergence on the sum of the two density variables in the continuum limit, as we will see in Theorem \ref{thm:precise}. Naturally this reasoning leads to the question of whether one can formulate a ``stable" initial configuration to  obtain a stronger result. This question, while under investigation by the authors, stands open beyond the one dimensional result for segregated solutions, Theorems~\ref{one_dimension_1}, \ref{one_dimension_2},  \ref{one_dimension_3}. 

\medskip

\medskip

In terms of gradient flows, the challenge lies in the lack of available estimates or convexity properties. Though not surprising in the context of the above discussion, this is an interesting contrast to the single species case, where for instance stability of the discrete gradient flow solutions based on $\lambda$-convexity properties played an important role in the analysis.  For us the higher order space regularity estimates in the JKO scheme are available only for functions of the sum of the two densities. For similar models considered in \cite{LabPhD} and \cite{LauMat} this difficulty was overruled by presence of separate diffusions, or ``separate entropies'' of the form $\e(\int f(\rho^1)+g(\rho^2)\dd x)$ in the free energy \eqref{free}, however estimates obtained here do not carry through as $\e\da 0$. Let us also mention that the flow interchange technique introduced in \cite{MatMcCSav},  which has been quite successful to analyze some non-convex gradient flow systems such as in \cite{DiFMat} or \cite{LauMat}, does not appear to be applicable to our system.  Thus here we derive all our estimates relying  only on the first order optimality conditions satisfied by the discrete in time minimizers in the JKO scheme (see for instance the proof of Theorem~\ref{thm:reg_1step}).  This procedure is rather natural yet appears to be unexploited in the literature for similar models.

%Indeed, the presence of these terms allows the authors to obtain strong enough estimates on the two densities separately, thus they obtain the existence of standard weak solutions for their models. It is easy to see that the free energy containing the previously described terms $\Gamma$-converges (w.r.t. the weak convergence of measures) as $\e\da 0$ to the free energy in \eqref{free}, however this convergence is too weak to construct usual weak solutions for our systems \eqref{eq:PME_m} and \eqref{eq:PME_infty} by this manner. 

%Instead, our analysis relies on the study of finer properties of the supports of minimizers in the discrete-time JKO scheme. We show in particular that if the initial densities are chosen to be stable w.r.t. the potentials, this property remains invariant under iterations. This, together with the estimates on the sum of the two densities give us sufficient regularity to pass to the limit with the time discretization parameter and obtain a weak segregated solution in the continuum limit. 

\medskip

{\it Structure of the paper}

\medskip

In Section \ref{sec:MM_m}  the discrete-time scheme for the gradient flow is introduced, set in the $W_2$-product space. In Section \ref{subsec:prop_min}  the properties of discrete-time minimizers are studied.  Here we observe that while the total density $\rho^1+\rho^2$ is relatively regular (Lipschitz continuous), separate densities may be segregated and discontinuous. The segregation of densities with respect to the ordering properties of their potentials are more obvious in Section \ref{sec:eq_sol}, where one discusses the equilibrium solutions. Such segregation and ordering property suggests that fingering and mixing is inevitable for densities starting from ``unstable" initial configurations, to position themselves into the stationary profile.  

\medskip

In Section \ref{sec:PDE_m} we analyze the continuum limit of discrete-time solutions by studying their convergence modes as the time step size is sent to zero. We show that the limit solution satisfies a system of transport equations which can be interpreted as a generalized solutions for the system \eqref{eq:PME_m}. We also introduce the standard notion of weak solution for our systems and show that the continuum limit satisfies this notion when pointwise convergence holds for separate densities. It remains an open question whether the densities indeed converge pointwise, i.e. whether we can track down the position of each density in the evolution of the problem in general framework or in general dimension. 

\medskip

Section \ref{sec:segregated_1D} is devoted to the analysis  in one space dimension, where we consider stable initial configurations that line up with the strength order of the drift potentials. In this setting we are able to guarantee that solutions stay segregated with an evolving interface between them. As a consequence it follows that pointwise convergence holds for each densities, which in turn yields the existence of weak solutions for the system \eqref{eq:PME_m}.  The continuum solutions of \eqref{eq:PME_m} are then shown to converge as $m$ tends to infinity to a weak solution of \eqref{eq:PME_infty} along a subsequence. Furthermore when the drift is compressive (or incompressible), we show that patch solutions appear, yielding a solution to the two-phase Hele-Shaw flow. 

\medskip

Finally, in the Appendices \ref{sec:appendix_ot} and \ref{sec:appendix_aubin-lions} we recall some results from the theory of optimal transport and a refined version of Aubin-Lions lemma respectively. %{\color{red}The results from both of these appendices are crucial in our analysis.}

\medskip
{\it Acknowledgements}

\medskip

The authors are thankful to G. Carlier, D. Matthes, F. Santambrogio and Y. Yao for many valuable discussions at different stages of the preparation of this paper. The authors warmly thank the referee for his/her constructive comments and remarks.

\section{Minimizing movement schemes and properties of the minimizers}\label{sec:MM_m}

\subsection{Setting and notations}
Let us introduce the setting of the problem and some notations. Let $\Om\subseteq\R^d$ be a bounded domain with smooth boundary. We denote by $\sP(\Om)$ the space of probability measures on $\Om$. For $M>0$ we denote by $\sP^M(\Om)$ the space of finite nonnegative Radon measures on $\Om$ ($\sM_+(\Om)$) with mass $M$. 

For a Borel measurable map $T:\Om\to\Om$ and $\mu,\nu\in\sP^M(\Om)$, we say that $T$ {\it pushes forward} $\mu$ {\it onto} $\nu$, and write $\nu=T_\#\mu$ if $\nu(B)=\mu(T^{-1}(B))$ for every $B\subseteq\Om$ Borel measurable set. Using test functions, the definition of pushforward translates to 
$$\int_{\Om}\phi(y)\dd\nu(y)=\int_{\Om}\phi(T(x))\dd\mu(x),\ \ \forall\phi:\Om\to\R, \ \ \text{bounded and measurable}.$$
We equip the space $\sP^M(\Om)$ with the well-known $2$-Wasserstein distance $W_{2,M}$, i.e. For $\mu,\nu\in \sP^M(\Om),$
$$
W_{2,M}^2(\mu,\nu):=
\min\left\{\int_{\Om\times\Om} |x-y|^2\,\dd\gamma\;:\;\gamma\in\Pi^M(\mu,\nu)\right\},
$$
where
$\Pi^M(\mu,\nu)$ is the set of the so-called {\it transport plans}, i.e. $\Pi^M(\mu,\nu):=\{\gamma\in\sP^{M^2}(\Om\times\Om):\,(\pi^x)_{\#}\gamma=\mu,\,(\pi^y)_{\#}\gamma=\nu\}$. In particular if $\mu\ll\sL^d\mres\Om$ then the previous problem has a unique solution, which is of the form $\gamma_{T}:=(\id,T)_\#\mu$.
Here, in particular we adjusted the usual distance defined on probability measures to measures having mass $M>0$. Since it shall be clear from the context, from now on we write $W_2$ instead of $W_{2,M}.$  On the forthcoming pages we shall use classical results from the optimal transport theory. All of these can be found for instance in \cite{OTAM, AmbGigSav, villani}.

We denote by $\sM^d(\Om)$ the space of finite vector-valued Radon measures on $\Om.$ If $\E\in\sM^d(\Om),$ we denote by $|\E|$ its variation.  We denote the subspaces of absolutely continuous measures (w.r.t. $\sL^d\mres\Om$) by $\sP^{\ac}(\Om), \sP^{\ac,M}(\Om),$ etc.; we always identify these absolutely continuous measures with their densities and write $\rho$ instead of $\rho\cdot\sL^d$ or $\rho\dd x.$  If $\rho\in\sM_+^{\ac}(\Om)$ and $c\ge 0$ by $\{\rho>c\}$ we mean the set (up to $\sL^d$-negligible sets) where $\rho(x)>c$ a.e. In particular a property holds a.e. in $\{\rho>0\}$ if and only if it holds $\rho-$a.e. Notice also that $\{\rho>0\}\subseteq_{\text{a.e.}}\spt(\rho).$ For a measurable set $B\subset\R^d,$ we denote the set of its Lebesgue point by $\Leb(B).$

\subsection{Minimizing movements}

The heart of our analysis is the well-known {\it minimizing movement} or {\it JKO} scheme (see for instance \cite{Amb, AmbGigSav, San,JKO}) on a product Wasserstein space.

Let us introduce the functionals.
We consider $\cF_m,\cF_\infty:\sP^{M_1}(\Om)\times\sP^{M_2}(\Om)\to\R\cup\{+\infty\}$ and $\cG:\sP^{M_1}(\Om)\times\sP^{M_2}(\Om)\to\R$ to be defined as
\be\label{def:F_m}
\cF_m(\br)=\left\{
\ba{ll}
\ds\int_\Om\frac{1}{m-1}(\rho^1(x)+\rho^2(x))^m\dd x, & \iif\ (\rho^1+\rho^2)^m\in L^1(\Om),\\ [12pt]
+\infty, & {\rm{otherwise}},
\ea
\right.
\ee
\be\label{def:F_infty}
\cF_\infty(\br)=\left\{
\ba{ll}
0, & \iif\ \|\rho^1+\rho^2\|_{L^\infty}\le 1,\\ [12pt]
+\infty, & {\rm{otherwise}},
\ea
\right.
\ee
and 
$\cG:\sP^{M_1}(\Om)\times\sP^{M_2}(\Om)\to\R$
\be\label{def:G}
\cG(\br)=\int_\Om\Phi_1(x)\dd\rho^1(x)+\int_\Om\Phi_2(x)\dd\rho^2(x),
\ee 
where $\br:=(\rho^1,\rho^2)$,  $m>1$ is fixed and $\Phi_1,\Phi_2:\Om\to\R$ are given continuous potentials. Notice that $\cF_\infty$ is the indicator function (in the sense of convex analysis) of the set
$$\cK_1:=\left\{ (\rho^1,\rho^2)\in\sP^{M_1,\ac}(\Om)\times\sP^{M_2,\ac}(\Om): \rho^1+\rho^2\le 1\ \ae\right\}.$$

%We define $\cJ:\sP^{M_1}(\Om)\times\sP^{M_2}(\Om)\to\R\cup\{+\infty\}$ as 
%$$\cJ:=\cF+\cG.$$ 

It is classical that $\cF_m,\cF_\infty$ and $\cG$ are l.s.c. w.r.t. the weak convergence of measures on $\sP^{M_1}(\Om)\times\sP^{M_2}(\Om)$. It is immediate to see that they are convex, moreover $\cF_m$ is also strictly convex (in the usual sense) on $\sP^{M_1}(\Om)\times\sP^{M_2}(\Om)$. We remark also that 
in general $\cF_m$ is not {\it displacement convex} (in the sense of \cite{Mcc}) on the product space. To see this, let us consider for simplicity $m=2$. In this case, for $\rho^1,\rho^2\in L^2(\Om)$ we can write $\cF_2(\rho^1,\rho^2)=\int_\Om(\rho^1)^2\dd x + \int_\Om (\rho^2)^2\dd x+2\int_\Om \rho^1\rho^2\dd x.$ If $\cF_2$ would be $\l$-displacement convex (for some $\l\in\R$), then the map $\rho^1\mapsto\cF_2(\rho^1,\rho^2)$ would share at least the same modulus of convexity for any $\rho^2\in\sP^{M_2}(\Om)\cap L^2(\Om)$ fixed. While the first term in the development of $\cF_2$ is $0$-displacement convex and the second term is a constant for fixed $\rho^2$, the last term would be $\l$-displacement convex if and only if $\rho^2$ would be $\l$-convex, i.e. $D^2\rho^2\ge\l I_d$ in the sense of distributions. However,  $\rho^2$ can be chosen in a way that the lower bound on its Hessian is arbitrarily negative. Therefore, this term fails to be $\l$-displacement convex for any $\l\in\R$ and so does the  functional $\cF_2$.

\medskip

We proceed as in the classical setting (see for instance \cite{AmbGigSav,JKO}): we define a recursive sequence of densities associated to a fixed time step $\t$, then we introduce suitable interpolations between these densities and take the limit as $\t\downarrow 0$.  %In order to perform the limiting procedure, one needs to obtain compactness for the interpolated curves indexed by $\t,$ which requires the study of the regularity of the minimizers at each step.  The limit curve of densities at the end (if it exists) represents a weak solution of a continuity equation with a velocity field that can be seen as the Wasserstein gradient of the associated functional.

Let us introduce now the scheme. For this, we consider $\t>0$  a fixed time step and $N\in\N$ such that $N\t =T.$ Let $(\rho^1_0,\rho^2_0)$ be two given initial densities. For all $k\in\{0,\dots,N\}$ we define $\br_k^\t:=(\rho^{1,\t}_k,\rho^{2,\t}_k)$ as
$$\br_0^\t=(\rho^{1,\t}_0,\rho^{2,\t}_0):=(\rho^1_0, \rho^2_0)$$
and for $k\ge 0$
\be\label{gf:tau}
\br^\t_{k+1}=(\rho^{1,\t}_{k+1},\rho^{2,\t}_{k+1})={\rm{argmin}}_{\br\in\sP^{M_1}(\Om)\times\sP^{M_2}(\Om)}\left\{\cF_m(\br)+\cG(\br)+\frac{1}{2\t}\bW_2^2(\br,\br^\t_{k})\right\}.\tag{$\text{MM}_m$}
\ee
In this scheme either $m>1$ but finite, or $m=\infty.$ Here $\bW_2$ denotes the Wasserstein distance on the product space $\sP^{M_1}(\Om)\times\sP^{M_2}(\Om)$, i.e. $\bW_2^2(\bm,\bn):=W_2^2(\mu^1,\nu^1)+W_2^2(\mu^2,\nu^2),$ where $\bm:=(\mu^1,\mu^2)$, and $\bn:=(\nu^1,\nu^2).$ %See Appendix \ref{sec:appendix_ot} on notions and classical results from optimal transport.

We state the following well-known lemma.
\begin{lemma}
The objective functional in the minimization problem \eqref{gf:tau} is l.s.c. and bounded from below and $\sP^{M_1}(\Om)\times\sP^{M_2}(\Om)$ is compact, thus the optimizer exists. Moreover, $\cF_m$ ($m\in[1,+\infty]$) and $\cG$ are convex functionals and the functional  $\br\mapsto\bW_2^2(\br,\bm)$ is strictly convex whenever $\bm=(\mu^1,\mu^2)$ has absolutely continuous density coordinates (see for instance \cite{OTAM}). Therefore, if the densities $(\rho^1_0,\rho^2_0)$ are absolutely continuous w.r.t. $\sL^d\mres\Om$, the optimizer $\br^\t_k$ is also unique at each step. 
\end{lemma}

\subsection{Different diffusion coefficients for the two densities}\label{constants}
In many cross-diffusion models (coming mainly from mathematical biology or fluid mechanics, see for instance in \cite{BerGurHilPel, LorLorPer}) considered in the literature, it is important to have different diffusion coefficients for the two densities. In our setting, this could be formulated as follows. Given $\k_1,\k_2$ positive constants, consider a system similar to \eqref{eq:PME_m} or \eqref{eq:PME_infty}, i.e. 
\be\label{eq:k_1-k_2}
\left\{
\ba{l}
\partial_t\rho^1-\diver\left(\k_1\nabla p\rho^1+\nabla\Phi_1\rho^1\right)=0\\ [5pt]
\partial_t\rho^2-\diver\left(\k_2\nabla p\rho^2+\nabla\Phi_2\rho^2\right)=0
\ea
\right.
\ee
on $[0,T]\times\Om,$ where $\ds p:=\frac{m}{m-1}(\rho^1+\rho^2)^{m-1},$
with $m>1$ and where $\Phi_1,\Phi_2:\Om\to\R$ are given potentials. Observe that \eqref{eq:PME_m} corresponds to $\k_1=\k_2=1$. Actually, even for $\k_1\neq\k_2$, this system enters naturally into the framework of gradient flows considered in this paper. Indeed, we can define the minimizing movement scheme as
$$\ds (\rho^{1,\t}_{k+1},\rho^{2,\t}_{k+1})={\rm{argmin}}_{(\rho^1,\rho^2)}\left\{\cF(\rho^1,\rho^2)+\int_\Om\frac{\Phi_1}{\k_1}\rho^1\dd x+\int_\Om\frac{\Phi_2}{\k_2}\rho^2\dd x+\frac{1}{2\t\k_1}W_2^2(\rho^1,\rho^{1,\t}_k)+\frac{1}{2\t\k_2}W_2^2(\rho^2,\rho^{2,\t}_k)\right\}.$$
Actually a part of the analysis that we perform in the forthcoming sections will be valid in this case as well. In particular the results from Section \ref{sec:PDE_m} can be easily adapted to the system \eqref{eq:k_1-k_2}.

\bigskip

\subsection{Properties of the minimizers}\label{subsec:prop_min}

We discuss now some properties of the minimizers in \eqref{gf:tau}.  For this, let us consider the following hypotheses
\be \label{hyp:rho}
\left\{
\ba{ll}
\rho^1_0,\rho^2_0\in L^m(\Om), & {\rm{if\ }} m\in(1,+\infty),\\[5pt]
\|\rho^1_0+\rho^2_0\|_{L^\infty(\Om)}\le 1\ \ {\rm{and}} \ \ \sL^d(\Om)> M_1+M_2, & {\rm{if\ }} m=+\infty;
\ea
\right.
\tag{$\text{H}_\rho^m$}
\ee
Notice that the structural condition $\sL^d(\Om)> M_1+M_2$ in the case of $m=+\infty$ is needed in order to have nontrivial competitors that satisfy the upper bound constraint.
\be\label{hyp:phi}
\Phi_1,\Phi_2\in W^{1,\infty}(\Om).\tag{H$_\Phi$}
\ee

First, let us derive the first order necessary optimality conditions for the minimizers in \eqref{gf:tau}.

\begin{lemma}[Optimality conditions: $m$ finite]\label{lem:opt_cond}
Let $m\in(1,+\infty)$ and let $\Phi_1$ and $\Phi_2$ satisfy \eqref{hyp:phi} and $(\rho_0^1,\rho^2_0)$ satisfy \eqref{hyp:rho}. Let $(\rho^1,\rho^2)$ be the unique minimizer in \eqref{gf:tau} with $k=0$. Then
\begin{enumerate}
\item[(1)] there exist Kantorovich potentials $\vphi^i$, $i=1,2$, in the transport of $\rho^i$ onto $\rho^i_0$ and $C_i\in\R$ ($i=1,2$) such that
%\be\label{opt_cond:ineq}
%\left\{
%\ba{ll}
%\frac{m}{m-1}(\rho^1+\rho^2)^{m-1}+\Phi_i+\frac{\vphi^i}{\t}=C_i, & \rho^i-{\rm{a.e.}},\\[5pt]
%\frac{m}{m-1}(\rho^1+\rho^2)^{m-1}+\Phi_i+\frac{\vphi^i}{\t}\ge C_i, & \ae\ {\rm{on}}\ \{\rho^i=0\}.
%\ea
%\right.
%\ee
%In a compact form, one can write 
\be\label{opt_cond:compact}
\frac{m}{m-1}(\rho^1+\rho^2)^{m-1}=\max\left(C_1-\Phi_1-\vphi^1/\t; C_2-\Phi_2-\vphi^2/\t;0\right),
\ee
In particular $(\rho^1+\rho^2)^{m-1}$ is Lipschitz continuous, $\rho^1+\rho^2 \in C^{0,1/(m-1)}(\Om),$  and these regularities degenerate as $\t\da 0.$

\item[(2)] One can differentiate the above equality a.e. and the optimal transport maps $T^i$ ($i=1,2$) in the transport of $\rho^i$ onto $\rho^i_0$ have the form 
$$
T^i=\id+\t\left(\frac{m}{m-1}\nabla (\rho^1+\rho^2)^{m-1}+\nabla\Phi_i\right)
$$
\end{enumerate}
\end{lemma}
\begin{proof}
The proof of these results are just easy adaptations of the ones from Lemma \ref{cor:opt_cond}, thus we omit it. %For (1), one takes variations of the form $(\rho^1_\e,\rho^2)$ and $(\rho^1,\rho^2_\e),$ where $\rho^i_\e=\rho^i+\e(\tilde\rho^i-\rho^i)$ with $\e\in[0,1]$ and $\tilde\rho^i\in\sP^{M_i,\ac}(\Om)$ and one uses the mentioned corollary to obtain \eqref{opt_cond:ineq}. Then one obtains \eqref{opt_cond:compact} by \eqref{opt_cond:ineq} and a partition of $\Om$ that consists four sets $\{\rho^1>0\}\cap\{\rho^2>0\},\dots,\{\rho^1=0\}\cap\{\rho^2=0\}.$  

%Since $\Om$ is bounded, the potentials $\vphi^i$ are Lipschitz continuous (see Theorem \ref{thm:classical_OT}), by \eqref{hyp:phi} $\Phi_i$ are also Lipschitz continuous, hence \eqref{opt_cond:compact} implies that $(\rho^1+\rho^2)^{m-1}$ is Lipschitz continuous as well, which further implies that $\rho^1+\rho^2$ is H\"older continuous on $\Om$ with exponent $1/(m-1)$. These regularities clearly degenerate as $\t\da 0.$ 

%To obtain (2), one only needs to use the Lipschitz continuity of \eqref{opt_cond:compact}, Rademacher's theorem and Theorem \ref{thm:classical_OT} (6) to obtain the expression of the optimal transport maps.
\end{proof}

\begin{lemma}[Optimality conditions: $m=\infty$]\label{lem:opt_cond_infty}
Let $m=\infty$ and let $\Phi_1$ and $\Phi_2$ satisfy \eqref{hyp:phi} and $(\rho_0^1,\rho^2_0)$ satisfy \eqref{hyp:rho}. Let $(\rho^1,\rho^2)$ be the unique minimizer in \eqref{gf:tau} with $k=0$. Then
\begin{itemize}
\item[(1)] there exist Kantorovich potentials $\vphi^i$ in the transport of $\rho^i$ onto $\rho^i_0$ ($i=1,2$) such that
\be
\int_\Om (\Phi_1+\vphi^1/\t)(\mu^1-\rho^1)\dd x+\int_\Om (\Phi_2+\vphi^2/\t)(\mu^2-\rho^2)\dd x\ge 0,
\ee
for any $(\mu^1,\mu^2)\in\sP^{M_1}(\Om)\times\sP^{M_2}(\Om)$ such that $\mu^1+\mu^2\le 1$  a.e. in $\Om$.
\item[(2)] There exists a Lipschitz continuous \emph{pressure function} $p$  that can be defined via the Kantorovich potentials $\vphi^1,\vphi^2$ from (1) as
\be\label{def:p_grad}
\nabla p=-\nabla\vphi^i/\t-\nabla\Phi_i, \ \rho^i-\ae,\ i=1,2, 
\ee
and $p\ge 0$ and $p(1-(\rho^1+\rho^2))=0$ a.e. in $\Om.$ In particular, the optimal transport map $T^i$ in the transportation of $\rho^i$ onto $\rho^i_0$ ($i=1,2$) has the form 
$$T^i=\id +\t\left(\nabla p+\nabla\Phi_i\right).$$
\end{itemize}
\end{lemma}
\begin{proof}
The proof of the above results are adaptations of the ones from \cite[Lemma 3.1-3.2]{MauRouSan1} and \cite[Lemma 6.11-Proposition 6.12]{LabPhD}, so we omit it. %Indeed, for (1) one can perform exactly as in \cite[Lemma 3.1]{MauRouSan1}. For (2), since $(\rho^1,\rho^2)$ is an optimizer of the functional  
%$$(\mu^1,\mu^2)\mapsto\int_\Om (\Phi_1+\vphi^1/\t)\mu^1\dd x+\int_\Om (\Phi_2+\vphi^2/\t)\mu^2\dd x,$$
%under the constraint $\mu^1+\mu^2\le 1$ a.e. in $\Om$, there exist two constants $C_1$ and $C_2$ such that
%$$
%\ba{ll}
%\Phi_i+\vphi^i/\t=C_i, & \text{in } \{\rho^1+\rho^2<1\}\cap\{\rho^i>0\},\\[5pt]  
%\ea
%$$
%and
%$$
%\ba{ll}
%\Phi_i+\vphi^i/\t<C_i, & \text{in } \{\rho^1+\rho^2=1\}\cap\{\rho^i>0\}, 
%\ea
%$$ 
%In particular, one has 
%$$\Phi_1+\vphi^1/\t-C_1= \Phi_2+\vphi^2/\t-C_2\ \ \text{in } \{\rho^1+\rho^2=1\}\cap\{\rho^1>0\}\cap\{\rho^2>0\}.$$
%Thus one can define
%\be\label{def:p}
%p:=\left\{
%\ba{ll}
%(C_1-\Phi_1-\vphi^1/\t)_+, & \text{in } \{\rho^1>0\},\\[5pt]
%(C_2-\Phi_2-\vphi^2/\t)_+, & \text{in } \{\rho^2>0\},\\[5pt]
%0, & \text{in } \Om\setminus\left(\{\rho^1>0\}\cup\{\rho^2>0\}\right).
%\ea
%\right.
%\ee  
%Notice that with this definition $p\ge 0$ and $p(1-(\rho^1+\rho^2))=0$ a.e. in $\Om.$ Moreover since $\vphi^i$  and $\Phi_i$ ($i=1,2$) are Lipschitz continuous, so is $p$ in the interior of $\{\rho^i>0\}$, and the expression \eqref{def:p_grad} holds true as well. To write the optimal transport map $T^i$ in the transport of $\rho^i$ onto $\rho^i_0$ one notices that $T^i=\id-\nabla\vphi^i$, then one uses the expression \eqref{def:p_grad}. 

\end{proof}

\subsection{Equilibrium solutions when $m<+\infty$}\label{sec:eq_sol}
Let us study the equilibrium solutions $(\ov\rho^1,\ov\rho^2)$ of the scheme \eqref{gf:tau}, meaning that $(\ov\rho^1,\ov\rho^2)$ is a minimizer of the \emph{free energy} $\cF+\cG$. This exists by the l.s.c. and boundedness from below of the functional and the compactness of $\sP^{M_1}(\Om)\times\sP^{M_2}(\Om)$. Then writing down the first order optimality conditions as in Lemma \ref{lem:opt_cond}, one obtains that
\be
\left\{
\ba{ll}
\frac{m}{m-1}(\ov\rho^1+\ov\rho^2)^{m-1}=C_i-\Phi_i &  \hbox{ in }\ \{\ov\rho^i >0\} ,\\[5pt]
\frac{m}{m-1}(\ov\rho^1+\ov\rho^2)^{m-1}\ge C_i-\Phi_i, & \hbox{ in }\ \{\ov\rho^i=0\},
\ea
\right.
\ee 
or in short
$$
\frac{m}{m-1}(\ov\rho^1+\ov\rho^2)^{m-1} = \max[C_1-\Phi_1; C_2-\Phi_2;0],
$$
for $i=1,2$ and for some constants $C_1,C_2\in\R$. For simplicity in this informal discussion one may suppose that both $\Phi_1$ and $\Phi_2$ are strictly convex with a unique minimizer in $\Om$. Otherwise the constants $C_i$ may vary on each connected component of $\{\ov\rho^i>0\}$. Observe that the above conditions imply in particular that whenever the potentials $\Phi_1$ and $\Phi_2$ are different and their difference is not only a constant, then the phases $\ov\rho^1$ and $\ov\rho^2$ are separated, i.e. $\sL^d\left(\{\ov\rho^1>0\}\cap\{\ov\rho^2>0\}\right)=0.$ Moreover, in general the interface $\{\ov\rho^1>0\} \cap \{\ov\rho^2>0\}$ is present and on the interface the densities $\rho^i$ ($i=1,2$) are positive. For instance this is the case when we take potentials $\Phi_1(x) = |x|^2$ and $\Phi_2 = 2|x|^2$ and $C_1, C_2$ are such that $0<C_1<C_2$ and both densities are present.

%This interesting feature of the problem is due to the structure of the internal energy $\cF$ and in general it would not be the case, whenever some other terms for the two densities separately would be also present in the internal energy.   

\medskip

In fact, with the above choice of potentials $\Phi_i$, $i=1,2$, suppose that we start our minimizing movements with initial configuration of densities $\rho^1_0 = \chi_{\{|x|\le 1\}}$ and $\rho^2_0 = \chi_{\{1<|x|<2\}}$. In the equilibrium limit we have $\{\ov\rho^2>0\} = \{|x| \leq r_1\}$ and $\{\ov\rho^2>0\} = \{r_1\leq |x|\leq r_2\}$ for some $0<r_1<r_2$. Thus, if solutions $(\rho^1,\rho^2)$ of the system \eqref{eq:PME_m} exist with these initial data and potentials, heuristically it is inevitable that the supports of $\rho^1_t$ and $\rho^2_t$ get mixed for some finite time $t>0$, while $\rho^1$ ``filtrates'' through $\rho^2$ to change the ordering of their supports from the initial configuration. Such situation indicates low regularity for each density, and illustrates the difficulty in obtaining a strong notion of limit solutions for \eqref{eq:PME_m} in the continuum limit. Indeed in general we are only able to obtain a very weak notion of solutions in the continuum limit, as we will see in Theorem~\ref{thm:precise}. Deriving this weak notion of solutions in general settings is our first main result in the paper.  To the best of the authors' knowledge, there does not seem to be a PDE approach to yield well-posedness on the continuum PDE \eqref{eq:PME_m}, especially when $\nabla\Phi_1\neq\nabla\Phi_2$.

\medskip

On the other hand, if the initial configuration of above example is in line with the potentials, i.e. if we switch the roles of $\rho^1_0$ and $\rho^2_0$, we expect the solutions to be well-behaved and to stay separated throughout the evolution, with stable interface in between them. It turns out that we can indeed show such separation in one spacial dimension. In this case stronger results are available, and one can derive stronger notion of solutions as well as the properties of the solutions and their interfaces in the incompressible limit $m\to\infty$, which in some cases leads to a type of two-phase Hele-Shaw flow with drifts (see Section \ref{sec:limit_with_m}).

% I'm not sure where is the good place for it. I don't know what you mean by mixing situation. Could you specify please?}

%\textcolor{blue}{not sure if this is the best place but anyway}

%\textcolor{blue}{you don't want to put the mixing situation?}
\subsection{Regularity of the minimizers in the \eqref{gf:tau} scheme}

\begin{theorem}\label{thm:reg_1step}
Let $m\in(1,+\infty).$ Let $(\rho^1_0,\rho^2_0)\in\sP^{M_1}(\Om)\times\sP^{M_2}(\Om)$ satisfying \eqref{hyp:rho} and let \eqref{hyp:phi} be fulfilled. Let $(\rho^1,\rho^2)$ be the minimizer in \eqref{gf:tau} constructed with the help of $(\rho^1_0,\rho^2_0)$. Then 
\be\label{estim:Lm_1step}
\rho^1, \rho^2\in L^m(\Om)
\ee
and
\be\label{estim:H1}
(\rho^1+\rho^2)^{m-1/2}\in H^1(\Om).
\ee
If $m=+\infty$, $\rho^1+\rho^2\le 1$ a.e. in $\Om$.
\end{theorem}

\begin{proof}
First, setting $\br=(\rho^1,\rho^2)$ and $\br_0=(\rho^1_0,\rho^2_0),$ by the optimality of $\br$ in \eqref{gf:tau} w.r.t. $\br_0$, one obtains
\begin{align*}
\cF_m(\br)=\frac{1}{m-1}\int_\Om (\rho^1+\rho^2)^m\dd x & \le \frac{1}{2\t}\bW_2^2(\br,\br_0)+\cF_m(\br_0)+\cG(\br_0)-\cG(\br)\\
&\le \frac{1}{2\t}\bW_2^2(\br,\br_0)+\cF_m(\br_0)+2(M_1\|\Phi_1\|_{L^\infty}+M_2\|\Phi_2\|_{L^\infty}),
\end{align*}
which by the assumptions \eqref{hyp:rho} and \eqref{hyp:phi} implies \eqref{estim:Lm_1step} for $m$ finite. If $m=\infty,$ then clearly $\rho^1+\rho^2\le 1$ a.e. in $\Om$.

Second, writing down the first order optimality conditions (see Lemma \ref{lem:opt_cond}) for the above problem, one obtains 

$$\frac{m}{m-1}(\rho^{1}+\rho^{2})^{m-1}+\Phi_i+\frac{\vphi^i}{\t}=C_i,\;\; \iin\ \{\rho^{i}>0\},\;\; i=1,2,$$
where $\vphi^i$ is a Kantorovich potential in the optimal transport of $\rho^{i}$ onto $\rho^{i}_{0}.$ This potential is linked to the optimal transport map between these densities as $T^i(x)=x-\nabla\vphi^i(x).$ So, by Lemma \ref{lem:opt_cond}(2) one can write
\begin{equation}\label{opt:cond}
-\frac{m}{m-1}\nabla(\rho^{1}+\rho^{2})^{m-1}-\nabla\Phi_i=\frac{\nabla\vphi^i}{\t},\; \rho^i-\ae,\; i=1,2.
\end{equation}
Since the r.h.s. of \eqref{opt:cond} is in $L^2_{\rho^{i}}(\Om)$ with $\int_\Om\frac{1}{\t^2}|\nabla\vphi^i|^2\rho^{i}\dd x=\frac{1}{\t^2}W_2^2(\rho^{i},\rho^{i}_{0})$ and $\nabla\Phi_i\in L^2_{\rho^i}(\Om;\R^d)$ we have the estimation
%$$\int_\Om\left| \frac{m}{m-1}\nabla(\rho^{1}+\rho^{2})^{m-1}+\nabla\Phi_i \right|^2\rho^{i}\dd x=\int_\Om\frac{1}{\t^2}|\nabla\vphi^i|^2\rho^{i}\dd x=\frac{1}{\t^2}W_2^2(\rho^{i},\rho^{i}_{0})$$
%Developing the l.h.s. and using Young's inequality, for $\e>0$ one obtains
%\begin{align*}
%\frac{m^2}{(m-1)^2}\int_\Om\left|\nabla(\rho^{1}+\rho^{2})^{m-1}\right|^2\rho^{i}\dd x&\le\frac{1}{\t^2}W_2^2(\rho^{i},\rho^{i}_{0})-\int_\Om\left|\nabla\Phi_i \right|^2\rho^{i}\dd x \\
%&+\frac{m}{m-1}\e\int_\Om\left|\nabla\Phi_i \right|^2\rho^{i}\dd x\\
%&+\frac{m}{m-1}\frac{1}{\e}\int_\Om\left|\nabla(\rho^{1}+\rho^{2})^{m-1}\right|^2\rho^{i}\dd x,
%\end{align*}
%which with the choice of $\ds\e:=2(m-1)/m$ implies that
$$\int_\Om\left|\nabla(\rho^{1}+\rho^{2})^{m-1}\right|^2\rho^{i}\dd x\le \frac{2(m-1)^2}{m^2}\left(\frac{1}{\t^2}W_2^2(\rho^{i},\rho^{i}_{0})+M_i\|\nabla\Phi_i\|_{L^\infty}^2\right).$$
Adding up the two inequalities for $i=1,2,$ one obtains after rearranging 
\begin{equation}\label{estim:gradient_1step}
\int_\Om\left|\nabla(\rho^{1}+\rho^{2})^{m-1/2}\right|^2\dd x\le\frac{2(m-1/2)^2}{m^2}\left(\frac{1}{\t^2}\bW_2^2(\br,\br_{0})+M_1\|\nabla\Phi_1\|_{L^\infty}^2+M_2\|\nabla\Phi_2\|_{L^\infty}^2\right).
\end{equation}
By the estimation \eqref{estim:Lm_1step} $\rho^{1}+\rho^{2}$ is  bounded in $L^m(\Om)$, so by the fact that $\Om$ is compact, $\rho^1+\rho^2$ is summable in $L^q(\Om)$ for any $1\le q\le m.$ This means in particular that the average can be bounded as 
$$ \fint_\Om(\rho^{1}+\rho^{2})^{m-1/2}\dd x\le\|\rho^1+\rho^2\|_{L^m}^{m-1/2} \sL^d(\Om)^{1/(2m)-1}$$ 
hence Poincar\'e's inequality yields that $\|(\rho^{1}+\rho^{2})^{m-1/2}\|_{L^2(\Om)}$ is  bounded, more precisely 
\begin{align*}
\|(\rho^{1}+\rho^{2})^{m-1/2}\|_{L^2(\Om)}&\le C_\Om\|\nabla(\rho^{1}+\rho^{2})^{m-1/2}\|_{L^2(\Om)}+\sL^d(\Om)^{\frac12}\fint_\Om(\rho^{1}+\rho^{2})^{m-1/2}\dd x\\
&=C_\Om\|\nabla(\rho^{1}+\rho^{2})^{m-1/2}\|_{L^2(\Om)}+\|\rho^1+\rho^2\|_{L^m}^{m-1/2} \sL^d(\Om)^{1/(2m)-1/2}
\end{align*}
where $C_\Om>0$ is the Poincar\'e constant associated to $\Om.$ Thus, \eqref{estim:H1} follows.
\end{proof}

\section{The continuum limit solutions in general dimension}\label{sec:PDE_m}

In this section we study the convergence of the time-discrete solutions in the continuum limit. The limit solutions can be interpreted as a \emph{very weak} solution for both systems \eqref{eq:PME_m} and \eqref{eq:PME_infty} in the following sense:

\begin{definition}[Notion of weak solution]\label{def:weak_and_very_weak}
By a \emph{weak solution} of system \eqref{eq:PME_m} we mean a pair $(\rho^1,\rho^2)$ such that 
$\rho^i\in AC^2([0,T];\sP^{M_i}(\Om))\cap L^{2m-1}([0,T]\times\Om)$, and setting $p:=\frac{m}{m-1}(\rho^1+\rho^2)^{m-1}$, $\nabla p\rho^i\in L^r([0,T]\times\Om; \R^d),$ for some $1<r<2$ ($i=1,2$). Moreover $\rho^i|_{t=0}=\rho^i_0$ ($i=1,2$) and the equation 
\be\label{PME_m_weak}
-\int_s^t\int_\Om\rho^i\partial_t\phi\dd x\dd \t+\int_s^t\int_\Om \v^i\cdot\nabla\phi\rho^i\dd x\dd \t=\int_\Om\rho^i_s(x)\phi(s,x)\dd x-\int_\Om\rho^i_t(x)\phi(t,x)\dd x,
\tag{Weak}
\ee
holds true for all $\phi \in C^1([0,T]\times\Om)$ and for all $0\le s< t\le T,$ where 
$$\v^i:=\nabla p+\nabla\Phi_i.$$ 

%We say that $(\rho^1,\rho^2)$ is a \emph{very weak solution} of \eqref{eq:PME_m} if the densities belong to the same spaces as before, the initial conditions are satisfied and and there exist vector fields $\v^i\in L^2([0,T];L^2_{\rho^i}(\Om;\R^d))$, $i=1,2$ such that \eqref{PME_m_weak} holds true for all $\phi \in C^1([0,T]\times\Om)$ and for all $0\le s< t\le T$ with 
%$$\v^1\rho^1+\v^2\rho^2=\nabla(\rho^1+\rho^2)^m+\nabla\Phi_1\rho_1+\nabla\Phi_2\rho^2$$
%in the sense of distributions.

\medskip

Similarly, by a \emph{weak solution} of \eqref{eq:PME_infty} we mean a triple $(\rho^{1,\infty},\rho^{2,\infty},p^\infty)$ such that 
$\rho^{i,\infty}\in AC^2([0,T];\\ \sP^{M_i}(\Om))\cap L^{\infty}([0,T]\times\Om)$, $i=1,2$ with $\|\rho^{1,\infty}+\rho^{2,\infty}\|_{L^\infty}\le 1$, $p^\infty\in L^2([0,T];H^1(\Om)),$ $p^\infty\ge 0$ and $p^\infty(1-\rho^{1,\infty}-\rho^{2,\infty})=0$ a.e. in $[0,T]\times\Om$. Moreover $\rho^i|_{t=0}=\rho^i_0$ ($i=1,2$) and the equation \eqref{PME_m_weak} 
holds true with $p$ replaced by $p^\infty$ for all $\phi \in C^1([0,T]\times\Om)$ and for all $0\le s< t\le T.$

%We say that $(\rho^{1,\infty},\rho^{2,\infty},p^\infty)$ is a \emph{very weak solution} of \eqref{eq:PME_infty}, if the densities and the pressure belong to the same spaces as before, the initial conditions are satisfied and there exist vector fields $\v^i\in L^2([0,T];L^2_{\rho^i}(\Om;\R^d))$, $i=1,2$ such that \eqref{PME_m_weak} holds true for all $\phi \in C^1([0,T]\times\Om)$ and for all $0\le s< t\le T$ with 
%$$\v^1\rho^{1,\infty}+\v^2\rho^{2,\infty}=\nabla p^\infty+\nabla\Phi_1\rho^{1,\infty}+\nabla\Phi_2\rho^{2,\infty}$$
%in the sense of distributions.

We underline that the above weak formulations encode in particular no-flux boundary conditions on $[0,T]\times\partial\Om$.
\end{definition}

\begin{remark}
\begin{itemize}
\item[(a)] Notice that by density arguments, in the definition of the weak solution of \eqref{eq:PME_m} one can consider $\phi\in W^{1,1}([0,T];L^q(\Om))\cap L^q([0,T];W^{1,q}(\Om))$ where $q=\max\{r',(2m-1)'\}$ and in the case of \eqref{eq:PME_infty} one can consider test functions in $W^{1,1}([0,T];L^1(\Om))\cap L^2([0,T];H^1(\Om)).$ 
\item[(b)] Also, by the fact that we impose that the densities are absolutely continuous curves in the Wasserstein space\footnote{See the Appendix on optimal transportation}, imposing the initial conditions is meaningful.
\item[(c)] %Observe that any weak solution is always a very weak solution, while the converse statement is not true in general. Heuristically, the very weak solutions become weak solutions only when the two densities are guaranteed to be segregated. Also, 
The uniqueness question of weak solutions seems to be very delicate and challenging. %However, we do not expect uniqueness of very weak solutions, morally because of the loss of information (one equation less than the number of unknowns).
\end{itemize}
\end{remark}

\subsection{Interpolations between the densities} Let $m\in(1,+\infty]$ and let us consider $(\rho^1_0,\rho^2_0)$ and $\Phi_1$ and $\Phi_2$ satisfying the hypotheses \eqref{hyp:rho} and \eqref{hyp:phi} respectively. We consider also $T>0$ a fixed time horizon, a time step $\t>0$ and $N\in\N$ such that $N\t=T$ and the densities $(\rho^{1,\t}_k,\rho^{2,\t}_k)_{k=0}^N$ obtained via the \eqref{gf:tau} scheme starting from $(\rho^1_0,\rho^2_0).$ We denote the optimal transport maps and the corresponding Kantorovich potentials between two consecutive densities $\rho^{i,\t}_{k+1}$ and $\rho^{i,\t}_k$ by $T^i_k$ and $\vphi^i_k$ respectively ($k\in\{0,\dots,N-1\}$, $i=1,2$).
If $m=\infty$, we consider also the pressure variables $p^\t_k$ ($k\in\{1,\dots,N\}$) constructed as in Lemma \ref{lem:opt_cond_infty}.

Since $\frac{\nabla\vphi^i_k}{\t}=\frac{\id-T_k^i}{\t}$ can be seen as a discrete velocity (displacement divided by time), it is reasonable to define the discrete velocity of the particles of the $i^{th}$ fluid located at $x\in\Om$ (for a.e. $x\in\Om$) as
\be\label{def:disc_vel}
\v_k^{i,\t}(x):=\left\{
\ba{ll}
-\frac{m}{m-1}\nabla(\rho_{k+1}^{1,\t}(x)+\rho^{2,\t}_{k+1}(x))^{m-1}-\nabla\Phi_i(x), & \text{if } m\in(1,+\infty),\\[5pt]
-\nabla p^\t_{k+1}(x)-\nabla\Phi_i(x), & \text{if\ } m=\infty.
\ea
\right.
\ee

As  technical tools, we shall consider continuous and piecewise constant interpolations between the discrete densities. We will also work with the associated velocities and momenta. These constructions and the estimates on them are standard for experts and are very similar to the ones from \cite[Chapter 8.3]{OTAM} and from \cite{MauRouSan1}. We refer to \cite{San} as well, as an overview of these techniques. 

{\it Continuous interpolations}. Using McCann's interpolation -- as it is done for instance in \cite[Chapter 8.3]{OTAM} -- we can consider families of continuous interpolations $[0,T]\ni t\mapsto (\rho^{1,\t}_t,\rho^{2,\t}_t)\in \sP^{M_1}(\Om)\times\sP^{M_2}(\Om)$ between the discrete in time densities parametrized with $\t>0.$ We denote the corresponding time dependent families of velocities and momenta by $\v^{i,\t},\E^{i,\t}$.
%\be\label{interp:cont}
%\rho^{i,\t}_t:=\left( \frac{t-k\t}{\t}\id + \frac{(k+1)\t-t}{\t}T_k^i \right)_\#\rho_{k+1}^{i,\t}=\left(\id+(t-(k+1)\t)\v_k^{i,\t}\right)_\#\rho_{k+1}^{i,\t},\ {\rm{if}\ }t\in[k\t,(k+1)\t),
%\ee
%where $\v_k^{i,\t}$ is the velocity field defined in \eqref{def:disc_vel}. By construction $[0,T]\ni t\mapsto\rho_t^{i,\t}$, $i=1,2$ is a geodesics in $\sP^{M_i}(\Om)$ w.r.t. $W_2$ on each time interval of the form $[k\t,(k+1)\t]$, hence setting $\br^\t_t:=(\rho^{1,\t}_t,\rho_t^{2,\t}),$ for all $t\in[0,T]$ is a geodesic in $\sP^{M_1}(\Om)\times\sP^{M_2}(\Om)$ w.r.t. $\bW_2$ on each time interval $[k\t,(k+1)\t]$.
%
%The corresponding velocity fields associated to the geodesic interpolations $\rho^{i,\t}_t$ are given by
%$$\v^{i,\t}_t:=\v^{i,\t}_k\circ\left(\id-((k+1)\t-t)\v_k^{i,\t}\right)^{-1}, \;\; {\rm{if}\ }t\in[k\t,(k+1)\t),$$
%and the family of time dependent momenta is given by $\E^{i,\t}_t=\rho_t^{i,\t}\v_t^{i,\t},\ t\in[0,T],\ i=1,2.$ In particular one has $\v^{i,\t}_t\in L^2_{\rho_t^{i,\t}}(\Om)$ for all $t\in[0,T]$ and $\|\v^{i,\t}_t\|_{L^2_{\rho_t^{i,\t}}}=|(\rho_t^{i,\t})'|_{W_2}$, where the last quantity denotes the metric derivative of the curve $t\mapsto\rho_t^{i,\t}$ in $(\sP^{M_i}(\Om),W_2)$ (see Appendix \ref{sec:appendix_ot} for more details and Lemma \ref{lem:estimates} for a more precise statement).

It is worth to notice that the above construction implies in particular that $(\rho^{i,\t},\E^{i,\t})$ ($i=1,2$) solves the continuity equation 
\be\label{eq:cont_tau}
\partial_t\rho^{i,\t}+\diver \E^{i,\t}=0.
\ee
on $[0,T]\times\Om$ in the weak sense, i.e. 
\be\label{eq:cont_weak}
\int_s^t\int_\Om\rho^i\partial_t\phi\dd x\dd \t+\int_s^t\int_\Om \E^i\cdot\nabla\phi\dd x\dd \t=-\int_\Om\rho^i_s(x)\phi(s,x)\dd x+\int_\Om\rho^i_t(x)\phi(t,x)\dd x
\ee
 for all $\phi\in C^1([0,T]\times\Om)$ and $0\le s<t\le T$.

{\it Piecewise constant interpolations}. We consider a second family of interpolations, simply taking
\be\label{interp:const}
\tilde\rho^{i,\t}_t:=\rho^{i,\t}_{k+1},\ \tilde \v^{i,\t}_t:=\v^{i,\t}_k,\ {\rm{and}}\  \tE_t^{i,\t}:= \tilde\rho_t^{i,\t}\tilde \v_t^{i,\t}\;\; {\rm{for}\ }t\in[k\t,(k+1)\t).
\ee
We consider the piecewise constant interpolation for the pressure variable (see Lemma \ref{lem:opt_cond_infty}) as well, i.e. 
\be\label{def:press_interp}
\tilde p^\t_t:=\left\{
\ba{ll}
(C_1-\Phi_1-\vphi^1_k/\t)_+, & \text{in } \{\rho^{1,\t}_{k+1}>0\},\\[5pt]
(C_2-\Phi_2-\vphi^2_k/\t)_+, & \text{in } \{\rho^{2,\t}_{k+1}>0\},\\[5pt]
0, & \text{in } \Om\setminus\left(\{\rho^{1,\t}_{k+1}>0\}\cup\{\rho^{2,\t}_{k+1}>0\}\right),
\ea
\right.
\ {\rm{for\ }} t\in[k\t,(k+1)\t),k\in\{0,\dots,N-1\}.
\ee
In addition we set $\tilde\br_t^\t:=(\tilde\rho_t^{1,\t},\tilde\rho_t^{2,\t})$ for all $t\in[0,T].$ We remark that by construction one has $\tilde\rho^{i,\t}_t=\rho^{i,\t}_t$ for $t=k\t,$ $k\in\{1,\dots N\}.$ 

\subsection{A priori estimates for the interpolations}
We discuss now some estimates on the interpolations that will be useful to pass to the limit as $\t\da 0.$ In general, all the constants in the estimates depend on the data $\rho^1_0,\rho^2_0,\Phi_1,\Phi_2, T$ and $m$, however it will be especially important to keep track the precise dependence of them on $m$ (in particular we use these estimates also in the limiting procedure when $m\to+\infty$). To highlight this dependence, we denote the constants as $C(m)$.

\begin{lemma}
For any $m\in(1,+\infty]$, $\t>0$ and any $k\in\{0,\dots,N-1\}$ one has 
\be
\frac{1}{2\t}\sum_{k=0}^{N-1}\bW_2^2(\br_{k+1}^\t,\br_{k}^\t)\le \cF_m(\br_0^\t)+\cG(\br_0^\t)-\cF_m(\br_N^\t)-\cG(\br_N^\t),
\ee
and
\be\label{estim:Lm}
\|\rho^{1,\t}_{k+1}+\rho^{2,\t}_{k+1}\|_{L^m(\Om)}\le C_1(m),
\ee
where 
\be\label{const:Lmbound}
C_1(m):=\left\{
\ba{ll}
\left((2m-2)\left(M_1\|\Phi_1\|_{L^\infty(\Om)}+M_2\|\Phi_2\|_{L^\infty(\Om)}\right) +\|\rho^{1}_{0}+\rho^{2}_{0}\|_{L^m(\Om)}^m\right)^{1/m}, & {\rm{if\ }} m<\infty,\\
1, & {\rm{if\ }} m=\infty.
\ea
\right.
\ee
\end{lemma}
\begin{proof}
The proofs of both inequalities are immediate by the optimality of $\br_{k+1}^\t$ w.r.t. $\br_k^\t$ in \eqref{gf:tau}. So we omit them. %i.e. for all $k\in\{1,\dots,N-1\}$
%$$\cF(\br_{k+1}^\t)+\cG(\br_{k+1}^\t)+\frac{1}{2\t}\bW_2^2(\br_{k+1}^\t,\br_k^\t)\le\cF(\br_{k}^\t)+\cG(\br_{k}^\t).$$
%By adding up all these inequalities, we obtain the first inequality. If we just use the inequality $\cF(\br_{k+1}^\t)+\cG(\br_{k+1}^\t)\le\cF(\br_{k}^\t)+\cG(\br_{k}^\t),$ one easily finds that $\cF(\br_{k+1}^\t)+\cG(\br_{k+1}^\t)\le\cF(\br_{0})+\cG(\br_{0})$ which translates into the second inequality. Indeed, this inequality implies
%$$\frac{1}{m-1}\int_{\Om}(\rho^{1,\t}_{k+1}+\rho^{2,\t}_{k+1})^m\dd x\le \frac{1}{m-1}\int_{\Om}(\rho^{1,\t}_{0}+\rho^{2,\t}_{0})^m\dd x+2\|\Phi_1\|_{L^\infty}M_1+2\|\Phi_2\|_{L^\infty}M_2,$$
%which implies \eqref{estim:Lm}.
\end{proof}

\begin{corollary}\label{cor:bdd}
Hypotheses \eqref{hyp:phi} and \eqref{hyp:rho} imply that
$$\frac{1}{\t}\sum_{k=0}^{N-1}\left(W_2^2(\rho_k^{1,\t},\rho_{k-1}^{1,\t})+W_2^2(\rho_k^{2,\t},\rho_{k-1}^{2,\t})\right)\le C_2(m),$$
where 
\be\label{const:metric_der}
C_2(m):=\left\{
\ba{ll}
\frac{2}{m-1}\|\rho_0^1+\rho_0^2\|_{L^m}^m+4M_1\|\Phi_1\|_{L^\infty}+4M_2\|\Phi_2\|_{L^\infty}, & {\rm{if\ }} m\in(1,+\infty),\\[5pt]
4M_1\|\Phi_1\|_{L^\infty}+4M_2\|\Phi_2\|_{L^\infty}, & {\rm{if\ }} m=+\infty.
\ea
\right.
\ee
is independent of $\t.$
\end{corollary}

\begin{lemma}[Bounds for $\rho^{i,\t}$, $\v^{i,\t}$ and $\E^{i,\t}$]\label{lem:estimates}
Assume that we constructed the discrete densities $\rho_k^{i,\t}$ for $\t>0$, $k\in\{0,\dots,N\}$ and $i=1,2$. Let $\rho^{i,\t}$ be the continuous interpolations and let $\v^{i,\t}$ and $\E^{i,\t}$ be the associated velocity field and momentum variables respectively. Then
\begin{itemize}
\item[(1)] $\rho^{i,\t}$ is  bounded in $AC^2([0,T];(\sP^{M_i}(\Om),W_2))$ uniformly in $\t>0$;
\item[(2)] $\v^{i,\t}$ is bounded in $L^2([0,T]; L^2_{\rho^{i,\t}}(\Om;\R^d))$ uniformly in $\t>0$;
\item[(3)] $\E^{i,\t}$ and $\tE^{i,\t}$ are bounded in $\sM^d([0,T]\times\Om)$ uniformly in $\t>0$.
\end{itemize} 
\end{lemma}

\begin{proof}
For $\t>0,$ by construction $\rho^{i,\t}$ is a constant speed geodesic interpolation with the corresponding velocity field $\v^{i,\t}.$ This implies that 
$$\int_0^T\|\v^{i,\t}_t\|^2_{L^2_{\rho^{i,\t}_t}}\dd t=\int_0^T|(\rho^{i,\t})'|_{W_2}^2(t)\dd t=\sum_{k=1}^{N-1}\frac{1}{\t}W_2^2(\rho^{i,\t}_{k-1},\rho^{i,\t}_k)\le C_2(m).$$
Now, by Corollary \ref{cor:bdd} we obtain that (1)-(2) hold true. 

To estimate the total variation of $\E^{i,\t}$ we write
\begin{align*}
|\E^{i,\t}|([0,T]\times\Om)&=\int_0^T\int_\Om|\v_t^{i,\t}|\rho_t^{i,\t}\,\dd x\dd t\le\int_0^T\left(\int_\Om |\v_t^{i,\t}|^2\rho_t^{i,\t}\,\dd x\right)^\frac12\left(\int_\Om\rho_t^{i,\t}\,\dd x\right)^\frac12\,\dd t\\
&\le \sqrt{M_i}\sqrt{T}\left(\int_0^T\int_\Om|\v_t^{i,\t}|^2\rho_t^{i,\t}\,\dd x\dd t\right)^\frac12\le \sqrt{M_iTC_2(m)}.
\end{align*}  
In the last inequality we used the previously obtained bound on $\v^{i,\t}.$ The bound on $\tE^{i,\t}$ rely on the same argument. 
\end{proof}

\begin{lemma}[Bounds on $\tilde p^\t$]\label{lem:press_bound}
Let us consider the piecewise constant interpolation $[0,T]\ni t\mapsto \tilde p^\t_t$ of the pressure variables defined in \eqref{def:press_interp}. Then $\tilde p^\t$ is bounded $L^2([0,T];H^1(\Om))$ independently of $\t>0.$
\end{lemma}

\begin{proof}
The proof is  similar to the ones in \cite[Proposition 6.13]{LabPhD} and \cite[Lemma 3.6]{MesSan}. We sketch it below.
Let us use the fact that $\nabla p^\t_{k}=-\nabla\vphi^i_k/\t-\nabla\Phi_i,$  $\rho^{i,\t}_{k+1}-\ae$, for all $k\in\{0,\dots,N-1\}$ where $\vphi^i_k$ is an optimal Kantorovich potential in the transport of $\rho^{i,\t}_{k+1}$ onto $\rho^{i,\t}_k.$ First, let us compute
\begin{align*}
\int_{\Om}|\nabla p^\t_{k}|^2\rho^i_{k+1}\dd x \le\frac{2}{\t^2}\int_{\Om}|\nabla\vphi^i_k|^2\rho^i_{k+1}\dd x+2\int_{\Om}|\nabla\Phi_i|^2\rho^i\dd x=\frac{2}{\t^2}W_2^2(\rho^{i,\t}_{k+1},\rho^{i,\t}_k)+2\|\nabla\Phi_i\|^2_{L^\infty(\Om)},
\end{align*} 
then adding up the two inequalities for $i=1,2$ (using the fact that $p^{\t}_k$ is supported on $\{\rho^1_{k+1}+\rho^2_{k+1}=1\}$), one finds
$$\int_{\Om}|\nabla p^\t_{k}|^2\dd x=\int_{\Om}|\nabla p^\t_{k}|^2(\rho^1_{k+1}+\rho^2_{k+1})\dd x\le\sum_{i=1}^2\left(\frac{2}{\t^2}W_2^2(\rho^{i,\t}_{k+1},\rho^{i,\t}_k)+2\|\nabla\Phi_i\|^2_{L^\infty(\Om)}\right).$$
Integrating in time $\|\nabla \tilde p^\t_t\|_{L^2(\Om)}$ using Corollary \ref{cor:bdd} one has that $\nabla \tilde p^\t\in L^2([0,T]\times\Om).$ Using the fact that $\sL^d(\{\tilde p^\t_t=0\})\ge\sL^d(\{\tilde\rho^{1,\t}_t+\tilde\rho^{1,\t}_t<1\})\ge \sL^d(\Om)-M_1-M_2>0$ for a.e. $t\in[0,T]$, one concludes by a suitable version of Poincar\'e's inequality that $\tilde p^\t$ uniformly bounded in $L^2([0,T];H^1(\Om))$ as desired.
\end{proof}

We show now gradient estimates, derived from the optimality conditions \eqref{opt:cond}. 
\begin{theorem}\label{thm:L2H1}
Let $m\in(1,+\infty)$. Then for the piecewise constant interpolation $\tilde\rho^{i,\t}$ ($i=1,2$) introduced in \eqref{interp:const} one has
\be\label{estim:L2grad}
\|\nabla(\tilde\rho^{1,\t}+\tilde\rho^{2,\t})^{m-1/2}\|_{L^2([0,T]\times\Om)}\le C_3(m),\ \ 
\|(\tilde\rho^{1,\t}+\tilde\rho^{2,\t})^{m-1/2}\|_{L^2([0,T]\times\Om)}\le C_4(m),
\ee
and 
\be\label{estim:LqLm}
\|\tilde\rho^{1,\t}+\tilde\rho^{2,\t}\|_{L^q([0,T]; L^m(\Om))}\le C_5(q,m),\ \forall q\ge 1,
\ee
where $C_3(m),C_4(m),C_5(q,m)>0$ are constants independent of $\t>0.$ The first two bounds imply in particular that
\be\label{estim:L2H1}
\|(\tilde\rho^{1,\t}+\tilde\rho^{2,\t})^{m-1/2}\|_{L^2([0,T];H^1(\Om))}\le \left(C_3(m)^2+C_4(m)^2\right)^{1/2}.
\ee
\end{theorem}
\begin{proof}

We use the inequality \eqref{estim:gradient_1step}, writing for $(\rho_{k+1}^{1,\t},\rho^{2,\t}_{k+1})$, i.e.
\begin{align*}
\int_\Om\left|\nabla(\rho_{k+1}^{1,\t}+\rho^{2,\t}_{k+1})^{m-1/2}\right|^2\dd x&\le \frac{2(m-1/2)^2}{m^2}\left(\sum_{i=1}^2 \frac{1}{\t^2}W_2^2(\rho_{k+1}^{i,\t},\rho_{k}^{i,\t})+\sum_{i=1}^2M_i\|\nabla\Phi_i\|^2_{L^\infty}\right).
\end{align*}
Since the curves $\tilde \rho^{i,\t}$ ($i=1,2$) are piecewise constant interpolations, i.e.  $\tilde \rho^{i,\t}_t=\rho^{i,\t}_{k+1}$ for $t\in (k\t,(k+1)\t],$ one has
\begin{align*}
\int_0^T\int_\Om\left|\nabla(\tilde\rho_{t}^{1,\t}+\tilde \rho^{2,\t}_{t})^{m-1/2}\right|^2\dd x\dd t&=\t\sum_{k=0}^{N-1}\int_\Om\left|\nabla(\rho_{k+1}^{1,\t}+\rho^{2,\t}_{k+1})^{m-1/2}\right|^2\dd x\\
&\le \frac{2(m-1/2)^2}{m^2}\sum_{i=1}^2\sum_{k=0}^{N-1}\frac{1}{\t}W_2^2(\rho_{k+1}^{i,\t},\rho_{k}^{i,\t})\\
&+\frac{2(m-1/2)^2}{m^2}\t\sum_{i=1}^2\sum_{k=0}^{N-1}M_i\|\nabla\Phi_i\|^2_{L^\infty}\\
&\le\frac{2(m-1/2)^2}{m^2}\left(C_2(m)+T\sum_{i=1}^2M_i\|\nabla\Phi_i\|^2_{L^\infty}\right)=:C_3(m)^2,
\end{align*}
which implies \eqref{estim:L2grad}, with
\be\label{const:L2H1}
C_3(m):=\frac{\sqrt{2}(m-1/2)}{m}\left(C_2(m)+T\sum_{i=1}^2M_i\|\nabla\Phi_i\|^2_{L^\infty}\right)^{1/2}.
\ee 
Similarly, using the estimations from Theorem \ref{thm:reg_1step} and \eqref{estim:Lm}, we can write
\begin{align*}
\int_0^T\|(\tilde\rho^{1,\t}_t&+\tilde\rho^{2,\t}_t)^{m-1/2}\|_{L^2(\Om)}^2\dd t=\t\sum_{k=0}^{N-1}\|(\rho^{1,\t}_{k+1}+\rho^{2,\t}_{k+1})^{m-1/2}\|_{L^2(\Om)}^2\\
&\le \t\sum_{k=0}^{N-1}\left(C_\Om\|\nabla(\rho^{1,\t}_{k+1}+\rho^{2,\t}_{k+1})^{m-1/2}\|_{L^2(\Om)}+\|\rho^{1,\t}_{k+1}+\rho^{2,\t}_{k+1}\|_{L^m(\Om)}^{m-1/2} \sL^d(\Om)^{1/(2m)-1/2}\right)^2\\
&\le2\t \sum_{k=0}^{N-1}\left(C_\Om^2\|\nabla(\rho^{1,\t}_{k+1}+\rho^{2,\t}_{k+1})^{m-1/2}\|_{L^2(\Om)}^2 + \|\rho^{1,\t}_{k+1}+\rho^{2,\t}_{k+1}\|_{L^m(\Om)}^{2m-1} \sL^d(\Om)^{1/m-1}\right)\\
&\le 2\left(C_\Om^2 C_3(m)^2 +TC_1(m)^{2m-1} \sL^d(\Om)^{1/m-1}\right)
\end{align*}
Thus, the second estimation in \eqref{estim:L2grad} holds true with 
\be
C_4(m):=\sqrt{2}\left(C_\Om^2 C_3(m)^2 +TC_1(m)^{2m-1} \sL^d(\Om)^{1/m-1}\right)^{1/2}.
\ee
Using \eqref{estim:Lm}, for any $q\ge 1$ we can write similarly as before 
\begin{align*}
\int_0^T\|\tilde\rho^{1,\t}_t+\tilde\rho^{2,\t}_t\|_{L^m(\Om)}^q\dd t&=\t\sum_{k=0}^{N-1}\|\rho^{1,\t}_{k+1}+\rho^{2,\t}_{k+1}\|_{L^m(\Om)}^q \le T C_1(m)^q.
\end{align*}
So defining $C_5(q,m):=T^{1/q}C_1(m),$ one obtains the last estimation \eqref{estim:LqLm}
\end{proof}

In what follows -- using a refined version of the Aubin-Lions lemma -- we prove a strong compactness result for $\tilde\rho^{1,\t}+\tilde\rho^{2,\t}$  where $\tilde\rho^{1,\t}$ and $\tilde\rho^{2,\t}$ are the piecewise constant interpolations.

\begin{proposition}\label{prop:strong_comp}
Let $m\in(1,+\infty).$ Then the sequence of curves defined as $\tilde\rho^{1,\t_n}+\tilde\rho^{2,\t_n}$ (for any sequence $(\t_n)_{n\ge 0}$ of positive reals that converges to 0) is strongly pre-compact in $L^{2m-1}([0,T]\times\Om).$
\end{proposition}

\begin{proof}
We will use a refined version of the classical Aubin-Lions lemma to prove this result (see \cite{RosSav} and Theorem \ref{thm:aubin_refined}). Then we will argue as in in \cite{DiFMat}.

Let us set $B:=L^{2m-1}(\Om),$ $\fF:L^{2m-1}(\Om)\to[0,+\infty]$ defined as 
$$
\fF(\rho):=\left\{
\ba{ll}
\|\rho^{m-1/2}\|_{H^1(\Om)}, & \text{if } \rho\in H^1(\Om)\cap\sP^{M_1+M_2}(\Om),\\[5pt]
+\infty, & \text{otherwise}
\ea
\right.
$$
and $g:L^{2m-1}(\Om)\times L^{2m-1}(\Om)\to[0,+\infty]$ defined as 
$$
g(\mu,\nu):=\left\{
\ba{ll}
W_2(\mu,\nu), & \text{if } \mu,\nu \in \sP^{M_1+M_2}(\Om),\\[5pt]
+\infty, & \text{otherwise}.
\ea
\right.
$$
In this setting, $\left(\tilde\rho^{1,\t_n}+\tilde\rho^{2,\t_n}\right)_{n\ge 0}$ and $\fF$ satisfy the assumptions of Theorem \ref{thm:aubin_refined}. Indeed, from Theorem \ref{thm:L2H1} one has in particular that $\ds\int_0^T\|(\tilde\rho^{1,\t}_t+\tilde\rho^{2,\t}_t)^{m-1/2}\|_{H^1(\Om)}^2\dd t\le C_3(m)^2+C_4(m)^2.$ The injection $H^1(\Om)\hookrightarrow L^2(\Om)$ is compact, the injection $i:\eta\mapsto\eta^{\frac{2}{2m-1}}$ is continuous from $L^2(\Om)$ to $L^{2m-1}(\Om)$ and the sub-level sets of $\rho\mapsto\|\rho^{m-1/2}\|_{H^1(\Om)}$ are compact in $L^{2m-1}(\Om)$. 

Moreover, by Corollary \ref{cor:bdd}, Lemma \ref{lem:ineq_sum_1} and by the fact that $g$ defines a distance on  $D(\fF)$, one has that $g$ also satisfies the assumptions from Theorem \ref{thm:aubin_refined}, hence the implication of the theorem holds and one has that $\left(\tilde\rho^{1,\t_n}+\tilde\rho^{2,\t_n}\right)_{n\ge 0}$ is pre-compact in $\sM(0,T;L^{2m-1}).$ Finally, the uniform bound \eqref{estim:Lm} implies the strong pre-compactness of $\left(\tilde\rho^{1,\t_n}+\tilde\rho^{2,\t_n}\right)_{n\ge 0}$ in $L^{2m-1}([0,T]\times\Om).$
\end{proof}

\subsection{The limit systems as $\t\downarrow 0$}
We proceed with the final step of our scheme, i.e. as the time step size goes to zero, we show that 
along a subsequence the discrete solutions converge to yield a {\it very weak solution} of the PDE systems, in the sense of Definition \ref{def:weak_and_very_weak}.
We use the convention of $L^{2m-1}([0,T]\times\Om)=L^\infty([0,T]\times\Om)$ whenever $m=+\infty$. 

\begin{proposition}\label{prop:conv_unif-measure}
Let $m\in(1,+\infty]$ and let us consider any sequence  $(\t_n)_{n\ge 0}$ which converges to zero.
%such that $\t_n\da 0$ as $n\to+\infty$. 
Then, along a subsequence the following holds:

%up to passing to a subsequence that we do not relabel, the following statements hold true.
\begin{itemize}
\item[(1)] There exists $\rho^i\in AC^2([0,T];(\sP^{M_i}(\Om),W_2))\cap L^{2m-1}([0,T]\times\Om)$ ($i=1,2$) s.t. $\rho^{i,\t_n}\to\rho^i$ and $\tilde\rho^{i,\t_n}\to\rho^i$ as $n\to +\infty$ uniformly on $[0,T]$ w.r.t. $W_2,$ in particular weakly$-\star$ in $\sP^{M_i}(\Om)$ for all $t\in[0,T].$ \item[(2)] There exists $\E^i\in\sM^d([0,T]\times\Om)$ ($i=1,2$) s.t. $\E^{i,\t_n}\weaklys\E^i$ and $\tE^{i,\t_n}\weaklys \E^i$ as $n\to+\infty.$ 
\end{itemize}
\end{proposition}

\begin{proof}
In Lemma \ref{lem:estimates} we obtained uniform bounds on the metric derivative of the continuous interpolations $\rho^{i,\t_n}$ ($i=1,2$), which is enough to get compactness. More precisely there exist $[0,T]\ni t\mapsto\rho^i_t\in\sP^{M_i}(\Om),$ $i=1,2$ continuous curves such that (up to taking subsequences for $\t_n$) $\rho^{i,\t_n}_t\to\rho^i_t$ uniformly on $[0,T]$ w.r.t. $W_2$ as $n\to+\infty$, in particular weakly-$\star$ in $\sP^{M_i}(\Om)$ for all $t\in[0,T].$

The other interpolation $\tilde\rho^{i,\t_n}$ coincides with $\rho^{i,\t_n}$ at every node point $k\t,$ hence it is straightforward that (up to a subsequence taken for $\t_n$) it converges to the same curve $\rho^i$ uniformly on $[0,T]$ w.r.t. $W_2.$

Lemma \ref{lem:estimates} states also that $\E^{i,\t_n}$ and $\tE^{i,\t_n}$ are uniformly bounded sequences in $\sM^d([0,T]\times\Om),$ hence there exist $\E^i\in \sM^d([0,T]\times\Om)$ such that (up to a subsequence taken for $\t_n$) $\E^{i,\t_n}\weaklys \E^i$ and $\tE^{i,\t_n}\weaklys \E^i$ ($i=1,2$) in  $\sM^d([0,T]\times\Om)$ as $n\to+\infty$. The convergence of $\E^{i,\t_n}$ and $\tE^{i,\t_n}$ to the same limit $\E^i$ follows from the same argument as in the proof of \cite[Theorem 3.1]{MesSan}.
\end{proof}

These convergences imply that one can pass to the limit in the weak formulation \eqref{eq:cont_weak} as $\t\downarrow 0$ and obtain that $(\rho^i,\E^i)$ solves as well the continuity equation 
\be\label{eq:cont_lim}
\partial_t\rho^i+\diver \E^i=0
\ee
on $[0,T]\times\Om$ (with initial condition $\rho^i(0,\cdot)=\rho^i_0$)
in the same weak sense.

In particular, by Lemma \ref{lem:estimates} one has that the sequence $\left(\cB_2(\rho^{i,\t_n},\E^{i,\t_n})\right)_{n\in\N}$ is uniformly bounded for any positive vanishing sequence $(\t_n)_{n\in\N}$, where $\cB_2$ denotes the Benamou-Brenier action functional (see its precise definition and properties in Appendix \ref{sec:appendix_ot}). In particular, by the lower semicontinuity of this functional, there exists $\v^i$ such that $\v^i_t\in L^2_{\rho^i}(\Om;\R^d)$ for a.e. $t\in[0,T]$ and at the limit (as $\t\da 0$) $\E^i=\v^i\cdot\rho^i.$ This implies further that the equation \eqref{eq:cont_lim} has the form
\be\label{eq:cont_lim_2}
\partial_t\rho^i+\diver(\v^i\rho^i)=0.
\ee

\subsubsection{Precise form of the limit systems}
Now we shall work with the piecewise constant interpolations $\tilde\rho^{i,\t}, i=1,2$ and with the corresponding momenta $\tE^{i,\t}, i=1,2$ to determine more properties of the limit systems. The more precise convergence results are summarized in Theorem \ref{thm:precise} and \ref{thm:precise_infty} below.

\begin{theorem}\label{thm:precise} Let $m\in(1,+\infty)$ and let $\tilde\rho^{i,\t}, i=1,2$ be the piecewise constant interpolations between the densities $(\rho_{k}^{i,\t})_{k=0}^N$ and $\tE^{i,\t}, i=1,2$ the corresponding momentum variables. Taking any sequence $(\t_n)_{n\ge1}$ that goes to zero, the following holds along a subsequence:
%of positive real numbers that converges to zero, the following holds:
\begin{itemize}
\item[(1)] $\left(\tilde\rho^{1,\t_n}+\tilde\rho^{2,\t_n}\right)_{n\ge0}$ converges strongly in $L^{2m-1}([0,T]\times\Om)$ to $\rho^1+\rho^2$;
\item[(2)] $\left(\tE^{1,\t_n}+\tE^{2,\t_n}\right)_{n\ge 0}, i=1,2$ converges in the sense of distributions to $-\nabla(\rho^1+\rho^2)^{m}-\nabla\Phi_1\rho^1-\nabla\Phi_2\rho^2$, i.e., 
\begin{equation}\label{eq:momentum_sum}
\v^1\rho^1+\v^2\rho^2=-\nabla (\rho^1+\rho^2)^m-\nabla\Phi_1\rho^1-\nabla\Phi_2\rho^2
\end{equation}
in the sense of distributions on $[0,T]\times\Om.$
\end{itemize}
\end{theorem}

\begin{proof}
Let us show $(1)$. Proposition \ref{prop:strong_comp} implies already that $\tilde\rho^{1,\t_n}+\tilde\rho^{2,\t_n}$ (up to some subsequence that we do not relabel) converges strongly in $L^{2m-1}([0,T]\times\Om).$ Also, by Proposition \ref{prop:conv_unif-measure} we have that $\tilde\rho^{i,\t_n}_t\weaklys\rho^i_t$ as $n\to\infty$ for all $t\in[0,T].$ Hence the limit of $(\tilde\rho^{1,\t_n}+\tilde\rho^{2,\t_n})_{n\ge0}$ is precisely $\rho^1+\rho^2$ and $\rho^i\in L^{2m-1}([0,T]\times\Om)\cap AC^2([0,T];(\sP^{M_i}(\Om),W_2)), i=1,2.$ 

We show now $(2).$ By definition of $\tE^{i,\t}$ on has that
$$\tE^{1,\t}+\tE^{2,\t}=-\nabla(\tilde\rho^{1,\t}+\tilde\rho^{2,\t})^m-\nabla\Phi_1\tilde\rho^{1,\t}-\nabla\Phi_2\tilde\rho^{2,\t}.$$
Since by $(1)$ $\tilde\rho^{1,\t_n}+\tilde \rho^{2,\t_n}\to \rho^1+\rho^2$ strongly in $L^{2m-1}([0,T]\times\Om)$ as $n\to+\infty,$ one has that $(\tilde\rho^{1,\t_n}+\tilde\rho^{2,\t_n})^m\to (\rho^1+\rho^2)^m$ strongly in $L^{2-\frac1m}([0,T]\times\Om)$ as $n\to +\infty.$ This, together with the weak$-\star$ convergence of $(\tilde \rho^{i,\t_n})_{n\ge 0}$ to $\rho^i$ implies the first part of the statement. On the other hand one has obtained already that $\tE^{i,\t_n}\weakly \E^i=\v^i\rho^i$ as $n\to +\infty,$ thus \eqref{eq:momentum_sum} follows as well.
\end{proof}

\begin{remark}\label{rmk:conv}
\begin{itemize}
\item[(1)] Let us underline the fact that it is unclear whether we could show a stronger version of Theorem \ref{thm:precise}(2), i.e. the convergence (up to passing to a subsequence) of $\left(\tE^{i,\t_n}\right)_{n\ge 0}$ to $-\frac{m}{m-1}\rho^i\nabla (\rho^1+\rho^2)^{m-1}-\nabla\Phi_i\rho^i$ $i=1,2$, which is necessary in order to obtain the weak formulation of the PDE system at the limit .
\item[(2)] In Theorem \ref{thm:precise} if in addition $\left(\tilde\rho^{i,\t_n}\right)_{n\ge 0}$ either for $i=1$ or $i=2$ converges a.e. in $[0,T]\times\Om$ then both sequences ($i=1,2$) converge strongly in $L^{2m-1}([0,T]\times\Om)$ to $\rho^i$ and the corresponding momentum $\left(\tE^{i,\t_n}\right)_{n\ge 0}$ ($i=1,2$) converge in the sense of distributions to $-\frac{m}{m-1}\rho^i\nabla (\rho^1+\rho^2)^{m-1}-\nabla\Phi_i\rho^i.$ The study of a stable scenario when this holds true is the subject of Section \ref{sec:segregated_1D}.
%This is mainly due to the lack of the strong (or pointwise a.e.) convergence in some $L^p([0,T]\times\Om)$ space of the density sequences $\left(\tilde\rho^{i,\t_n}\right)_{n\ge 0}$ ``separately'': Theorem \ref{thm:precise}(3) shows exactly how such a convergence implies the precise limit. As mentioned before, such convergence can be obtained when solutions are guaranteed to be segregated,  as is the case when supposing \eqref{fmw:1D_ordered}.
\end{itemize}
\end{remark}
\begin{proof}[Proof of Remark \ref{rmk:conv}(2)]
First, clearly the pointwise convergence of $\left(\tilde\rho^{i,\t_n}\right)_{n\ge 0}$ and the strong convergence of $\left(\tilde\rho^{1,\t_n}+\tilde\rho^{2,\t_n}\right)_{n\ge 0}$ in $L^{2m-1}([0,T]\times\Om)$ by a suitable version of Vitali's convergence theorem %Lebesgue's dominated convergence theorem 
imply the strong convergence of $\left(\tilde\rho^{i,\t_n}\right)_{n\ge 0}$ in $L^{2m-1}([0,T]\times\Om)$ to $\rho^i$. Since one of the terms in the sum of these sequences and the sum itself converges strongly, so does the other term as well.
Let us recall the formula 
$$\tE^{i,\t_n}=-\frac{m}{m-1}\tilde\rho^{i,\t_n}\nabla(\tilde \rho^{1,\t_n}+\tilde \rho^{2,\t_n})^{m-1}-\nabla\Phi_1\tilde\rho^{1,\t_n}.$$
The strong convergence of $\left(\tilde\rho^{i,\t_n}\right)_{n\ge 0}$ implies that the second term of $\tE^{i,\t_n}$, i.e. $-\nabla\Phi_i\tilde\rho^{i,\t_n}$ (since $\nabla\Phi_i$ is in $L^\infty(\Om;\R^d)$) converges strongly in $L^{2m-1}([0,T]\times\Om;\R^d)$ to $-\nabla\Phi_i\rho^i.$ Thus in particular weakly-$\star$ in $\sM^d([0,T]\times\Om).$ The first term of $\tE^{i,\t_n}$ can be written as 
$$-\frac{m}{m-1}\tilde\rho^{i,\t_n}\nabla(\tilde \rho^{1,\t_n}+\tilde \rho^{2,\t_n})^{m-1}=-\frac{m}{m-1}\left[\nabla (\tilde\rho^{1,\t_n}+\tilde\rho^{2,\t_n})^{m-1}\right](\tilde\rho^{1,\t_n}+\tilde\rho^{2,\t_n})^{1/2}\frac{\tilde\rho^{i,\t_n}}{(\tilde\rho^{1,\t_n}+\tilde\rho^{2,\t_n})^{1/2}},$$
and notice furthermore that 
$$\ds-\frac{m}{m-1}\left[\nabla (\tilde\rho^{1,\t_n}+\tilde\rho^{2,\t_n})^{m-1}\right](\tilde\rho^{1,\t_n}+\tilde\rho^{2,\t_n})^{1/2}=-\frac{m}{m-1/2}\nabla(\tilde\rho^{1,\t_n}+\tilde\rho^{2,\t_n})^{m-1/2}.$$
Theorem \ref{thm:L2H1} implies that the sequence $\left(-\frac{m}{m-1/2}(\tilde\rho^{1,\t_n}+\tilde\rho^{2,\t_n})^{m-1/2}\right)_{n\ge0}$ is uniformly bounded in the space $\ds L^2([0,T];H^1(\Om)),$ hence there exists a subsequence (not relabeled) and some $\xi\in L^2([0,T];H^1(\Om))$ such that 
$\left(-\frac{m}{m-1/2}(\tilde\rho^{1,\t_n}+\tilde\rho^{2,\t_n})^{m-1/2}\right)_{n\ge0}$ is converging weakly to $\xi$ as $n\to+\infty.$ In particular, $-\frac{m}{m-1/2}(\tilde\rho^{1,\t_n}+\tilde\rho^{2,\t_n})^{m-1/2}\weakly\xi$ weakly in $L^2([0,T]\times\Om)$ and $-\frac{m}{m-1/2}\nabla(\tilde\rho^{1,\t_n}+\tilde\rho^{2,\t_n})^{m-1/2}\weakly\nabla \xi$ weakly in $L^2([0,T]\times\Om;\R^d)$ as $n\to+\infty.$

By Proposition \ref{prop:strong_comp} one has that $\left(\tilde\rho^{1,\t_n}+\tilde\rho^{2,\t_n}\right)_{n\ge 0}$ converges strongly to $\rho^{1}+\rho^{2}$ in $L^{2m-1}([0,T]\times\Om),$ which implies in particular that $\left((\tilde\rho^{1,\t_n}+\tilde\rho^{2,\t_n})^{m-1/2}\right)_{n\ge 0}$ converges strongly in $L^2([0,T]\times\Om)$ to $(\rho^{1}+\rho^{2})^{m-1/2}.$ This together with the above weak convergences implies that $\nabla\xi=-\frac{m}{m-1/2}\nabla(\rho^{1}+\rho^{2})^{m-1/2}\in L^2([0,T]\times\Om;\R^d).$

Now, by the strong convergence of $\tilde\rho^{i,\t_n}$ to $\rho^{i}$ in $L^{2m-1}([0,T]\times\Om)$ one has that 
$$\frac{\tilde\rho^{i,\t_n}}{(\tilde\rho^{1,\t_n}+\tilde\rho^{2,\t_n})^{1/2}} \to \frac{\rho^{i}}{(\rho^{1}+\rho^{2})^{1/2}}\ \ \text{as}\ \ n\to+\infty$$ 
pointwisely a.e. in $[0,T]\times\Om.$ Moreover, since the densities are non-negative one has $\frac{\tilde\rho^{i,\t_n}}{(\tilde\rho^{1,\t_n}+\tilde\rho^{2,\t_n})^{1/2}}\le (\tilde\rho^{i,\t_n})^{\frac12}$ and $(\tilde\rho^{i,\t_n})^{\frac12}$ converges to $(\rho^i)^{\frac12}$ strongly in $L^{2(2m-1)}([0,T]\times\Om).$ Hence Lebesgue's dominated convergence theorem implies that 
$$\ds \frac{\tilde\rho^{i,\t_n}}{(\tilde\rho^{1,\t_n}+\tilde\rho^{2,\t_n})^{1/2}} \to \frac{\rho^{i}}{(\rho^{1}+\rho^{2})^{1/2}}\ \ \text{as}\  n\to+\infty,$$ 
strongly in $L^{2(2m-1)}([0,T]\times\Om).$

Gluing together the two previous results, one obtains that $\ds -\frac{m}{m-1}\nabla (\tilde\rho^{1,\t_n}+\tilde\rho^{2,\t_n})^{m-1}\tilde\rho^{i,\t_n}$ converges weakly to $\ds -\frac{m}{m-1}\nabla(\rho^1+\rho^2)^{m-1}\rho^i$ in $L^r([0,T]\times\Om;\R^d)$ as $n\to+\infty$, where $\frac{1}{2}+\frac{1}{2(2m-1)}+\frac1r=1,$ i.e. $r=\frac{2(2m-1)}{2m-2}>1$. So in particular the convergence is weakly-$\star$ in $\sM^d([0,T]\times\Om)$, which together with the strong convergence of the term $\nabla\Phi_i\tilde\rho^{i,\t_n}$ implies the thesis.

\end{proof}

\begin{theorem}\label{thm:precise_infty}
Let $m=+\infty$ and let and let us consider $\tilde\rho^{i,\t}, i=1,2$ the piecewise constant interpolations between the densities $(\rho_{k}^{i,\t})_{k=0}^N$ and $\tE^{i,\t}, i=1,2$ the corresponding momentum variables. Let us consider moreover $p^\t$ the piecewise constant interpolations between the pressure variables $(p^\t_k)_{k=0}^{N-1}.$ Let us take any positive sequence  $(\t_n)_{n\ge 0}$ such that $\t_n\da 0$ as $n\to+\infty$, and let us consider the weak limit $p$ of $(p^{\t_n})_{n\ge 0}$ in $L^2([0,T];H^1(\Om))$ and $\rho^i$ the limit of $(\rho^{i,\t_n})_{n\ge 0}$ in $L^\infty([0,T];(\sP^{M_i}(\Om),W_2))$ (up to passing to a subsequence that we do not relabel). Then we have the following:
$$p(1-(\rho^1+\rho^2))=0\ \ \ae\ {\rm{in\ }}[0,T]\times\Om.$$
\end{theorem}

\begin{proof}
First notice that by Lemma \ref{lem:press_bound} and Proposition \ref{prop:conv_unif-measure} the weak limits $p$ and $\rho^i$ ($i=1,2$) exist. Furthermore, (1) follows from the previously mentioned results, Lemma \ref{lem:ineq_sum_1} and \cite[Lemma 3.5]{MesSan}.
\end{proof}

\begin{remark}\label{rmk:conv_infty}
In Theorem \ref{thm:precise_infty} if $\tilde\rho^{1,\t_n}$ and $\tilde\rho^{2,\t_n}$ are such that $\sL^d(\{\tilde\rho^{1,\t_n}_t>0\}\cap\{\tilde\rho^{2,\t_n}_t>0\})=0$ for all $t\in[0,T]$ and for all $n\in\N,$ then
$$\tE^{i,\t_n}\weaklys -\nabla p -\nabla\Phi_i\rho^i,\ \ {\rm{as\ }} n\to+\infty,\ {\rm{in\ }} \sM^d([0,T]\times\Om),\ i=1,2.$$
If moreover $\sL^d(\{\rho^{1}_t>0\}\cap\{\rho^{2}_t>0\})=0$ for a.e. $t\in[0,T]$, then
$$\nabla p=\rho^i\nabla p,\ \ae\ {\rm{in}}\ [0,T]\times\Om,\ \ i=1,2.$$
\end{remark}

\begin{proof}
To show the remark let us recall the form of the momentum variables, i.e.
$$\tE^{i,\t_n}=\tilde\v^{i,\t_n}\tilde\rho^{i,\t_n}=-\nabla p^{\t_n}\tilde\rho^{i,\t_n}-\nabla\Phi_i\tilde\rho^{i,\t_n}=-\nabla p^{\t_n}-\nabla\Phi_i\tilde\rho^{i,\t_n},$$
where the last equality holds true since $p^{\t_n}(1-(\tilde\rho^{1,\t_n}+\tilde\rho^{2,\t_n}))=0$ for all $t\in[0,T]$ and a.e. in $\Om$ and by the assumption  $\sL^d(\{\tilde\rho^{1,\t_n}_t>0\}\cap\{\tilde\rho^{2,\t_n}_t>0\})=0$ for all $t\in[0,T]$, one has $\{\tilde\rho^{1,\t_n}_t+\tilde\rho^{2,\t_n}_t=1\}=_{\ae}\{\tilde\rho^{1,\t_n}_t=1\}\cup\{\tilde\rho^{2,\t_n}_t=1\}$ for all $t\in[0,T]$ and the two sets are disjoint a.e. in $\Om$. By the weak convergences of $(p^{\t_n})_{n\ge 0}$ to $p$ and $(\tilde\rho^{i,\t_n})_{n\ge 0}$ to $\rho^i$ ($i=1,2$) one can easily conclude that 
$$\tE^{i,\t_n}\weaklys -\nabla p -\nabla\Phi_i\rho^i,\ \ {\rm{as\ }} n\to+\infty,\ {\rm{in\ }} \sM^d([0,T]\times\Om),\ i=1,2.$$
Also, $\nabla p=\rho^i\nabla p,\ i=1,2$ a.e. in $[0,T]\times\Om$ follows easily from (1) and the assumption $\sL^d(\{\tilde\rho^{1}_t>0\}\cap\{\tilde\rho^{2}_t>0\})=0$ for all $t\in[0,T]$, as desired.
\end{proof}

%Motivated by the geometric properties of the equilibrium solutions from section \ref{sec:eq_sol}, 

\subsection{Segregation of the densities}\label{subsec:separation_gen_dim}

As mentioned in the introduction, it seems natural to look for initial configurations of the system \eqref{eq:PME_m} where there is no mixing of the densities, to strengthen our convergence results in the continuum limit. We shall describe such initial configurations in one space dimension in the next section. Here we describe some properties of the time-discrete solutions (obtained by the JKO scheme) which hold for all dimensions. In particular, we show that when $\Phi_2=\Phi_1+C$ (for some $C\in\R$), then the densities stay segregated if initially they were so. We derive also some properties of the mixed region, when the two initial densities are mixed in a special way.  Still, these statements are only true for time-discrete solutions, and we cannot rule out the possibility that the limiting densities end up mixed, for instance, due to ``fingering" phenomena (see also the numerical observations in \cite{LorLorPer} which displays fingering phenomena when a system, similar to \eqref{eq:k_1-k_2}, has unstable combination of diffusion constants and source terms). It seems that additional geometric property is required to preserve the segregation property in the continuum limit.

\begin{proposition}\label{higher_D}
Let $m\in(1,+\infty]$. Let us assume moreover that $\Phi_1$ and $\Phi_2$ are such that $\Phi_2=\Phi_1+C$ on $\Om$ for some $C\in\R$, with the hypothesis \eqref{hyp:phi} fulfilled. Let $(\rho^1_0,\rho^2_0)\in\sP^{M_1}(\Om)\times\sP^{M_2}(\Om)$ satisfy \eqref{hyp:rho} and let $(\rho^1,\rho^2)$ be the minimizers in \eqref{gf:tau} constructed with the help of $(\rho^1_0,\rho^2_0)$. Then the following statements hold true.

\begin{itemize}
\item[(1)] If $\sL^d\left(\{\rho^1_0>0\}\cap\{\rho^2_0>0\}\right)=0,$ then $\sL^d\left(\{\rho^1>0\}\cap\{\rho^2>0\}\right)=0.$
\item[(2)] Let us define the Borel measurable sets $A:=\{\rho^1_0>0\}\cap\{\rho^2_0>0\}$ and $B:=\{\rho^1>0\}\cap\{\rho^2>0\}$. Let us suppose that $\sL^d(A)>0$ and $\sL^d(B)>0$. If  there exists $r>0$ such that  $r\rho^1_0\le \rho^2_0$ a.e. in $A$, then $r\rho^1\le \rho^2$ a.e. in $B$.
\end{itemize}
\end{proposition}

\begin{proof}
Let us use the notation $\nabla\Phi:=\nabla\Phi_1=\nabla\Phi_2.$ Using \eqref{opt:cond}, the optimal transport maps $T^i$ ($i=1,2$) in the transport of $\rho^i$ onto $\rho^i_0$ (see Lemma \ref{lem:opt_cond}-\ref{lem:opt_cond_infty}) can be written ($\rho^i-\ae$) as 
$$T^i=\left\{\ba{ll}
{\text{id}}+\t\left(\frac{m}{m-1}\nabla (\rho^1+\rho^2)^{m-1}+\nabla\Phi\right), & \text{if } m\in(1,+\infty),\\[5pt]
\id+\t(\nabla p +\nabla\Phi), & \text{if } m=\infty.
\ea
\right.$$
In particular, observe that $T^1=T^2$ a.e. in $\{\rho^1>0\}\cap\{\rho^2>0\}.$

We show (1). Suppose that the Borel measurable set $B:=\{\rho^1>0\}\cap\{\rho^2>0\}$ has positive Lebesgue measure. For any $x_0\in B$ such that $x_0$ is a Lebesgue point of $\rho^1$, $\rho^2$ and $T^1, T^2$ and $T^1(x_0)=T^2(x_0)$ is a Lebesgue point for both $\rho^1_0$ and $\rho^2_0,$ one has (since $T^1(x_0)=T^2(x_0)$) that $T^1(x_0)=T^2(x_0)\in\{\rho^1_0>0\}\cap\{\rho^2_0>0\}.$ In particular the positive mass of each  $\rho^i$ ($i=1,2$) on $B$ is transported onto $\{\rho^1_0>0\}\cap\{\rho^2_0>0\}$. On the other hand, since both $\rho^1_0$ and $\rho^2_0$ are absolutely continuous w.r.t. $\sL^d,$ this mass cannot be supported on an $\sL^d$-null set, which is a contradiction to the assumption $\sL^d\left(\{\rho^1_0>0\}\cap\{\rho^2_0>0\}\right)=0$.

\vspace{0.2cm}
We show (2). First observe that one can write two Jacobian equation in a weak sense, i.e. 
\be\label{eq:MA_12}
\det(DT^i)=\frac{\rho^i}{\rho^i_0\circ T^i},\ \ \rho^i-\ae.
\ee
Since the measures $\rho^i$, $i=1,2$ are absolutely continuous w.r.t. the Lebesgue measure, the maps $T^i$ are differentiable $\rho^i-\ae$ and the previous equation holds true  pointwisely $\rho^i-\ae$ (see for instance \cite[Theorem 3.1]{DePFig}). Let us choose $x_0\in B$ such that it is a Lebesgue point of both $\rho^1$ and $\rho^2$ and it is a point of differentiability of both $T^1$ and $T^2$ (in particular $\rho^1(x_0)>0$ and $\rho^2(x_0)>0$). Since the optimal transport maps coincide on the common support of $\rho^1$ and $\rho^2$, one may assume that $T(x_0):=T^1(x_0)=T^2(x_0)$ is a Lebesgue point of both $\rho^1_0$ and $\rho^2_0$. The Jacobian equation \eqref{eq:MA_12} yields that $\rho^2(x_0)/\rho^1(x_0)=\rho^2_0(T(x_0))/\rho^1_0(T(x_0))\ge r$, which concludes the proof.
\end{proof}

%\begin{proposition}
%Let $m\in [1,\infty]$. Let $e\in \R^d$ be given and suppose that $\nabla\Phi_1 - \nabla \Phi_2= e$  in ?. For the given densities $(\rho^1_0,\rho^2_0)$ let us consider the minimizer $(\rho^1, \rho^2) in $(MM_m)$. We suppose moreover that the following geometric configuration takes place: the sets $\{\rho_0^i>0\}$ are separated by the hyperplane $\{z\cdot e =0\} i.e. for $x\cdot \e< 0<y\cdot e $ for $ x\in\{\rho^1_0> 0}$ and $y\in \{\rho^2_0>0\}$. Then $L^d(\{\rho^1>0\}\cap\{\rho^2>0\})=0$.
%\end{proposition}

\section{Segregated weak solutions in 1D}\label{sec:segregated_1D}

In this section we study the local segregation property of the supports for the time-discrete solutions.  As a consequence we show the existence of {\it segregated weak solutions} of the systems \eqref{eq:PME_m} and \eqref{eq:PME_infty} in one spacial dimension.  

\subsection{Separation of the supports and ordering property in one space dimension} 

\begin{customthm}{Hyp-1D}\label{fmw:1D_ordered}
We set the following geometric framework (see also Figure \ref{fig:framework} below for illustration).
\begin{itemize}
\item[(1)] $d=1$, $\Om$ a bounded open interval, the potentials $\Phi_i$, $i=1,2$ are semi-convex and $C^1(\Om)$;
\item[(2)] The drifts are `ordered', in the sense that $\partial_x\Phi_2(x) \geq \partial_x\Phi_1(x)$ for all $x\in\Om$. This means in particular that $\Phi_2-\Phi_1$ is increasing;
\item[(3)] $\rho^1_0$ and $\rho^2_0$ are two densities such that for a.e. $x\in\{\rho^1_0>0\}$ and $y\in\{\rho^2_0>0\}$ one has that $y<x$. We refer to this last property as \emph{``ordering of the supports''} of the initial densities. This implies in particular that $\sL^1(\{\rho^1_0>0\}\cap\{\rho^2_0>0\})=0.$
\end{itemize}
\end{customthm}
Let us point out that the assumptions from \ref{fmw:1D_ordered} immediately imply with reasoning parallel to Proposition~\ref{higher_D} that $\sL^1(\{\rho^1>0\}\cap\{\rho^2>0\})=0$ for one-step minimizers $(\rho^1, \rho^2)$ given by \eqref{gf:tau}. To see this, suppose  $\{\rho^1>0\}\cap\{\rho^2>0\}=_{a.e}B$ for some Borel measurable set $B$ such that $\sL^1(B)>0.$  As in the proof of Proposition~\ref{higher_D}, The optimal transport map $T^i$ ($i=1,2$) in the transport of $\rho^i$ onto $\rho^i_0$ is given by 
$$T^i=\left\{\ba{ll}
{\text{id}}+\t\left(\frac{m}{m-1}\partial_x (\rho^1+\rho^2)^{m-1}+\partial_x\Phi_i\right), & \text{if } m\in(1,+\infty),\\[5pt]
\id+\t(\partial_x p +\partial_x\Phi_i), & \text{if } m=\infty.
\ea
\right.$$
%In particular $T^2=T^1+\t\left(\partial_x\Phi_2-\partial_x\Phi_1\right)$ a.e. on $B$. 
The above formula and the assumption $\partial_x\Phi_2-\partial_x\Phi_1\ge 0$ in $\Om$ yield that $T^2(x)\ge T^1(x)$ a.e. in $B$, which contradicts the ordering property of the initial data.

\medskip

Still, this separation property is not enough to iterate over time steps unless the ordering property of the initial configuration is preserved for $(\rho^1, \rho^2)$. This is what we prove next. 

\begin{proposition}\label{prop:separation_1D}
Let $m\in(1,+\infty]$ and suppose  the assumptions in \eqref{fmw:1D_ordered} and the hypotheses \eqref{hyp:phi}-\eqref{hyp:rho} are in place. Let us denote by $(\rho^1, \rho^2)$ the one-step time discrete solutions given by \eqref{gf:tau} for $k=0$. Then the \emph{ordering property} from  \eqref{fmw:1D_ordered} holds true for $\{\rho^1>0\}$ and $\{\rho^2>0\}$.
\end{proposition}

\medskip

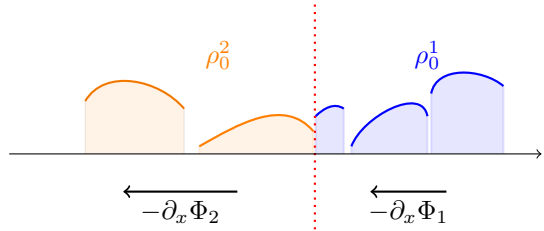
\begin{figure}[h]
\begin{tikzpicture}

%% \rho^2_0
\draw [thick, orange]  (9,0.7) to [out=60, in=130] (10.3,0.6);
\draw[fill=orange, opacity=0.1] (9,0) to (9,0.7) to [out=60, in=130] (10.3,0.6) -- (10.3,0);
\draw [thick, orange]  (10.5,0.1) to [out=30, in=130] (12,0.3);
\draw [thick, orange, fill=orange, opacity=0.1] (10.5,0) -- (10.5,0.1) to [out=30, in=130] (12,0.3) -- (12,0);
%\draw [thick, orange, dashed] (12,0.3) -- (12,0);
\node [above, orange] at (10.75,1) {$\rho^2_0$};
% the drift
\draw[<-, thick] (9.5,-0.5) -- (11,-0.5);
\node [below] at (10.25,-0.5) {$-\partial_x\Phi_2$};

%% \rho^1_0
\draw [thick, blue] (12.03,0.5) to [out=45, in=145] (12.4,0.6);
\draw [thick, blue, fill=blue, opacity=0.1] (12.03,0) -- (12.03,0.5) to [out=45, in=145] (12.4,0.6) -- (12.4,0);
\draw [thick, blue] (12.5,0.1) to [out=70, in=100] (13.5,0.5);
\draw [thick, blue, fill=blue, opacity=0.1] (12.5,0) -- (12.5,0.1) to [out=70, in=100] (13.5,0.5) -- (13.5,0);
\draw [thick, blue] (13.55,0.8) to [out=80, in=140] (14.5,0.9);
\draw [thick, blue, fill=blue, opacity=0.1] (13.55,0) -- (13.55,0.8) to [out=80, in=140] (14.5,0.9) -- (14.5,0);
\node [above, blue] at (13.5,1) {$\rho^1_0$};
% the drift
\draw[<-, thick] (12.75,-0.5) -- (13.75,-0.5);
\node [below] at (13.25,-0.5) {$-\partial_x\Phi_1$};

%separation line

\draw [thick, dotted, red] (12.02,-1) -- (12.02,2);

% new coordinate axes
\draw [->] (8,0) -- (15,0);
%\draw [->] (9.2,-1) -- (9.2,2);

%bracket and \Om_j under the configuration 
%\draw[decoration={brace,mirror,raise=5pt},decorate] (9,-1) -- node[below=6pt] {$\Om_j$} (14.5,-1);
\end{tikzpicture}
\caption{Ordering of the supports of the initial data} \label{fig:framework}
\end{figure}

\begin{proof}

Suppose the contrary, i.e. there exist $B_1\subseteq \{\rho^1>0\}$ and $B^2\subseteq \{\rho^2>0\}$ with $\sL^1(B^1)>0$ and $\sL^1(B_2)>0$ such that for a.e. $x\in B^1$ and $y\in B^2$ $x<y$ (see Figure \ref{fig:1D_contrary} for illustration). 

\emph{Claim:} there exist $E^i\subseteq B^i$, $i=1,2$ Borel measurable sets, $\theta>0$ and $\d>0$ such that $\sL^1(E^1)=\sL^1(E^2)>0$, $E^2=E^1+\theta$ and $\rho^i\ge \d$ a.e. on $E^i$, $i=1,2$ (see Figure \ref{fig:1D_contrary} for illustration).

\emph{Proof of the claim.} Let us take $x_0\in B^1, y_0\in B^2$  Lebesgue points. This means in particular that $\rho^1(x_0)>0,$ $\rho^2(y_0)>0$ and  
\be\label{eq:Lebesgue}
\lim_{r\da 0}\fint_{B_r(x_0)}\left| \rho^1(x) -\rho^1(x_0)\right|\dd x = 0,\ \ \lim_{r\da 0}\fint_{B_r(y_0)}\left| \rho^2(x) -\rho^2(y_0)\right|\dd x = 0.
\ee
Now let us take $r>0$ small (we fix it later) and let $\d:=\min\left\{\rho^1(x_0)/2,\rho^2(y_0)/2\right\}.$ Let us consider moreover the measurable sets $\tilde E^1\subseteq B^1\cap B_r(x_0)$ and $\tilde E^2\subseteq B^2\cap B_r(y_0)$ defined as $\tilde E^i:=\left\{\rho^i\ge \d\right\},\ i=1,2.$ By construction, for $r>0$ small enough one has that $\sL^1(B_r(x_0)\setminus \tilde E^1)/ \sL^1(B_r(x_0))\le 1/3$ and $\sL^1(B_r(y_0)\setminus\tilde E^2)/ \sL^1(B_r(y_0))\le 1/3.$ Indeed, one has
\begin{align*}
\fint_{B_r(x_0)}\left| \rho^1(x) -\rho^1(x_0)\right|\dd x&\ge \frac{1}{\sL^1(B_r(x_0))}\int_{B_r(x_0)\setminus \tilde E^1}\left| \rho^1(x) -\rho^1(x_0)\right|\dd x\\
&\ge\frac{\rho^1(x_0)}{2}\frac{\sL^1(B_r(x_0)\setminus \tilde E^1)}{\sL^1(B_r(x_0))},
\end{align*}
and by \eqref{eq:Lebesgue} the l.h.s. tends to 0 as $r\da 0,$ so for $r>0$ small enough $\sL^1(B_r(x_0)\setminus \tilde E^1)/\sL^1(B_r(x_0))\le 1/3$. Similarly for $\rho^2$ and $\tilde E^2.$ Fix such an $r>0.$

Furthermore, set $\theta := y_0-x_0$ and define $E^1:= \tilde E^1\cap (\tilde E^2-\theta)$ and $E^2:=E^1+\theta.$ Thus,
$$\frac{\sL^1(E^2)}{\sL^1(B_r(y_0))}=\frac{\sL^1(E^1)}{\sL^1(B_r(x_0))}\ge  1 - \frac{\sL^1(B_r(x_0)\setminus \tilde E^1)}{\sL^1(B_r(x_0))}-\frac{\sL^1(B_r(y_0)\setminus\tilde E^2)}{\sL^1(B_r(y_0))}=\frac13.$$
This finishes the proof of the claim, since $r>0$ is a fixed small number.

Now we construct a new competitor $(\tilde\rho^1,\tilde\rho^2)$ in \eqref{gf:tau} which has less energy (we refer to Figure \ref{fig:1D_contrary} for the illustration) than $(\rho^1,\rho^2)$, yielding the contradiction. Define $\tilde\rho^1$ and $\tilde\rho^2$ as
$$
\tilde\rho^1=\left\{
\ba{ll}
\rho^1, & \text{in}\ \Om\setminus(E^1\cup E^2),\\
\rho^1-\d, & \text{in}\ E^1,\\
\d, & \text{in}\ E^2,
\ea
\right.\ \ 
\text{and}\ \ 
\tilde\rho^2=\left\{
\ba{ll}
\rho^2, & \text{in}\ \Om\setminus(E^1\cup E^2),\\
\d, & \text{in}\ E^1,\\
\rho^2-\d, & \text{in}\ E^2.
\ea
\right.
$$
We construct corresponding transport maps (not necessarily optimal ones), $\tilde T^1$ between $\tilde\rho^1$ and $\rho^1_0$ and $\tilde T^2$ between $\tilde\rho^2$ and $\rho^2_0$ as
$$
\tilde T^1=\left\{
\ba{ll}
T^1, & \text{in}\ \Om\setminus E^2,\\
T^1(\cdot-\theta), & \text{in}\ E^2,
\ea
\right.
\ \ \ \text{and}\ \ \ 
\tilde T^2=\left\{
\ba{ll}
T^2, & \text{in}\ \Om\setminus E^1,\\
T^2(\cdot+\theta), & \text{in}\ E^1.
\ea
\right.
$$
By construction $\tilde T^i_\#\tilde\rho^i=\rho^i_0,$ $i=1,2$. Let us use the notation $E^1_0:=T^1(E^1)$ and $E^2_0:=T^2(E^2)$, these are Borel measurable sets and subsets of $\{\rho^1_0>0\}$ and $\{\rho^2_0>0\}$ respectively.

\begin{figure}[h]
\begin{tikzpicture}
%% we draw first the densities \rho^1 and \rho^2 when they violate the ordering property
% first orange bump
\draw [thick, orange] (1,0) to [out=90, in=110] (2.5,0.5);
\draw [thick, orange, dashed] (2.5,0.5) -- (2.5,0);
\draw [fill=orange, opacity=0.1] (1,0) to [out=90, in=110] (2.5,0.5) to (2.5,0);
\node [above, orange] at (1.75,1) {$\rho^2$};

% middle blue bump
\draw [thick, blue] (2.7,0.7) to [out=50, in=100] (4,0);
\draw [thick, blue, dashed] (2.7,0.7) -- (2.7,0);
\draw [fill=blue, opacity=0.1] (2.7,0.7) to [out=50, in=100] (4,0) to (2.7,0);
\node [above, blue] at (3.35,1) {$\rho^1$};
% the gray part in the blue bump 
\path [fill=lightgray] (3.3,0) -- (3.3,0.33)  -- (3.7,0.33) -- (3.7,0); % (3.45,0.33) -- (3.45,0)   (3.5,0) -- (3.5,0.33)
\draw [thick] (3.3,0.33)  -- (3.7, 0.33); % (3.45, 0.33) (3.5,0.33)
\draw [dashed] (3.3,0) -- (3.3,0.33) (3.7,0.33) -- (3.7,0); % (3.45, 0.33) -- (3.45,0) (3.5,0) -- (3.5,0.33)

\draw [line width=5pt, blue] (3.31,0) -- (3.69,0);
\draw [fill] (3.4, 0) circle  [radius=0.04];
\node [below, blue] at (3.45,0) {$x_0$};
\draw [->,dashed] (2.8, -0.5) -- (3.3, -0.1);
\node [below, blue] at (2.7, -0.4) {$E^1$};

%% right orange bump

\draw [thick, orange] (4.2,0.2) to [out=60, in=180] (4.7,1) to [out=0, in=220] (5.6, 0.55);
\draw [thick, orange, dashed] (5.6, 0.55) -- (5.6, 0); 
\draw [fill=orange, opacity=0.1] (4.2,0) to (4.2,0.2) to [out=60, in=180] (4.7,1) to [out=0, in=220] (5.6, 0.55) to (5.6,0);
\node [above, orange] at (4.9,1) {$\rho^2$};
% the gray part in the orange bump
\path [fill=lightgray] (4.7,0) -- (4.7,0.33)  -- (5.1,0.33) -- (5.1,0); 
\draw [thick] (4.7,0.33)  -- (5.1, 0.33); 
\draw [dashed] (4.7,0) -- (4.7,0.33) (5.1,0.33) -- (5.1,0); 

\draw [line width=5pt, orange] (4.71,0) -- (5.09,0);
\draw [fill] (4.8, 0) circle  [radius=0.04];
\node [below, orange] at (4.8,0) {$y_0$};
\draw [->,dashed] (5.6, -0.5) -- (5.1, -0.1);
\node [below, orange] at (5.7, -0.4) {$E^2$};

\draw [<->, thick] (3.4,-0.5) -- (4.8,-0.5);
\node [below] at (4.2,-0.5) {$\theta$};

% coordinate axes, plus the \delta on the y-axis
\draw [->] (0,0) -- (7,0);
%\draw [->] (1.2,-1) -- (1.2,2);
\draw [dashed] (0.8,0.33) -- (6,0.33);
\draw [fill] (1, 0.33) circle  [radius=0.04];
\node [left] at (1,0.33) {$\delta$};

%then draw the original configurations \rho^1_0 and \rho^2_0
\draw [thick, orange, dashed] (9.5,1.2) to [out=-10, in=180] (10.5,1.1);
\draw [thick, orange] (10.5,1.1) to [out=-10, in=130] (12,0.3);
\draw [thick, orange, dashed] (12,0.3) -- (12,0);
\draw [fill=orange, opacity=0.1] (9.5,0) to (9.5,1.2) to [out=-10, in=180] (10.5,1.1) to [out=-10, in=130] (12,0.3) to (12,0);
\node [above, orange] at (10.75,1) {$\rho^2_0$};
\draw [line width=5pt, orange] (11.4,0) -- (11.8,0);
\node [below, orange] at (11.6,0) {$E^2_0$};

\draw [thick, blue] (12.5,0.1) to [out=70, in=200] (13.5,1);
\draw [thick, blue, dashed] (13.5,1) to [out=30, in=190] (14.5,1.5);
%\draw [thick, blue, dashed] (9.5,1) to [out=60, in=180] (10.5,1.1);
\draw [fill=blue, opacity=0.1] (12.5,0) to (12.5,0.1) to [out=70, in=200] (13.5,1) to [out=30, in=190] (14.5,1.5) to (14.5,0);
\node [above, blue] at (13.5,1) {$\rho^1_0$};
\draw [line width=5pt, blue] (12.7,0) -- (13,0);
\node [below, blue] at (12.85,0) {$E^1_0$};

% new coordinate axes
\draw [->] (8,0) -- (15,0);
%\draw [->] (9.2,-1) -- (9.2,2);
\end{tikzpicture}
\caption{Ordering property for $\{\rho^1_0>0\}$ and $\{\rho^2_0>0\}$ (on the right). This is is violated by $\{\rho^1>0\}$ and $\{\rho^2>0\}$ (on the left)} \label{fig:1D_contrary}
\end{figure}
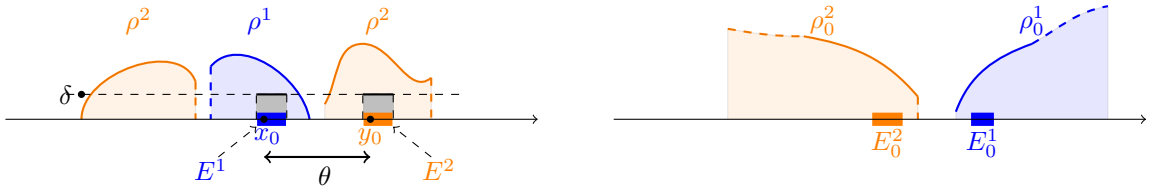 

Notice that by construction $\tilde \rho^1+\tilde\rho^2=\rho^1+\rho^2$ in $\Om$, hence 
\be\label{en:internal_change}
\cF_m(\tilde\rho^1,\tilde\rho^2)=\cF_m(\rho^1,\rho^2).
\ee
Now let us see how the other two energy terms in \eqref{gf:tau} change by considering $(\tilde\rho^1,\tilde\rho^2)$ as competitors. Let us use the notation $h(x):=\partial_x\Phi_2(x)-\partial_x\Phi_1(x)$.
First, 
\begin{align*}
E_\cG&:=\cG(\tilde\rho^1,\tilde\rho^2)-\cG(\rho^1,\rho^2)=\int_\Om\Phi_1\tilde\rho^1\dd x+\int_\Om\Phi_2\tilde\rho^2\dd x-\int_\Om\Phi_1\rho^1\dd x-\int_\Om\Phi_2\rho^2\dd x\\
&=\d\left(\int_{E^2}\Phi_1(x)\dd x-\int_{E^1}\Phi_1(x)\dd x\right)+\d \left(\int_{E^1}\Phi_2(x)\dd x-\int_{E^2}\Phi_2(x)\dd x\right)\\
&=\int_{E^1}\d \left[(\Phi_1(x+\theta)-\Phi_2(x+\theta))-(\Phi_1(x)-\Phi_2(x))\right]\dd x\\
&=\int_{E^1}\d\theta\left[\partial_x\Phi_1(\xi_{x,\theta})-\partial_x\Phi_2(\xi_{x,\theta})\right]\dd x
%&\le -\d\int_{E^1}h(x)\theta +(\theta^2/2)h'(\xi_{x,\theta})\dd x,
\end{align*}
where in the last equality we used the mean value theorem and $\xi_{x,\theta}$ is some point in $(x,x+\theta)$. 
We compute now the change in the $W_2$ terms. Recall the structure of the transport maps $\tilde T^i$, $i=1,2$ and mind that they might be not optimal. Thus one has
\begin{align*}
E_{W_2}&:=\frac{1}{2\t}W_2^2(\tilde\rho^1,\rho^1_0)+\frac{1}{2\t}W_2^2(\tilde\rho^2,\rho^2_0)-\frac{1}{2\t}W_2^2(\rho^1,\rho^1_0)-\frac{1}{2\t}W_2^2(\rho^2,\rho^2_0) \\
&\le \frac{1}{2\t}\int_{E^2}|x-T^1(x-\theta)|^2\d \dd x+\frac{1}{2\t}\int_{E^1}|x-T^2(x+\theta)|^2\d \dd x\\ 
& - \frac{1}{2\t}\int_{E^1}|x-T^1(x)|^2\d \dd x-\frac{1}{2\t}\int_{E^2}|x-T^2(x)|^2\d \dd x\\
&= \frac{1}{2\t}\int_{E^1}\left(|x+\theta-T^1(x)|^2-|x-T^1(x)|^2\right)\d \dd x\\
&+\frac{1}{2\t}\int_{E^2}\left(|x-\theta-T^2(x)|^2-|x-T^2(x)|^2\right)\d\dd x\\
&=\frac{\d\theta}{\t}\int_{E^1}(T^2(x+\theta)-T^1(x))\dd x
%&=\frac{1}{\t}|\theta|^2\d \sL^1(E^1)+\frac{1}{\t}\int_{E^1}\theta\cdot(x-T^1(x))\d \dd x-\frac{1}{\t}\int_{E^2}\theta\cdot(x-T^2(x))\d \dd x\\
%&=\frac{1}{\t}|\theta|^2\d\sL^1(E^1) +\frac{1}{\t}\theta\cdot\left(\bar_{E^1}(\eta^1)-\bar_{E^1_0}(\eta^1_0)\right)-\frac{1}{\t}\theta\cdot \left( \bar_{E^2}(\eta^2) - \bar_{E^2_0}(\eta^2_0) \right)\\
%&=\frac{1}{\t}|\theta|^2\d\sL^1(E^1) + \frac{1}{\t}\theta\cdot\left(\bar_{E^1}(\eta^1)-\bar_{E_2}(\eta^2)\right)+\frac{1}{\t}\theta\cdot \left(\bar_{E^2_0}(\eta^2_0) - \bar_{E^1_0}(\eta^1_0) \right),
\end{align*}
where $\eta^i=\d\cdot\sL^1\mres E^i$ and $\eta^i_0=T^i_\#\eta^i,$ $i=1,2.$  %Here, for the nonnegative measure $\mu$ on $\Om$, we denote by $\bar_{\Om}(\mu)$  the \emph{barycenter} (or center of mass) of $\mu$ on $\Om$, defined as $\bar_{\Om}(\mu):=\int_{\Om} x\dd\mu(x)\in\text{conv}(\Om),$ and $\text{conv}(\Om)$ denotes the convex hull of $\Om$. If $T_\#\mu=\nu,$ one has as well $\bar_{\Om}(\nu)=\int_{\Om} y\dd \nu(y)=\int_{\Om} T(x)\dd\mu(x)$. 

Now, it is easy to see that $E_\cG+E_{W_2}<0$.
Indeed, by the assumptions (2) from \eqref{fmw:1D_ordered} one has that $\partial_x\Phi_1-\partial_x\Phi_2$ nonpositive, thus 
\begin{equation}\label{eq:error}
E_\cG+E_{W_2}\le \d\theta\int_{E^1}\left\{\left[\partial_x\Phi_1(\xi_{x,\theta})-\partial_x\Phi_2(\xi_{x,\theta})\right]+\frac{1}{\t}[T^2(x+\theta)-T^1(x)]\right\}\dd x
\end{equation}
is negative since by the assumption (3) from \eqref{fmw:1D_ordered} $T^2(x+\theta)-T^1(x)<0$.

Thus one concludes that $E_\cG+E_{W_2}<0$, which together with \eqref{en:internal_change} imply that $(\tilde\rho^1,\tilde\rho^2)$ is a better competitor than $(\rho^1,\rho^2)$. This is clearly a contradiction to the uniqueness of the minimizer in \eqref{gf:tau}. Thus the ordering property for $\{\rho^1>0\}$ and $\{\rho^2>0\}$ follows.
\end{proof}

\subsection{Discussion on possibly mixed initial data}

Extending the above proposition to more general cases seems to be challenging, due to possible presence of the {\it mixing zone} $\{\rho_1>0\}\cap\{\rho_2>0\}$. The main issue, for instance to localize our argument, would be to ensure the finite propagation of mixing zone. The only available result in this direction arises in the case of the stiff pressure limit, $m=\infty$, $\Phi_i = c_i x$, and with full saturation, that is when we have the constraint $\rho_1+\rho_2 = 1$. In this case  Otto (\cite{Ott1}) showed in one dimensional setting that there is a unique description of the mixing zone that propagates with finite speed generated by the entropy solution of a conservation law. While we are not sure whether the same uniqueness results hold for our undersaturated case, we believe that the mixing zone should travel with finite speed at least in one dimension.

\subsection{Existence of a solution for \eqref{eq:PME_m}  supposing  \eqref{fmw:1D_ordered}}\label{sec:existence-1D}

\begin{theorem}\label{thm:sep_limit}
Let us suppose that $m\in(1,+\infty]$ and the setting of \eqref{fmw:1D_ordered} takes place. Let us consider $(\rho^1,\rho^2)$ to be any subsequential limit (uniformly in time w.r.t. $W_2$) of the piecewise constant interpolation curves $(\tilde\rho^{1,\t_n},\tilde\rho^{2,\t_n})$ when $\t_n\da 0$, with the initial densities $(\rho^1_0,\rho^2_0)$. Then  $(\rho^1,\rho^2)$ satisfies
$$\sL^1\left(\{\rho^1_t>0\}\cap\{\rho^2_t>0\}\right)=0, \ \forall t\in[0,T]$$
and the sets $\{\rho^1_t>0\}$ and $\{\rho^2_t>0\}$ are \emph{ordered} in the sense of  \eqref{fmw:1D_ordered} for all $t\in[0,T].$
\end{theorem}

\begin{proof}
First, let us recall that the \emph{ordering} of $\{\rho^1_0>0\}$ and $\{\rho^2_0>0\}$ in \eqref{fmw:1D_ordered} is such that $\{\rho^2_0>0\}$ is to the left of $\{\rho^1_0>0\}.$

Second, let us underline that by Proposition \ref{prop:conv_unif-measure} (1) $\rho^i$ is obtained as the uniform limit in time w.r.t. $W_2$ (as $\t\da 0$) of the piecewise constant interpolation curves $\tilde \rho^{i,\t}$ ($i=1,2$). 
For a fixed time step $\t>0$, considering the above mentioned interpolations,  we introduce the following functions $I^{1,\t},I^{2,\t}:[0,T]\to\ov\Om$ defined as
$$I^{1,\t}(t):=\inf\left\{x: x\in \Leb\left(\{\tilde \rho^{1,\t}_t>0\}\right) \right\}\ \ 
\text{and}\ \  
I^{2,\t}(t):=\sup\left\{x: x\in \Leb\left(\{\tilde \rho^{2,\t}_t>0\}\right) \right\}.$$
These functions are well-defined, since $\Om$ is bounded and in particular Proposition \ref{prop:separation_1D} implies that $I^{2,\t}(t)\le I^{1,\t}(t)$ for all $t\in[0,T]$ and for any $\t>0$. Also, by the boundedness of $\Om$, these functions are uniformly bounded in $t$ and $\t$.

Let us take a sequence $\left(\t_n\right)_{n\ge 0}$, s.t. $\t_n\da 0$ as $n\to+\infty$ and $\sup_{t\in[0,T]}W_2(\tilde\rho^{i,\t}_t,\rho^{i}_t)\to 0$ as $n\to+\infty$, ($i=1,2$). $\left(I^{i,\t_n}(t)\right)_{n\ge 0}$ is a bounded sequence for each $t\in[0,T]$, so up to passing to a subsequence (that we do not relabel), it has a poitwise limit as $n\to+\infty$ that we denote by $I^i(t)$ for $t\in[0,T]$ and $i=1,2$. Now we show  the following.

{\it Claim:}  
\begin{itemize}
\item[(1)] $\rho^2_t(y)=0$ for a.e. $y>I^2(t)$ \hbox{ and \quad (2) }  $\rho^1_t(x)=0$ for a.e. $x<I^1(t)$.\\
\end{itemize}

{\it Proof of the claim.} Let us suppose that the claim is false, i.e. the first statement fails to be true (the proof of (2) is parallel). Then there exits $r>0$ and $\d>0$ small such that 
$$\int_{I^2(t)+r}^{I^2(t)+2r}\rho^2_t(x)\dd x>\d>0.$$ 
But, for $n\in\N$ large enough such that $\ds\left| I^{2,\t_n}(t)-I^2(t) \right|<r/2$ one has that
$$W_2^2(\tilde\rho^{2,\t_n}_t,\rho^2_t)\ge (r/2)^2\int_{I^2(t)+r}^{I^2(t)+2r}\rho^2_t(x)\dd x=(r/2)^2\d,$$
which yields a contradiction to the fact that $W_2(\tilde\rho^{2,\t}_t,\rho^{2}_t)\to 0$ as $n\to+\infty$. A similar argument can be performed to show (2), thus the claim follows.

Now, since $I^{2,\t_n}(t)\le I^{1,\t_n}(t)$ for all $n\in\N$ and $t\in[0,T],$ after passing to subsequences if necessary, one has that $I^2(t)\le I^1(t)$ for any limit points $I^1(t),I^2(t)$ and for all $t\in[0,T].$ This together with the Claim imply that $\sL^1\left(\{\rho^1_t>0\}\cap\{\rho^2_t>0\}\right)=0, \ \forall t\in[0,T]$
and that the sets $\{\rho^1_t>0\}$ and $\{\rho^2_t>0\}$ are ordered in the sense of \eqref{fmw:1D_ordered} for all $t\in[0,T].$ The result follows.
\end{proof}

\begin{remark}
When $m\in(1,+\infty),$ the above result allows to obtain the strong convergence result of the density sequences $(\tilde\rho^{i,\t_n})_{n\ge 0}$, $i=1,2$ separately. When $m=+\infty$, together with Proposition \ref{prop:separation_1D} this result is crucial to fulfill the hypotheses in Remark \ref{rmk:conv_infty}, which will lead to the precise weak form of the \eqref{eq:PME_infty} system.  
\end{remark}

\begin{theorem}\label{thm:separation_limit_1D}
Let us suppose that $m\in(1,+\infty)$ and the setting of  \eqref{fmw:1D_ordered} takes place. Consider the piecewise constant interpolations $\tilde\rho^{i,\t_n}$ ($i=1,2$) for some $(\t_n)_{n\ge 0}$ such that $\t_n\da 0$ as $n\to+\infty$. Then up to passing to a subsequence with $(\t_n)_{n\ge 0},$ $\left(\tilde\rho^{i,\t_n}\right)_{n\ge 0}$ ($i=1,2$) converges strongly in $L^{2m-1}([0,T]\times\Om)$, in particular pointwise a.e. in $[0,T]\times\Om.$
\end{theorem}

\begin{proof} Let us show first that $(\tilde\rho^{i,\t_n})_{n\ge 0}$ (up to passing to a subsequence) converges strongly to $\rho^i$ ($i=1,2$) in $L^1([0,T]\times\Om)$. We pass to subsequences if necessary (that we do not relabel) to ensure that $(\tilde\rho^{1,\t_n}+\tilde\rho^{2,\t_n})_{n\ge 0}$ converges strongly to $\rho^1+\rho^2$ in $L^{2m-1}([0,T]\times\Om)$ and $(\tilde\rho^{i,\t_n})_{n\ge 0}$ converges to $\rho^i$ ($i=1,2$)  weakly in $L^{2m-1}([0,T]\times\Om)$ as $n\to+\infty$. We compute
\begin{align*}
\|\tilde\rho^{i,\t_n}-\rho^i\|_{L^1([0,T]\times\Om)}&=\int_{[0,T]\times\Om}|\tilde\rho^{i,\t_n}-\rho^i|\dd t\otimes\dd x\\
&=\int_{\{\rho^i>0\}}|\tilde\rho^{i,\t_n}-\rho^i|\dd t\otimes\dd x +\int_{([0,T]\times\Om)\setminus\{\rho^i>0\}}\tilde\rho^{i,\t_n}\dd t\otimes \dd x\\
&=\int_{\{\rho^i>0\}}|\tilde\rho^{1,\t_n}+\tilde\rho^{2,\t_n}-(\rho^1+\rho^2)|\dd t\otimes \dd x+\int_{([0,T]\times\Om)\setminus\{\rho^i>0\}}\tilde\rho^{i,\t_n}\dd t\otimes \dd x\\
&\le \int_{[0,T]\times\Om}|\tilde\rho^{1,\t_n}+\tilde\rho^{2,\t_n}-(\rho^1+\rho^2)|\dd t\otimes \dd x\\
&+\int_0^T\int_{([0,T]\times\Om)\setminus\{\rho^i>0\}}\tilde\rho^{i,\t_n}\dd t\otimes \dd x\\
&\to 0, \ \ \text{as}\ n\to+\infty,
\end{align*}
where in the third equality we used the facts (see Theorem \ref{thm:sep_limit}) that $\rho^i=0$ a.e. in $\{\rho^{i+1}>0\}$ and $\tilde\rho^{i,\t_n}=0$ a.e. in $\{\rho^{i+1}>0\}$, with the convention $i+1=1,$ when $i=2$. Moreover, both terms in the last sum converge to 0. Indeed, the convergence of the first term is a consequence of the strong convergence of $(\tilde\rho^{1,\t_n}+\tilde\rho^{2,\t_n})_{n\ge 0}$ to $\rho^1+\rho^2$ in $L^{2m-1}([0,T]\times\Om)$ as $n\to+\infty$. The convergence to 0 of the last term is a consequence of the weak convergence of $\tilde\rho^{i,\t_n}$ to $\rho^i$ in $L^{2m-1}([0,T]\times\Om).$ 

This together with Theorem \ref{thm:L2H1} and Proposition \ref{prop:strong_comp} imply that (up to passing to a subsequence) $\left(\rho^{i,\t_n}\right)_{n\ge 0}$ converges strongly in $L^{2m-1}([0,T]\times\Om)$.
\end{proof}

We state now the results on the existence of weak solutions of the PDE systems \eqref{eq:PME_m} and \eqref{eq:PME_infty}.

\begin{theorem}\label{thm:existence_PME}
Let us assume that $m\in(1,+\infty)$, the hypotheses \eqref{hyp:rho} and \eqref{hyp:phi} are fulfilled and the setting in  \eqref{fmw:1D_ordered} takes place. Then the system \eqref{eq:PME_m} has a weak solution $(\rho^1,\rho^2)$ in the sense of \eqref{PME_m_weak} such that $\rho^i\in L^{2m-1}([0,T]\times\Om)\cap AC^2([0,T];(\sP^{M_i}(\Om),W_2)), i=1,2$ and $(\rho^1+\rho^2)^{m-1/2}\in L^2([0,T]; H^1(\Om))$. In addition, $\rho^i\in L^q([0,T]\times\Om)$ for all $1\le q\le m$ and $\E^i:=-\frac{m}{m-1}\partial_x(\rho^1+\rho^2)^{m-1}\rho^i-\partial_x\Phi_i\rho^i$ belongs to $L^r([0,T]\times\Om;\R^d)$ for some $1\le r< 2$ with uniform bounds  in $m$. Lastly, if $m\to+\infty,$ $q$ can be arbitrary large and $r$ can be chosen arbitrary close to 2.
\end{theorem}

\begin{proof}
By Theorem \ref{thm:precise} one has that the limit densities $\rho^1$ and $\rho^2$ belong to 
$$L^{2m-1}([0,T]\times\Om)\cap AC^2([0,T];(\sP^{M_i}(\Om),W_2)).$$ 
The same theorem establishes the convergence of $(\tilde\rho^{i,\t},\tE^{i,\t})$ and the precise form of the limit. By the fact that $(\tilde\rho^{i,\t},\tE^{i,\t})$ and $(\rho^{i,\t},\E^{i,\t})$ converge weakly as measures to the same limit $(\rho^i,\E^i)$ and by the fact that this latter pair solves the continuity equation \eqref{eq:cont_lim} in the weak sense \eqref{eq:cont_weak}, so does the precise limit of $(\tilde\rho^{i,\t},\tE^{i,\t})$ developed in Remark \ref{rmk:conv}(2) (notice that by Theorem \ref{thm:separation_limit_1D} the assumptions in Remark \ref{rmk:conv}(2) are fulfilled). This means in particular that the limit equation reads (for $i=1,2$) as
$$\partial_t\rho^i-\partial_x\left(\frac{m}{m-1}\partial_x(\rho^1+\rho^2)^{m-1}\rho^i+\partial_x\Phi_i\rho^i\right)=0,$$
that has to be understood in the weak sense \eqref{PME_m_weak} with no-flux boundary condition.  

Finally, let us obtain the uniform (w.r.t $m$) bounds on $\rho^i$ and $\E^i$. First, by Theorem \ref{thm:L2H1} \eqref{estim:LqLm} one has that the limit curves are bounded in  $L^p([0,T];L^m(\Om))$ for all $p\ge 1$ with uniform bounds. Take $p=m$ and any $1\le q\le m$. Then H\"older's inequality yields
\be\label{estim:Lq_unif_rho}
\|\rho^i\|_{L^q([0,T]\times\Om)}\le (T\sL^1(\Om))^{\frac{m-q}{qm}}\|\rho^i\|_{L^m([0,T]\times\Om)}\le T^{\frac{1}{q}}\sL^1(\Om)^{\frac{m-q}{qm}}C_1(m).
\ee

Second, let us write 
$$\E^i=\frac{m}{m-1/2}\partial_x(\rho^1+\rho^2)^{m-1/2}\frac{\rho^i}{(\rho^1+\rho^2)^{1/2}}+\partial_x\Phi_i\rho^i.$$
Notice that by \eqref{estim:L2grad} (Theorem \ref{thm:L2H1}) the $L^2$ bound for $\partial_x(\rho^1+\rho^2)^{m-1/2}$ remains the same after passing to the limit with the time step $\t$. Also, by the previous bound on $\rho^i$, $\frac{\rho^i}{(\rho^1+\rho^2)^{1/2}}$ is bounded uniformly in $L^{2q}([0,T]\times\Om)$ (since $\frac{\rho^i}{(\rho^1+\rho^2)^{1/2}}\le (\rho^i)^{1/2}$ a.e.). These observations, together with the fact that $\partial_x\Phi_i$ is uniformly bounded let us conclude by H\"older's inequality that
\be\label{estim:Lr_unif_E}
\|\E^i\|_{L^r}\le\frac{m}{m-1/2}\|\partial_x(\rho^1+\rho^2)^{m-1/2}\|_{L^2}\|\rho^i\|_{L^{r/(2-r)}}^{1/2}+\|\partial_x\Phi_i\|_{L^\infty}\|\rho^i\|_{L^r},
\ee
provided $1\le r<2$ and $\max\left\{\frac{r}{2-r},r\right\}\le q.$

Thus the thesis of the theorem follows.
\end{proof}

\begin{lemma}\label{lem:m_limit}
Let $m\in(1,+\infty)$ and let us consider $(\rho^{1}, \rho^{2})$ the solution of \eqref{eq:PME_m} supposing \eqref{fmw:1D_ordered} with given initial data $(\rho^1_0,\rho^2_0)$. We assume -- similarly to the hypotheses \eqref{hyp:rho} in the $m=+\infty$ case -- that the measure of $\Om$ is large enough, i.e. 
\be\label{eq:Om_size}
\sL^1(\Om)> (M_1+M_2),
\ee
where $M_i$ denotes the total mass of $\rho^i_0.$ Then -- uniformly in $m$ -- we have the following regularity estimates
\begin{itemize}
\item[(1)] $(\rho^{1}+\rho^{2})^m\in L^1([0,T]; C^{0,\alpha}(\Om))$ for some $0<\a<1/2$ and in particular it is uniformly bounded in $L^r([0,T]\times\Om)$ for some $1<r<2$;
\item[(2)] $(\rho^{1}+\rho^{2})^{m-1/2}\in L^1([0,T]; C^{0,1/2}(\Om))$ and in particular it is uniformly bounded in $L^2([0,T]\times\Om)$.
\end{itemize}
\end{lemma}

\begin{proof}
We show (1). Using the notations from Theorem \ref{thm:existence_PME}, one has that 
$$\E^1+\E^2=-\partial_x(\rho^1+\rho^2)^m-\partial_x\Phi_1\rho^1-\partial_x\Phi_2\rho^2.$$
By the estimations from Theorem \ref{thm:existence_PME} we know that the quantities $\E^1+\E^2$ and $\partial_x\Phi_1\rho^1+\partial_x\Phi_2\rho^2$ are bounded uniformly in $L^r([0,T]\times\Om;\R^d)$ for some $1<r<2$.

This implies first that $\partial_x(\rho^1+\rho^2)^m$ is uniformly bounded in $L^r([0,T]\times\Om;\R^d)$. Furthermore, the Poincar\'e-Wirtinger inequality yields that 
\be\label{eq:poincare}
\left\|(\rho^1+\rho^2)^m-\frac{1}{T\sL^1(\Om)}\int_0^T\int_\Om(\rho^1+\rho^2)^m\dd x\dd t\right\|_{L^r([0,T]\times\Om)}\le\left\|\partial_x(\rho^1+\rho^2)^m\right\|_{L^r([0,T]\times\Om)}.
\ee
So the l.h.s. is uniformly bounded. Let us show that the average of $(\rho^1+\rho^2)^m$ is uniformly bounded. 

Let us fix $0<\e<1-(M_1+M_2)/\sL^1(\Om)$.

{\it Claim 1. For every $t\in[0,T]$ there exists $r>0$ and $x_0\in\Om$ such that $\rho^1_t+\rho^2_t\le 1-\e$ a.e. in $B_r(x_0).$}

Fix  $t\in[0,T]$. Let us suppose that the claim is not true. Then in every ball $B_r(x_0)$, $\rho^1_t+\rho^2_t> 1-\e$ a.e. Since $\Om$ is a bounded interval, this in particular means that $\rho^1_t+\rho^2_t>1-\e$ a.e. in $\Om$. Furthermore,
$$\int_\Om (\rho^1_t+\rho^2_t)\dd x>(1-\e)\sL^1(\Om)> M_1+M_2,$$
and this is clearly a contradiction (by the choice of $\e$) to fact that $\int_\Om (\rho^1_t+\rho^2_t)\dd x=M_1+M_2$, thus the claim follows.

{\it Claim 2. $(\rho^{1}+\rho^{2})^m\in L^1([0,T]; C^{0,\alpha}(\Om))$ for some $0<\a<1/2$. In particular, for a.e. $t\in[0,T]$, $(\rho^1_t+\rho_t^2)^m$ has bounded oscillation uniformly in $m$.}

Notice that for $f:[0,T]\times\Om\to\R$ measurable such that $\partial_x f\in L^r([0,T]\times\Om)$ for some $r>1$ and $a,b\in\Om,\ a<b$ defining 
$$
{\rm osc}_{[a,b]} f_t:= \sup_{x\in [a,b]} f_t(x) - \inf_{x\in[a,b]} f_t(x),
$$
one has the estimate
\begin{align*}
\int_0^T ({\rm osc }_{[a,b]} f_t)\dd t & \leq \int_0^T \int_a^b | \partial_xf_t|\dd x\dd t \leq \left(\int_0^T \int_{a}^b |\partial_x f_t|^r \dd x\dd t\right)^{\frac1r} (T|a-b|)^{\frac{1}{r'}}\\
&\le \left(\int_0^T \int_{\Om} |\partial_x f_t|^r\dd x\dd t\right)^{\frac1r} (T|a-b|)^{\frac{1}{r'}},
\end{align*}
where $1/r+1/r'=1$. Since the integrand on the l.h.s. of the previous inequality is non-negative, this implies that 
$${\rm osc }_{[a,b]} f_t\le C|a-b|^\frac{1}{r'}$$
for a.e. $t\in[0,T],$ hence in particular $f_t$ has bounded oscillation, with a constant that depends only on $\|\partial_x f\|_{L^r}$ and $T$. Applying this reasoning to $(\rho^1+\rho^2)^m$, one obtains the statement of the claim.

Now Claim 1 and Claim 2 imply that $(\rho^1_t+\rho^2_t)^m$ is uniformly bounded for a.e. $t\in T$. This means furthermore that the average $\frac{1}{T\sL^1(\Om)}\int_0^T\int_\Om(\rho^1+\rho^2)^m\dd x\dd t $ is uniformly bounded, which together with \eqref{eq:poincare} implies (1).

\smallskip

The proof of (2) follows the same lines. The bound $\|\partial_x(\rho^{1,m}+\rho^{2,m})^{m-1/2}\|_{L^2([0,T]\times\Om)}\le C_3(m)$ in \eqref{estim:L2grad} from Theorem \ref{thm:L2H1} remains uniform, since $C_3(m)$ remains bounded uniformly when $m\to+\infty$. This bound is enough to perform the same analysis as in (1), thus we can conclude the same way.

\end{proof}

\begin{theorem}\label{thm:existence_PME_inf}
Let us assume that $m=+\infty$, the hypotheses \eqref{hyp:rho} and \eqref{hyp:phi} are fulfilled and the setting in \eqref{fmw:1D_ordered} takes place. Then the system \eqref{eq:PME_infty} 
has a weak solution $(\rho^1,\rho^2,p)$ in the sense of \eqref{PME_m_weak} such that $\rho^i\in L^{\infty}([0,T]\times\Om)\cap AC^2([0,T];(\sP^{M_i}(\Om),W_2)), i=1,2$ and $p\in L^2([0,T];H^1(\Om))$. One has moreover $\rho^1+\rho^2\le 1$ a.e. in $[0,T]\times\Om$,  $p\ge 0$, 
$p(1-\rho^i)=0$ a.e. in $\{\rho^i>0\}$, $i=1,2$. \end{theorem}

\begin{proof}
Let us take a positive vanishing sequence of time steps $(\t_n)_{n\ge0}$ and consider the piecewise constant and continuous interpolations of density curves $(\tilde\rho^{i,\t_n})_{n\ge 0}$, $(\rho^{i,\t_n})_{n\ge 0}$ and momenta $(\tE^{i,\t_n})_{n\ge 0}$, $(\E^{i,\t_n})_{n\ge 0}$. By Proposition \ref{prop:conv_unif-measure} (up to passing to subsequences) these objects converge (to $\rho^i$ and $\E^i$ respectively) in the appropriate weak senses and one has a limit system as in \eqref{eq:cont_lim}-\eqref{eq:cont_lim_2}. To identify a precise form of the system, we use the fact that the momentum sequences $(\tE^{i,\t_n})_{n\ge 0}$ and $(\E^{i,\t_n})_{n\ge 0}$ and the curve sequences $(\tilde\rho^{i,\t_n})_{n\ge 0}$ and $(\rho^{i,\t_n})_{n\ge 0}$ converge to the same limit. 

Now observe that the setting in \eqref{fmw:1D_ordered} implies that Theorem \ref{thm:sep_limit} can be applied, so the assumptions of Remark \ref{rmk:conv_infty} are fulfilled. This implies that the limit momenta have the form 
$$\E^i=-\partial_x p-\partial_x\Phi_i\rho^i=-\partial_x p\rho^i-\partial_x\Phi_i\rho^i, \ \ i=1,2.$$
Here $p\in L^2([0,T];H^1(\Om))$ is the weak limit of $(p^{\t_n})_{n\ge 0}$ obtained in Theorem \ref{thm:precise_infty}, so in particular $p\ge 0$ and $p(1-(\rho^1+\rho^2))=0$ a.e. in $[0,T]\times\Om.$ These imply that the limit system has the form
$$\partial_t\rho^i-\partial_x\left(\partial_x p\rho^i+\partial_x\Phi_i\rho^i\right)=0,\ \ i=1,2,$$
which has to be understood in the weak sense with no-flux boundary conditions. 

At last, since Theorem \ref{thm:sep_limit} implies in particular that $\sL^1(\{\rho^1_t>0\}\cap \{\rho^2_t>0\})=0$ for all $t\in[0,T]$, the relation $p(1-(\rho^1+\rho^2))=0$ a.e. in $[0,T]\times\Om$ reads as $p(1-\rho^i)=0$ a.e. in $\{\rho^i>0\}$, $i=1,2$.

\end{proof}

It is not hard to verify that, for a fixed $\tau>0$, the functionals in \eqref{gf:tau} $\Gamma$-convergence 
as $m\to\infty$ to the functional where $\cF_m$ is replaced by $\cF_\infty$. Thus, it is natural to pose the question about the convergence of the corresponding gradient flow solutions in the spirit of Sandier and Serfaty (see \cite{SanSer}). Unfortunately, one cannot use these kinds of results directly, and obtain the convergence of the continuum solutions of \eqref{eq:PME_m} to the solutions of \eqref{eq:PME_infty}, mainly due to the lack of uniqueness. Hence, it is necessary to proceed by studying the convergence of the continuum solutions at the PDE level. This will be addressed in the next section.

\subsection{Passing to the limit as $m\to+\infty$}\label{sec:limit_with_m}
We will show that solutions of \eqref{eq:PME_m} converge, along a subsequence as $m\to+\infty$, to a solution of \eqref{eq:PME_infty}. 

We suppose that the initial data satisfy
$$\|\rho^1_0+\rho^2_0\|_{L^\infty}\le 1.$$ 

Let us recall (see Definition \ref{def:weak_and_very_weak}) that a triple of nonnegative functions $(\rho^{1,\infty},\rho^{2,\infty},p^\infty)$, such that 
$\rho^{i,\infty}\in AC^2([0,T];\sP^{M_i}(\Om))$, $\|\rho^{1,\infty}+\rho^{2,\infty}\|_{L^{\infty}([0,T]\times \Om )} \leq 1$,  and $p^\infty\in L^2([0,T];H^1(\Om)),$ is a weak solution of \eqref{eq:PME_infty} if for any $\phi\in C^1([0,T]\times \Om)$ and $0<s<t\leq T$ we have
$$
-\int_s^t\int_\Om\rho^{i,\infty}\partial_t\phi\dd x\dd \t-\int_s^t\int_\Om (\partial_x p+\partial_x\Phi_i)\rho^{i,\infty}\cdot\partial_x\phi\dd x\dd \t=\int_\Om\rho^{i,\infty}(s,x)\phi(s,x)\dd x-\int_\Om\rho^{i,\infty}(t,x)\phi(t,x)\dd x,
$$
and $p^\infty(1-\rho^{1,\infty}-\rho^{2,\infty})=0$ a.e. in $[0,T]\times\Om$.

\begin{theorem}\label{thm:limit_m}
Let $(\rho^{1,m},\rho^{2,m})$ be a weak solution to \eqref{eq:PME_m} in the setting of \eqref{fmw:1D_ordered} with initial data satisfying $\|\rho^1_0+\rho^2_0\|_{L^\infty}\le 1$, where now we have noted $m$ as a parameter. We assume moreover that the geometric condition \eqref{eq:Om_size} holds true for the domain $\Om$.

Then, there exist $\rho^{i,\infty}\in L^\infty([0,T]\times\Om)\cap AC^2([0,T];(\sP^{M_i}(\Om),W_2))$ and $p^\infty\in L^r([0,T]; W^{1,r}(\Om))$ for all $r\in(1,2)$, such that along a subsequence when $m\to+\infty$,  $\rho^{i,m}\weakly \rho^{i,\infty}$ weakly in $L^q([0,T]\times\Om)$ for any $q\ge 1$, $(\rho^{1,m} + \rho^{2,m})^{m}\weakly p^\infty$ in $L^r([0,T]; W^{1,r}(\Om))$ and $\E^{i,m}\weakly\partial_x p^\infty\rho^{i,\infty}+\partial_x\Phi_i\rho^{i,\infty}$ weakly in $L^r([0,T]\times \Om)$.
 
Moreover $(\rho^{1,\infty},\rho^{2,\infty},p^\infty)$ is a weak solution of \eqref{eq:PME_infty}.
\end{theorem}

\begin{proof}
Let us recall the weak formulation of the system \eqref{eq:PME_m}. 
\be\label{eq:weak_passing_limit}
-\int_s^t\int_\Om\rho^{i,m}\partial_t\phi\dd x\dd \t-\int_s^t\int_\Om\E^{i,m}\cdot\partial_x\phi\dd x\dd \t=\int_\Om\rho^{i,m}_s(x)\phi(s,x)\dd x-\int_\Om\rho^{i,m}_t(x)\phi(t,x)\dd x,
\ee
for all $0\le s<t\le T$ and $\phi\in C^1([0,T]\times\Om),$ where $\E^{i,m}:=-\left(\frac{m}{m-1}\partial_x(\rho^{1,m}+\rho^{2,m})^{m-1}+\partial_x\Phi_i\right)\rho^{i,m}.$ 

First, by the assumption $\|\rho^1_0+\rho^2_0\|_{L^\infty}\le 1$, the bounds for $\rho^{i,m}$ in $AC^2([0,T];(\sP^{M_i}(\Om),W_2))$ (see Lemma \ref{lem:estimates} and \eqref{const:metric_der}) are uniform in $m$, so clearly up to passing to a subsequence with $m$, $(\rho^{i,m})_{m>1}$ converges weakly-$\star$ to some $\rho^{i,\infty}\in AC^2([0,T];(\sP^{M_i}(\Om),W_2))$. In particular this convergence is uniform in time w.r.t. $W_2$. By the uniform estimation \eqref{estim:Lq_unif_rho}, it follows that along a subsequence $(\rho^{i,m})_{m>1}$ converges weakly  to $\rho^{i,\infty}$ in $L^q([0,T]\times\Om)$ for all $q\ge 1$. In particular, these weak convergences allow us to obtain that in \eqref{eq:weak_passing_limit} the first term on the l.h.s., and  both terms on the r.h.s. pass to the limit.

\smallskip

Second, Lemma \ref{lem:m_limit} ensures the uniform boundedness of $(\rho^{1,m}+\rho^{2,m})^m$ in $L^r([0,T]; W^{1,r}(\Om))$ for some $r\in(1,2)$ (where $r$ can be chosen arbitrarily close to 2 for $m$ large enough) hence there exists $p^\infty\in L^r([0,T]; W^{1,r}(\Om))$ such that up to passing to a subsequence in $m$, $(\rho^{1,m} + \rho^{2,m})^{m}\weakly p^\infty$ weakly in $L^r([0,T]; W^{1,r}(\Om))$. In particular, one  has also that $\partial_x(\rho^{1,m} + \rho^{2,m})^{m}\weakly \partial_x p^\infty$ weakly in $L^r([0,T]\times\Om)$.

\smallskip

Notice that the convergence  $\E^{i,m}\weakly\partial_x p^\infty\rho^{i,\infty}+\partial_x\Phi_i\rho^{i,\infty}$ weakly in $L^r([0,T]\times \Om)$ is much more delicate, since both terms in the product $-\frac{m}{m-1}\partial_x(\rho^{1,m}+\rho^{2,m})^{m-1}\rho^{i,m}$ (in the definition of $\E^{i,m})$ converge only weakly.  

\smallskip

We shall provide the convergence $\E^{i,m}\weakly\partial_x p^\infty\rho^{2,\infty}+\partial_x\Phi_i\rho^{i,\infty}$ only for $i=2$, the other case is analogous. Observe that by the uniform estimation (in $m$) on $\E^{2,m}$ in $L^r([0,T]\times\Om)$ (for some $1<r<2$) (see \eqref{estim:Lr_unif_E}), there exists $\E^{2,\infty}\in L^r([0,T]\times\Om)$ such that up to passing to a subsequence, $\E^{2,m}\weakly \E^{2,\infty}$ weakly in $L^r([0,T]\times\Om)$ as $m\to+\infty$. Now let us identify the limit $\E^{2,\infty}.$ Let us fix a subsequence (that for simplicity of notation we denote by $m$), such that $(\rho^{2,m})_{m>1}$ converges weakly to $\rho^{2,\infty}$ and $(\E^{2,m})_{m>1}$ converges weakly to $\E^{2,\infty}$ as $m\to+\infty$ in the previously described spaces. 

Let us fix $0\le s<t\le T$. For all $\t\in[s,t]$, we define $I^{2,m}(\t)$ be the ``right-most point'' of the support of $\rho^{2,m}_\t$, i.e. 
$$I^{2,m}(\t):=\sup\left\{x: x\in \Leb\big(\{\rho^{2,m}_\t>0\}\big) \right\}$$
and let us define the set
$$
\mathcal{I} = \left\{(\t,x)\in[s,t]\times\Om: \exists\ (\t_n,x_n)_{n\ge 0}, \hbox{ s.t. } x_n= I^{2,m_n}(\t_n), \hbox{ and } (\t_n,x_n)\to(\t,x) \hbox{ as } n\to\infty\right\},
$$
where $(m_n)_{n\ge 0}$ is a subsequence of the previously chosen subsequence. Then $\mathcal{I}$ is a closed subset of $[s,t]\times\Om$.
Let $\mathcal{I}(\tau):= \mathcal{I} \cap\left(\{\tau\}\times\Om\right)$. Note that in particular $\mathcal{I}(\t)$ is the collection of all subsequential limits of $I^{2,m}(\t).$ 

Observe that if for some $\t\in(s,t)$, $y\in\Om$ lies to the left of $\mathcal{I(\t)}$, i.e. if $y<x$ for any $x\in\mathcal{I}(\tau)$, then $(\t,y)$ lies in the complement of $\{\rho^{1,m}>0\}$ for sufficiently large $m$, and thus defining 
$$\mathcal{J}^-:= \{(\t,y)\in(s,t)\times\Om: y<x \hbox{ for any } x\in \mathcal{I}(\tau)\},$$
one has that when restricted to $\cJ^-$, 
$$\E^{2,m}=-\partial_x (\rho^{2,m})^m-\partial_x\Phi\rho^{2,m},$$ 
in the sense of distributions for sufficiently large $m$. Similarly, defining 
$$\mathcal{J}^+:= \{(\t,y)\in(s,t)\times\Om: y>x \hbox{ for any } x\in \mathcal{I}(\tau)\},$$ 
one has that $\rho^{2,m}=0$ a.e. on $\cJ^+,$ hence when restricted to $\cJ^+$, $\E^{2,m}=0$ a.e. for sufficiently large $m.$ Clearly, $\cJ^-$, $\mathcal{I},\cJ^+$ are Lebesgue measurable and one can write $(s,t)\times\Om=\cJ^-\cup\mathcal{I}\cup\cJ^+$, $\sL^1\mres[0,T]\otimes\sL^{1}\mres\Om-$a.e. Thus, we write furthermore
\begin{align*}
\int_s^t\int_\Om\E^{2,m}\cdot\partial_x\phi\dd x\dd \t&=\int_{\cJ^-}\E^{2,m}\cdot\partial_x\phi\dd \t\otimes\dd x+\int_{\mathcal{I}}\E^{2,m}\cdot\partial_x\phi\dd \t\otimes\dd x+\int_{\cJ^+}\E^{2,m}\cdot\partial_x\phi\dd \t\otimes \dd x\\
&=\int_{\cJ^-}\E^{2,m}\cdot\partial_x\phi\dd \t\otimes\dd x+\int_{\mathcal{I}}\E^{2,m}\cdot\partial_x\phi\dd \t\otimes\dd x.
\end{align*}
Moreover, the very same decomposition remains valid for the weak limit $\E^{2,\infty}$ as well.

\smallskip

{\it Claim 1. $\int_{\mathcal{I}}\E^{2,\infty}\cdot\partial_x\phi\dd \t\otimes\dd x$ can be made arbitrarily small for any  smooth test function $\phi$.} 

If $(\sL^1\otimes\sL^1)(\mathcal{I})=0$, then this is obvious. So one can suppose that this set has positive measure. To show the claim, let us define the width of $\mathcal{I}(\tau)$, i.e.
$$
W(\tau):= \max\{|x-y|: x,y\in\mathcal{I}(\tau)\}=x^2_\t-x^1_\t,
$$
where $x^2_\t:=\max\{x: x\in\mathcal{I}(\t)\}$ and $x^1_\t:=\min\{x: x\in\mathcal{I}(\t)\}$ and these are well-defined since $\mathcal{I}(\t)$ is compact. Let $T_n:= \left\{\t\in(s,t): W(\t) \geq \frac{2}{n}\right\}.$
 Then $\mathcal{I} \subset A_n \cup B_n$, where 
$$
A_n:=\bigcup_{\t\in T_n} \left(\{\t\}\times (x^1_\t+1/n,x^2_\t) \right)\;\; \mbox{and}\;\; B_n = \mathcal{I} \setminus A_n.
$$
Note that $A_n$ and $B_n$ are Lebesgue measurable and the measure of $B_n$ in $[0,T]\times\Om$ is at most $2T/n$, which goes to zero as $n\to\infty$. %There are two cases to consider now. 

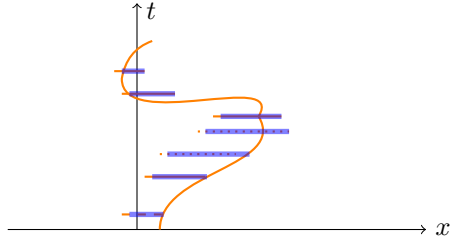
\begin{figure}[h]
\begin{tikzpicture}
\draw [->] (9.2,-1) -- (9.2,2);
\draw [->] (7.5,-1) -- (13,-1);
\node [right] at (9.2,1.9) {$t$};
\node [right] at (13,-1) {$x$};

\draw [thick, orange] (9.5,-1) to [out=90, in=300] (10.8,0.5) to [out=60, in=270] (9, 1) to [out=80, in= 200] (9.4, 1.5);
\draw[thick, orange, dashed] (9,-0.8) to (9.6,-0.8);
\draw[line width=2pt, blue, opacity=0.5] (9.1,-0.8) -- (9.55,-0.8);

\draw[thick, orange] (9.3,-0.3) to (10.1,-0.3);
\draw[line width=2pt, blue, opacity=0.5] (9.4,-0.3) -- (10.12,-0.3);

\draw[thick, orange, dotted] (9.5,0) to (10.5,0);
\draw[line width=2pt, blue, opacity=0.5] (9.6,0) -- (10.68,0);

\draw[thick, orange, dotted] (10,0.3) to (11.2,0.3);
\draw[line width=2pt, blue, opacity=0.5] (10.1,0.3) -- (11.2,0.3);

\draw[thick, orange] (10.2,0.5) to (11.1,0.5);
\draw[line width=2pt, blue, opacity=0.5] (10.3,0.5) -- (11.1,0.5);

\draw[thick, orange] (9,0.8) to (9.7,0.8);
\draw[line width=2pt, blue, opacity=0.5] (9.1,0.8) -- (9.7,0.8);

\draw[thick, orange] (8.9,1.1) to (9.3,1.1);
\draw[line width=2pt, blue, opacity=0.5] (9,1.1) -- (9.3,1.1);

\end{tikzpicture}
\caption{The sets {\color{orange}$\mathcal{I}$} (with orange) and {\color{blue}$A_n$}} (with blue) \label{fig:I_An}
\end{figure}

%{\it Case 1.} For a given $\t\in T_n$ there exists $\ov\t\in(\t-1/n^3,\t+1/n^3)\cap T_n$. Then, due to the $W_2$-equicontinuity of the entire $\rho^{2,m}$ sequence in time, one has that 
%$$W_2(\rho^{2,m}_\t,\rho^{2,m}_{\ov\t})\le C|\t-\ov\t|^{\frac12}.$$
We know that $W_2(\rho^{i,m}_\t,\rho^{2,\infty}_\t)\to 0$ as $m\to+\infty$, uniformly in $\t$. This implies in particular that the sequence $(\rho^{2,m}_\t)_{m}$ are Cauchy  w.r.t. $W_2$ uniformly in $\t$, which means that for any $n\in\mathbb{N},$ there exists $N(n)>0$ such that for all $\t\in[0,T]$
\be\label{eq:Cauchy}
W_2^2(\rho^{2,m_k}_\t,\rho^{2,m_l}_\t)\le\frac{1}{n^3}, \ \forall k,l>N(n),
\ee
where $m_k$ and $m_l$ denote elements of the sequence denoted by $m$. Now, let us define $A_n(\t):=A_n\cap \{\t\}\times\Om.$

Let us show that for all $\t\in(s,t)$ on $A_n(\t)$ all elements of the sequence $(\rho^{2,m}_\t)_m$ have small mass. Indeed, on the one hand, for any point $x\in A_n(\t)\cap\mathcal{I}(\t)$ (if the intersection is $\emptyset$, then there is nothing to show), there is a subsequence $\tilde m$ of $m$ such that $I^{2,\tilde m}(\t)\to x$ as $\tilde m\to+\infty.$ On the other hand, since $x^1_\t\in \mathcal{I}(\t),$ there exists another subsequence $\ov m$ of $m$ such that $I^{2,\ov m}(\t)\to x^1_\t$ as $\ov m\to+\infty.$ The inequality \eqref{eq:Cauchy} implies that
\begin{align*}
\frac{1}{n^2}\int_{A_n(\t)}\rho^{2,\tilde m}_t\dd x \le W_2^2(\rho^{2,\tilde m}_\t,\rho^{2,\ov m}_\t) \le\frac{1}{n^3} \end{align*}
for $n$ large enough and $\tilde m$ and $\ov m$ larger than $N(n)$. The first inequality holds because of the fact that $x-x^1_\t\ge 1/n$. Thus for all $\t\in(s,t)$
$$
\int_{A_n(\t)}\rho^{2,m}_\t \dd x \leq 1/n,\ \ \hbox{ for any } m \mbox{ large enough.}
$$

Considering any smooth test function $\phi$ supported in $A_n\cup\cJ^+$, the weak formulation  \eqref{eq:weak_passing_limit} together with the fact that $\E^{2,m}$ and $\rho^{2,m}$ vanish on $\cJ^+$ yield
\begin{align*}
\left|\int_{A_n} \E^{2,m}\cdot \partial_x\phi\dd t\otimes\dd x\right|& \leq  \sup_{A_n} |\partial_t\phi|\frac{T}{n}+\sup_{A_n}|\phi|\frac{2}{n},
\end{align*}
for sufficiently large $m$. 
This together with the fact that the measure of $B_n$ is at most $2T/n$ and $\mathcal{I}\subset A_n\cup B_n$ implies that  $\int_{\mathcal{I}} E^{2,m}\cdot\partial_x\phi\dd t\otimes\dd x$ is arbitrary small provided $m$ is large enough, which implies in particular that $\int_{\mathcal{I}} E^{2,\infty}\cdot\partial_x\phi\dd t\otimes\dd x$ is as small as we would like.

\smallskip 

The above claim shows that one needs to describe the weak limit $\E^{2,\infty}$ only on the set $\cJ^-.$ When restricted to $\cJ^-$, one can write 
$$\E^{2,m}=-\partial_x (\rho^{2,m})^m-\partial_x\Phi_2\rho^{2,m}=:-\partial_x p_m-\partial_x\Phi_2\rho^{2,m},$$
and the second term on the r.h.s. of the previous formula passes to the limit due to the weak convergence of $(\rho^{2,m})_{m>1}$. Let us consider a smooth test function $\phi$ compactly supported in $\cJ^-$. Then one has
\begin{align*}
\lim_{m\to+\infty}-\int_{\cJ^-}\partial_x p_m\cdot\partial_x\phi\dd \t\otimes\dd x&=\lim_{m\to+\infty}\int_{\cJ^-} p_m\partial^2_{xx}\phi\dd \t\otimes\dd x=\lim_{m\to+\infty}\int_{\cJ^-} (\rho^{2,m})^m\partial^2_{xx}\phi\dd \t\otimes\dd x\\
&=\lim_{m\to+\infty}\int_s^t\int_\Om(\rho^{1,m}+\rho^{2,m})^m\partial^2_{xx}\phi\dd x\dd \t=\int_s^t\int_\Om p^{\infty}\partial^2_{xx}\phi\dd x\dd \t\\
&=-\int_s^t\int_\Om \partial_x p^{\infty}\partial_{x}\phi\dd x\dd \t=-\int_{\cJ^-}\partial_x p^\infty\cdot\partial_x\phi\dd \t\otimes\dd x
\end{align*}
Hence when restricted to $\cJ^-,$ $\E^{2,\infty}= -\partial_x p^\infty-\partial_x\Phi_2\rho^{2,\infty}.$  %As we have seen in the Claim, on $\mathcal{I}\cup \cJ^+$ $E^{2,\infty}$ vanishes, thus
%$$\E^{2,\infty}=-\partial_x p^\infty-\partial_x\Phi_2\rho^{2,\infty}\ \ \mbox{on}\ [0,T]\times\Om$$ as desired. 
Note that $\rho^{2,\infty}=0$ as well in $\mathcal{I}\cup \cJ^+$, so we can write
$$
\E^{2,\infty} = -\partial_x p^{\infty}\chi_{\{\rho^{2,\infty}>0\}} -\partial_x\Phi_2\rho^{2,\infty}.
$$
Below we will show that $p^{\infty}$ vanishes in $\rho^{2,\infty}<1$. This allows us to write 
$$
\E^{2,\infty} = -\partial_x p^{\infty}\rho^{2,\infty} -\partial_x\Phi_2\rho^{2,\infty}.
$$
Similar reasoning yields the concrete form of $\E^{1,\infty}$ as well. 

\medskip

Let us show that $\|\rho^{i,\infty}\|_{L^\infty}\le 1$. By Lemma \ref{lem:m_limit} we know that  for a.e. $t\in[0,T]$,  $\ds\int_{\Om} (\rho^{i,m}_t)^m\dd x \leq C$, where the constant $C$ is independent of $m$. Thus for any $\d>0$, on the set where $\rho^{i,m}_t\ge 1+\d$ a.e. we have by Chebyshev's inequality that
\be\label{eq:chebyshev}
(1+\d)^m \sL^1(\{\rho^{i,m}_t \ge 1+\d\})\le \int_{\{\rho^{i,m}_t \ge 1+\d\}} (\rho^{i,m}_t)^m\dd x\le \int_{\Om} (\rho^{i,m}_t)^m\dd x \leq C.
\ee
This implies 
$$
\sL^1(\{\rho^{i,m}_t \ge 1+\d\}) \leq C/(1+\d)^m \to 0 \hbox{ as } m\to\infty,
$$
and thus by the arbitrariness of $\d>0$ we conclude $\rho^{i,\infty} \leq 1$ a.e in $[0,T]\times\Om$.  

\medskip

At last, it remains to show that $p^\infty$ is supported in the region $\{\rho^{1,\infty}+\rho^{2,\infty}=1\}$. Notice that since Theorem \ref{thm:sep_limit} yields that $\sL^1(\{\rho^{1,\infty}_t>0\}\cap\{\rho^{2,\infty}_t>0\})=0$ for all $t\in[0,T]$, it is enough to show that $p^\infty(1-\rho^{i,\infty})=0$ a.e. in $\{\rho^{i,\infty}>0\} $ $i=1,2$. We show this property only in the case of $i=2$, the other case is analogous. Let us use the notations 
\begin{align*}
& p_m:=(\rho^{1,m}+\rho^{2,m})^m=_{\ae}(\rho^{2,m})^m,\ \ \ \mbox{in}\ \{\rho^{2,\infty}>0\}\subseteq_{\ae}\cJ^-\\
& \tilde p_m:=(\rho^{1,m}+\rho^{2,m})^{m-1/2}=_{\ae}(\rho^{2,m})^{m-1/2}=(p_m)^{1-\frac{1}{2m}},\ \ \ \mbox{in}\ \{\rho^{2,\infty}>0\}\subseteq_{\ae}\cJ^-
\end{align*}

{\it Claim 2. When $m\to+\infty$, $\ds\int_0^T\int_\Om p_m(1-\rho^{2,m})\dd x\dd t$ and $\ds\int_0^T\int_\Om \tilde p_m(1-\rho^{2,m})\dd x\dd t$ are arbitrary small.}

Let $\d>0$ be small. To show the claim, first observe that $[0,T]\ni t\mapsto \int_\Om\chi_{\{\rho^{2,m}_t\ge 1+\d\}}\dd x$ is a measurable function and by \eqref{eq:chebyshev}, $[0,T]\ni t\mapsto \int_\Om\chi_{\{\rho^{2,m}_t\ge 1+\d\}}\dd x$ is integrable. Similar properties are valid for other characteristic functions of the (sub)level sets. Thus, 
\begin{align*}
\left|\int_0^T\int_\Om p_m(1-\rho^{2,m})\dd x\dd t\right|&\le \int_0^T\int_{\{\rho^{2,m}_t<1-\d\}} p_m(1-\rho^{2,m})\dd x\dd t+\int_0^T\int_{\{1-\d\le \rho^{2,m}_t \le 1+\d\}} p_m|1-\rho^{2,m}|\dd x\dd t\\
&+\int_0^T\int_{\{\rho^{2,m}_t>1+\d\}} p_m(\rho^{2,m}-1)\dd x\dd t\\
&\le (1-\d)^m T\sL^1(\Om)+\d C\\
&+\|p_m\|_{L^{q_1}([0,T]\times\Om)}\|\rho^{2,m}\|_{L^{q_2}([0,T]\times\Om)}\left(\int_0^T\int_{\{\rho^{2,m}_t>1+\d\}}1\dd x\dd t\right)^{1/q_3}\\
& \le (1-\d)^m T\sL^1(\Om)+\d C+\left(\frac{TC}{(1+\d)^m}\right)^{1/q_3},
\end{align*}
where $\frac{1}{q_1}+\frac{1}{q_2}+\frac{1}{q_3}=1$, $q_1<r,\ q_2\le m$ and $C$ is independent of $m$. Here we used also the uniform estimations from Lemma \ref{lem:m_limit}(1) and Theorem \ref{thm:existence_PME}. Now the last sum is as small as desired by choosing $\d>0$ small and $m$ large enough, which shows the first part of the claim. A similar reasoning and Lemma \ref{lem:m_limit}(2) yield the second part of the claim.

Last, one can use the same arguments as in \cite[Lemma 3.4, Step 3. in the proof of Theorem 2.1.]{MauRouSan1} to conclude that $p^\infty(1-\rho^{2,\infty})=0$ a.e. in $\cJ^-$. Indeed, we know that $\tilde p_m$ is uniformly bounded in $L^2([0,T];H^1(\Om))$ and by Lemma \ref{lem:m_limit}(2) $\rho^{2,m}_t$ is uniformly bounded for a.e. $t\in[0,T]$. These together with {\it Claim 2} imply in particular that mentioned results from \cite{MauRouSan1} can be applied to obtain that 
$$\tilde p^\infty(1-\rho^{2,\infty})=0,\ \ \mbox{a.e. in } \cJ^-,$$
where $\tilde p^\infty$ is the weak limit of $(\tilde p_m)_{m>1}$ in $L^2([0,T];H^1(\Om))$ as $m\to+\infty$.
It remains to show only that $p^\infty=\tilde p^\infty$ a.e. in $\cJ^-.$ This is straight forward. Indeed, notice  that $\tilde p_m=(p_m)^{1-\frac{1}{2m}}$ a.e. in $\cJ^-$. Lemma \ref{lem:m_limit} implies that for a.e. $t\in[0,T]$ both $p_m(t,\cdot)$ and $\tilde p_m(t,\cdot)$ are (uniformly) H\"older continuous and converge uniformly (up to passing to a subsequence). This means that $p^\infty_t$ and $\tilde p^\infty_t$ are the uniform limits for a.e. $t\in[0,T]$ and $p^\infty_t=\tilde p^\infty_t$ for a.e. $t\in[0,T]$. The result follows.

%Let us consider now $0\le s<t\le T$ and consider 
\end{proof}

\medskip

\subsection{Characterization of the pressure in the case of $m=+\infty$}

Here we establish the optimal regularity of the pressure, which is Lipschitz continuity for a.e. time. The pressure can be discontinuous in time even in the single density case, when two components of the congested density zone merge into one (see the discussion in \cite{Kim} for instance.)

\begin{proposition}
Let $(\rho^{1,\infty},\rho^{2,\infty},p^\infty)$ be a solution of the system \eqref{eq:PME_infty} in the setting of \eqref{fmw:1D_ordered}. Then $p^{\infty}(t,\cdot)$ is uniformly Lipschitz continuous for a.e. $t\in[0,T]$. Moreover, $p^{\infty}(t,\cdot)$ is as smooth as $\Phi_i$ in the interior of the sets $\{\rho^{i,\infty}=1\}$, $i=1,2$ for a.e. $t\in[0,T]$.
\end{proposition}

Let us remark first that $p^{\infty}$ may have positive boundary data on $\partial\{\rho^{1,\infty}=1\}\cap \partial\{\rho^{2,\infty}=1\}$. Also, notice that the set $\{\rho^{i,\infty}=1\}$ may have empty interior.

\begin{proof}
%Let $0\le s< t\le T$. Note that, by definition of the weak solution, we have
%$$
%\frac{1}{t-s}\int^t_s\int_\Om \left(\partial_x p^{\infty} + \partial_x\Phi_i \rho^{i,\infty}\right) \partial_x\phi \dd x\dd\tau = \int_\Om \dfrac{\rho^{i,\infty}_s - \rho^{i,\infty}_t}{t-s} \phi \dd x,
%$$
%for every stationary function $\phi \in C^{\infty}_c(\Om)$.  Since $\rho^{i,\infty} \leq 1$ a.e., this means that if $\phi$ is supported in the interior of $\{\rho^{i,\infty}_\tau=1\}$ we have that the r.h.s. is nonnegative by choosing $s=\tau$ and $t>\tau$. Thus we have
%
%$$
%p^{\infty,-}(\cdot,t):=\lim_{s\to t^-} \frac{1}{t-s}\int^t_s p^{\infty}(\cdot,\tau) d\tau \hbox{ satisfies } -\partial^2_xx p^{\infty} \leq \partial_xx\Phi_i \hbox{ in the interior of } \{\rho^i(\cdot,t)=1\}.
%$$
%
%\medskip
%
%I believe that $p^{\infty,+}=p^{\infty}$ a.e. time \textcolor{blue}{this I am not sure but probably can be shown first taking spatial average and using Holder regularity in space...}. 
%
% Likewise if we choose $s>\tau$ then we get an opposite sign, so we have, for a.e. time,
%
%\textcolor{blue}{as we discussed we should only discuss the sum.}

Considering the sum of the two weak equations tested against smooth test functions supported in the interior of $\{\rho^{1,\infty}_t+\rho^{2,\infty}_t=1\}$ one obtains that

\begin{equation}\label{semi_convexity}
-\partial^2_{xx} p^{\infty}(t,\cdot) = \partial_x(\partial_x\Phi_1\rho^{1,\infty}+ \partial_x\Phi_2\rho^{2,\infty}),  \hbox{ in the interior of }\{\rho^{1,\infty}_t+\rho^{2,\infty}_t=1\},
\end{equation}
for a.e. $t\in[0,T]$ with homogeneous Dirichlet boundary data (this is because of the fact that $p^{\infty}(t,\cdot)$ is H\"{o}lder continuous a.e. $t\in[0,T]$). Since $\sL^1\left(\{\rho^{1,\infty}_t>0\}\cap\{\rho^{2,\infty}_t>0\}\right)=0$ and the sets $\{\rho^{1,\infty}_t>0\}$ and $\{\rho^{2,\infty}_t>0\}$ are ordered in the sense of  \eqref{fmw:1D_ordered}, one gets

\begin{equation}\label{1}
-\partial^2_{xx} p^{\infty}(t,\cdot) = \partial_{xx}^2\Phi_i, \ \hbox{  in the  interior of } \{\rho^{i,\infty}_t=1\}.
\end{equation}
Therefore $p^\infty(t,\cdot)$ is as smooth as $\Phi_i$ in the interior of the sets  $\{\rho^{i,\infty}_t=1\}$, $i=1,2$, for a.e. $t\in[0,T]$. Let us remark also that since  $p^\infty(t,\cdot)$ is H\"older continuous, the set $\{p^\infty(t,\cdot)>0\}\subseteq\{\rho^{1,\infty}_t+\rho^{2,\infty}_t=1\}$ is open. Thus $p^\infty(t,\cdot)$ is as smooth as $\Phi_i$ in $\{p^\infty(t,\cdot)>0\}\cap {\rm{int}}\{\rho^{i,\infty}_t=1\}.$

\smallskip

Let us show that $p^{\infty}(t,\cdot)$ is Lipschitz continuous for a.e. $t\in[0,T]$. Notice that by the previous arguments, $p^\infty(t,\cdot)$ fails to be differentiable in at most countably many points of $\Om$, and except these points it is smooth. By the fact that $p^\infty(t,\cdot)$ is a Sobolev function, we know that it is absolutely continuous for a.e $t\in[0,T]$. Moreover 
$$|\partial_x p^\infty(t,\cdot)|\le C,\ \ \ae\ \iin\ \Om,$$
where $C>0$ is a positive constant that depends only on $\max\{\|\partial_x\Phi_i\|_{L^\infty}:i=1,2\}.$
This together with the absolute continuity imply that $p^\infty(t,\cdot)$ Lipschitz continuous for a.e. $t\in[0,T].$

%{\color{red}Q: is it possible that the pressure is singular in the ``meeting point'' $\{\rho^{1,\infty}_t=1\}\cap \{\rho^{2,\infty}_t=1\}$?}

\end{proof}

\subsection{Patch solutions}

\begin{proposition}\label{prop:patch}
Let $(\rho^{1,\infty},\rho^{2,\infty},p^\infty)$ be a solution of the system \eqref{eq:PME_infty} in the setting of \eqref{fmw:1D_ordered}. Let us suppose that $\rho^{1,\infty}_0$ and $\rho^{2,\infty}_0$ are \emph{patches}, i.e. $\rho^{i,\infty}_0=\chi_{A^i_0}$ for open intervals $A^i_0$ in $\Om$, $i=1,2$. We suppose moreover that the drifts $-\partial_x\Phi_i$, $i=1,2$ are compressive, meaning that $\partial^2_{xx}\Phi_i \ge 0$ on $\Om$. Then $\rho^{i,\infty}_t$, $i=1,2$ is a patch for all $t\in[0,T]$, i.e. there exists $\{A^i_0(t)\}_{t\in[0,T]}$: a family of open intervals such that $\rho^{i,\infty}_t=\chi_{A^i_0(t)}.$
\end{proposition}

\begin{remark}
While we believe our argument can be extended to measurable sets instead of open intervals, we do not pursue this generalization for simplicity.
\end{remark}

\begin{proof}

%\begin{itemize}
%\item[1.] Show that $\rho^{i,\infty}=1$ in the support of the patch moving with the normal velocity $\partial_x p\chi_{\{\rho^i=1\}} + \partial_x\Phi_i$;
%(given that above proof is correct this seems doable with some smoothing argument)\\
%\item[2.] Show that in that region the volume of $\rho$ is nondecreasing (using the compressibility of $\Phi$)\\
%\item[3.] Therefore the only option for $\rho$ is to stay zero outside of that region and stay in the incompressible region (that is where $p$ is nonzero).
%\end{itemize}

%We need to be careful about the singularity of $\partial_x p$ on the boundary points of $\{\rho^i=1\}$. There are two cases. When $(x_0,t)\in\partial\{\rho^1+\rho_2=1\}$, there the other density is zero and so $p(x_0,t)=0$. This makes $\partial_{xx} p$  convex (or positive dirac delta) at that point, which is the correct sign for the argument to work.  That leaves the case $(x_0,t) \in \partial\{\rho_1=1\}$ and also in the interior of $\{\rho_1+\rho_2=1\}$. Due to the disjoint support property of the densities, this is the left-most point in the set $\{\rho_1>0\}$, in particular $\rho_1$ vanishes to the right of this point. This means we do not have to worry about this point, in the sense that in the weak expression we had $\rho_1$ attached to every term, and thus here the singularity of $\partial_xx p$ does not show up (we can just take one-sided limit from the support of $\rho_1$. 

\medskip

%Below is the rigorous argument with above heuristics in mind. 

Let us recall that the sets $\{\rho^{2,\infty}_t>0\}$ and $\{\rho^{1,\infty}_t>0\}$ are ordered for all $t\in[0,T]$ in the sense of \eqref{fmw:1D_ordered}, with $\rho^{2,\infty}$ supported to the left of $\rho^{1,\infty}$. We show the proposition only for $\rho^{2,\infty}$, the case of $\rho^{1,\infty}$ is analogous. 

We define an extension $\tilde{p}$ of $p^\infty$ to the right of the support of $\rho^{2,\infty}$. For a.e. $t$ where $p^\infty$ satisfies \eqref{1}, let $x_0(t):=\sup\Leb \left(\{\rho^{2,\infty}_t>0\}\right)$. If $p^\infty(t,x_0(t))=0$ then we let $\tilde{p}(t,\cdot)=p^\infty(t,\cdot)$. If $p(t,x_0(t))>0$, this means $x_0(t)$ lies in the interior of $\{\rho^{1,\infty}_t+\rho^{2,\infty}_t=1\}$. In this case let us define $\tilde{p}$ to be a $C^2$ extension of $p^\infty$ to the right of $x_0(t)$ such that $-\partial^2_{xx}\tilde{p} = \partial^2_{xx}\Phi_2$. 

\medskip

Let us consider the test function $\phi$, as the solution of the transport equation 
$$\partial_t\phi_t - v^\e \partial_x\phi = 0$$ 
with initial condition $\phi(0,\cdot) = \chi_{A^2_0}$, where we define
$v^\e:= v\star \eta_\e$ with $v=(\partial_x \tilde{p}+ \partial_x\Phi_i)$ and $\eta_\e$ is a standard mollifier (the mollification being performed only w.r.t. the space variable). Then $ \|v^\e(t,\cdot)-v(t,\cdot)\|_{L^q}\le C\e$ in the set $\{\rho^{2,\infty}_t>0\}$ for any $q>0$ and a.e. $t\in[0,T]$, where $C$ depends on the Lipschitz constant of $p^\infty$ and $\Phi_2$.  Then, from the weak expression we have
$$
\left|\int_\Om (\rho^{2,\infty}\phi)(t,\cdot)\dd x -\int_\Om (\rho^{2,\infty}\phi)(0,\cdot)\dd x\right| = \left|\int^t_0\int_\Om(v_\e-v) \rho^{2,\infty}\partial_x\phi \dd x \dd\tau\right | \le C\e.
$$
Since $\phi$ solves a transport equation with spatially smooth velocity which is integrable in time, this equation is well-posed and $\phi$ can be represented using the method of characteristics. In particular $\phi(t,\cdot) = \chi_{A^2_t}$ for some measurable set $A^2_t$ for each time $t>0$.  Hence the l.h.s. of the above equation is
\begin{equation}\label{111}
\int_\Om (\rho^{2,\infty}\phi)(t,\cdot)\dd x -\int_\Om (\rho^{2,\infty}\phi)(0,\cdot)\dd x = \int_{A^2_t} \rho^{2,\infty}_t\dd x -\int_{A^2_0} 1\dd x.
\end{equation}

Now we claim that 
\begin{equation}\label{compressive}
\partial_x v^\e \geq 0.
\end{equation}
 This is true because $\tilde{p}$ satisfies $-\partial^2_{xx} \tilde{p} =\partial^2_{xx}\Phi_2$  (or $\tilde p=p^\infty$) to the right of $x_0(t)$ as well as in any interior point of  $\{\rho^{2,\infty}_t=1\}$, and to the left of $x_0(t)$ we have $\tilde{p}= p^\infty= \max\{p^\infty,0\}$, which makes $p^\infty$ convex at any boundary point of the set $\{\rho^{2,\infty}=1\}$ to the left of $x_0(t)$. Thus we conclude that 
$$
\partial_x v = \partial^2_{xx} p^\infty + \partial^2_{xx}\Phi_2 \geq 0 \hbox{ on } (-\infty,x_0(t))\cap\Om,
$$
in the distributional sense, which yields $\partial_x v^\e \geq 0$.
 
 \medskip
 
 Due to \eqref{compressive} we have $\sL^1(A^2_t) \leq \sL^1(A^2_0)$. This, the fact that $\rho^{2,\infty}_t \leq 1$ a.e., and \eqref{111} yield that 
 $$
\int_{A^2_t} (1-\rho^{2,\infty}_t)\dd x\le  \int_{A^2_t} (1-\rho^{2,\infty}_t)\dd x + \sL^1(A^2_0) -\sL^2(A^2_t) \le C\e.
$$

%It follows that $\rho_1(\cdot,t) \geq 1-C\e$ in $A(t)$.

\medskip

We still need to let $\e\da 0$ to achieve the desired result. If $A^2_0$ is an interval then $A^2_t$ is an interval (that depends on $\e>0$) with uniformly bounded velocity with respect to $\e$, hence along a subsequence the endpoints (as a function of $t$ they are equicontinuous) uniformly converge to limiting endpoints $(a(t), b(t))_{\{t>0\}}$ as $\e\da 0$. Let $A_0^2(t) := (a(t), b(t))$. Here $\rho^{2,\infty}_t$ should be identically one (because of the previous inequality). But this means that along this subsequence, $v$ was incompressible except in a small set in $A^2_t$, so that makes $\sL^1(A^2_t)$ very close to $\sL^1(A^2_0)$ and $|b(t)-a(t)| = \sL^1(A^2_0)$. Since $\rho^{2,\infty}$ preserves mass over time, this means that 
$$
\rho^{2,\infty}_t = \chi_{A_0^2(t)}.
$$
\end{proof}

\medskip

It remains to describe the evolution of the patches $\{\rho^{i,\infty}=1\}$. As we see in \cite{KimPoz} in \cite{MPQ}, the evolution laws are different depending on whether there are regions of the densities with values between zero and one.  In the above Proposition, we have patch solutions supported on an interval, and the  continuity of the densities over time  in $W_2$-distance yields that each patch $\{\rho^{i,\infty}=1\}$ evolves continuously in time. Therefore it follows that the space-time interior of those sets taken at time $t$ equals spatial interior at time $t$. Thus from \eqref{1} we have $p^{\infty}(t,\cdot)\in C^2$ at every time in the interior of $\{\rho^{1,\infty}_t+\rho^{2,\infty}_t=1\}$.

\medskip

\begin{remark}

With the aforementioned regularity of $p^{\infty}$ and $\{\rho^{i,\infty}=1\}$ at hand, one can verify with test functions in the weak formulation that the following holds: the velocity law on one-phase boundary points  is given by
\begin{equation}\label{2}
V=  \nu^i_x (-\partial^i_x p^{\infty}  -\partial_x\Phi_i)\quad {\rm{on}}\ \  \partial\{\rho^{i,\infty}=1\},
\end{equation}
where $\nu_x$ is the outward normal of the set $\{\rho^{i,\infty}=1\}$, $V$ is the normal velocity of the interface and  $\partial^i_x$ denotes the $x$-derivative taken from the interior of the set $\{\rho^{i,\infty}=1\}$.  This yields the flux matching across different densities,
\begin{equation}\label{3}
\partial^1_x p^{\infty}_x +\partial_x\Phi_1 = \partial^2_x p^{\infty}_x +\partial_x\Phi_2 \quad {\rm{on}}\ \ \partial\{\rho^{1,\infty}=1\} \cap \partial\{\rho^{2,\infty}=1\}.
\end{equation}
The equations \eqref{1}, \eqref{2} and \eqref{3} corresponds to a generalized two-phase Hele-Shaw flow evolving by the pressure variable $p^{\infty}$ where different drift potentials are present for each phase $\rho^{i,\infty}$.

\end{remark}

\appendix

\section{Optimal transport toolbox}\label{sec:appendix_ot}

\begin{lemma}\label{cor:opt_cond}
Let $f:[0,+\infty)\to\R$ be a $C^1$ convex function that is superlinear at $+\infty.$ Let $M>0.$ We consider $\cF:\sP^M(\Om)\to\R\cup{+\infty}$ defined as 
$$\cF(\rho)=\left\{
\ba{ll}
\ds\int_\Om f(\rho(x))\dd x, & {\rm{if}}\ \rho\ll\sL^d,\\
+\infty, & {\rm{otherwise}}.
\ea
\right.
$$
Let $\nu\in\sP^{M}(\Om)$ be given. Then there exists a solution $\varrho\in\sP^{\ac,M}(\Om)$ of the minimization problem 
$$\min_{\rho\in\sP^M(\Om)}\left\{\cF(\rho)+\frac12 W_2^2(\rho,\nu)\right\}.$$
If in addition $\nu\ll\sL^d$ or if $f$ is strictly convex, then $\varrho$ is unique.

Moreover, $\exists C\in\R$ such that for a suitable Kantorovich potential $\vphi$ in the optimal transport of $\varrho$ onto $\nu$ one has the following first order necessary optimality condition fulfilled 
\be\label{cond:opt}
\left\{
\ba{ll}
f'(\varrho)+\vphi=C, & \varrho-{\rm{a.e.}},\\
f'(\varrho)+\vphi\ge C, & {\rm{on}}\ \{\varrho=0\}.
\ea
\right.
\ee
If $f'(0)$ is finite, then one can express the above condition as $f'(\varrho)=\max\{C-\varphi, f'(0)\}.$
\end{lemma}

\begin{proof} The proof of the previous results can be found in \cite{ButSan} or \cite[Chapter 7]{OTAM}.
%The existence of an optimizer $\varrho$ is the consequence of the lower semicontinuity of the objective functional. Notice that the objective functional is convex. In addition, if one of the functionals becomes strictly convex -- which is the case if $f$ is strictly convex, or $\nu\ll\sL^d$ --and  then $\varrho$ is unique.
%
%The optimality condition \eqref{cond:opt} can be obtained via the ``vertical perturbation'' technique (as in Theorem \ref{thm:classical_OT2} (2)). See for instance \cite{ButSan} or \cite[Chapter 7]{OTAM} for the details. If one assumes that $\spt(\nu)=\Om,$ then the Kantorovich potential in \eqref{cond:opt} is exactly $\hat\vphi$ just in Theorem \ref{thm:classical_OT2} (2). Otherwise, one can perform an approximation of $\nu$ by strictly positive measures on $\Om$ and the results follow by a $\Gamma-$convergence argument (see \cite[Lemma 3.6]{ButSan}). In this case since the Kantorovich potentials are not unique, the optimality condition holds for one suitable one.
\end{proof}

It turns out that $(\sP^M(\Om),W_2)$ is a geodesic space and constant speed geodesics (and absolutely continuous curves in general) can be characterized by special solutions of continuity equations. Since this characterization is true for any $M>0,$ we simply set $M=1$ in the theorem below.

\begin{theorem}[see \cite{AmbGigSav, OTAM}]
\begin{enumerate}
\item Let $\Om\subset\R^d$ compact and $(\mu_t)_{t\in[0,T]}$ be an absolutely continuous curve in $(\sP(\Om),W_2).$ Then for a.e. $t\in[0,T]$ there exists a vector field $\v_t\in L^2_{\mu_t}(\Om;\R^d)$ s.t.
\begin{itemize}
\item the continuity equation $\partial_t\mu_t+\nabla\cdot (\v_t\mu_t)=0$ is satisfied in the weak sense;
\item for a.e. $t\in[0,T],$ one has $\|v_t\|_{L^2_{\mu_t}}\le |\mu'|_{W_2}(t),$ where 
$$|\mu'|_{W_2}(t):=\lim_{h\to 0}\frac{W_2(\mu_{t+h},\mu_t)}{|h|}$$ denotes the metric derivative of the curve $[0,T]\ni t\mapsto \mu_t$ w.r.t. $W_2,$ provided the limit exists.
\end{itemize}
\item Conversely, if $(\mu_t)_{t\in[0,T]}$ is a family of measures in $\sP(\Om)$ and for each $t$ one has a vector field $\v_t\in L^2_{\mu_t}(\Om;\R^d)$ s.t. $\int_0^T\|\v_t\|_{L^2_{\mu_t}}\dd t<+\infty$ and $\partial_t\mu_t+\nabla\cdot(\v_t\mu_t)=0$ in the weak sense, then $[0,T]\ni t\mapsto \mu_t$ is an absolutely continuous curve in $(\sP(\Om),W_2)$, with $|\mu'|_{W_2}(t)\le \|v_t\|_{L^2_{\mu_t}}$ for a.e. $t\in[0,T]$ and $W_2(\mu_{t_1},\mu_{t_2})\le \int_{t_1}^{t_2} |\mu'|_{W_2}(t)\dd t.$ If moreover $|\mu'|\in L^2(0,T)$, then we say that $\mu$ belongs to the space $AC^2([0,T];(\sP(\Om), W_2)).$
\item For curves $(\mu_t)_{t\in[0,1]}$ that are geodesics in $(\sP(\Om),W_2)$ one has the equality
$$W_2(\mu_0,\mu_1)=\int_0^1 |\mu'|_{W_2}(t)\dd t=\int_0^1 \|v_t\|_{L^2_{\mu_t}}\dd t.$$
\item For $\mu_0,\mu_1\in\sP^{\ac}(\Om)$, a constant speed geodesic connecting them is a curve $(\mu_t)_{t\in[0,1]}$ such that $W_2(\mu_s,\mu_t)=|t-s|W_2(\mu_0,\mu_1)$ for any $t,s\in[0,1].$ One can compute this constant speed geodesic using McCann's interpolation, i.e. $\mu_t:=\left(T_t\right)_\#\mu_0,$ for all $t\in[0,1]$, where $T_t:=(1-t)\id +tT$ with $T_\#\mu_0=\mu_1$ the optimal transport map between $\mu_0$ and $\mu_1$. Moreover, the velocity field in the continuity equation is given by $\v_t:=(T-\id)\circ (T_t)^{-1}.$
\end{enumerate}
\end{theorem}

Let us introduce the {\it Benamou-Brenier functional} $\cB_2:\sM([0,T]\times \Om)\times \sM^d([0,T]\times\Om)\to\R\cup\{+\infty\}$ defined as 
$$
\cB_2(\mu, \E):=\left\{
\begin{array}{ll}
\ds\int_0^T\int_\Om |\v_t|^2\dd\mu_t(x)\dd t, & \text{if}\ \E=\E_t\otimes\dd t, \mu=\mu_t\otimes\dd t \ \text{and}\  \E_t=\v_t\cdot\mu_t,\\
\ds+\infty, & \text{otherwise}.
\end{array}
\right.
$$
%In this definition we evaluate $\cB_2$ in general for $(\mu_t)_{t\in[0,T]}$, a Borel family of probability measures on $\Om$ and $\E$ an element of $\sM^d([0,T]\times\Om)$, and $B_2(\mu,\E)$ is finite only if $\E$ is induced by a Borel family $(\E_t)_{t\in[0,T]}$ of measures from $\sM^d(\Om)$ such that $\E=\E_t\otimes\dd t$. More precisely $\int_{[0,T]\times\Om}\phi(t,x)\cdot \dd \E(t,x)=\int_0^T\dd t\int_{\Om}\phi(t,x)\cdot\dd\E_t(x)$ for any $\phi\in C([0,T]\times\Om;\R^d).$ 

It is well-known (see for instance \cite[Proposition 5.18]{OTAM}) that $\cB_2$ is jointly convex and lower semicontinuous w.r.t. the weak$-\star$ convergence. In particular if $(\mu,\E)$ solves $\partial_t\mu+\nabla\cdot\E=0$ in the weak sense with $\cB_2(\mu,\E)<+\infty,$ implies that $t\mapsto\mu_t$ is a curve in $AC^2([0,T];(\sP(\Om), W_2)).$ 
%We have the {\it Benamou-Brenier formula}, i.e. for any $\mu,\nu\in\sP(\Om)$
%$$\frac{1}{T}W_2^2(\mu,\nu)=\inf\left\{ \cB_2(\rho,\E): \partial_t\rho+\nabla\cdot \E=0,\rho_0=\mu;\rho_T=\nu\right\}.$$
%Notice that the optimum in the above formula is realized by a pair $(\rho,\E)$ such that $\E=\v\rho$ and $\rho$ is a constant speed geodesic connecting $\mu$ and $\nu$.

%Let $M_1,M_2>0$. For $\mu^1\in\sP^{M_1}(\Om)$ and $\mu^2\in\sP^{M_2}(\Om)$ we denote $\bm:=(\mu^1,\mu^2)$ the pair of these measures, i.e. an element of the product space $\sP^{M_1}(\Om)\times\sP^{M_2}(\Om)$. 

%We denote by $\bW_2$ the induced $2-$Wasserstein distance on the product space $\sP^{M_1}(\Om)\times\sP^{M_2}(\Om)$, i.e. $\bW_2^2(\bm,\bn):=W_2^2(\mu^1,\nu^1)+W_2^2(\mu^2,\nu^2),$ where $\bm:=(\mu^1,\mu^2)$, and $\bn:=(\nu^1,\nu^2).$

The following comparison result appears to be well-known but we write it here for completeness.

\begin{lemma}\label{lem:ineq_sum_1}
Let $\mu^1,\nu^1\in\sP^{M_1}(\Om)$ and  $\mu^2,\nu^2\in\sP^{M_2}(\Om).$ Then the following inequality holds true
\begin{equation}\label{ineq:sum_1}
W_2^2(\mu^1+\mu^2,\nu^1+\nu^2)\le W_2^2(\mu^1,\nu^1) + W_2^2(\mu^2,\nu^2).
\end{equation}
\end{lemma}
\begin{remark}
Note that with the abuse of notation, $W_2$ on the l.h.s. of \eqref{ineq:sum_1} denotes the $2-$Wasserstein distance on $\sP^{M_1+M_2}(\Om),$ while on the r.h.s. $W_2$ denotes the corresponding distances on $\sP^{M_1}(\Om)$ and $\sP^{M_2}(\Om)$ respectively. 
\end{remark}

\begin{proof}[Proof of Lemma \ref{lem:ineq_sum_1}]
The quantity on the l.h.s. of \eqref{ineq:sum_1} is realized by an optimal plan $\gamma\in\Pi^{M_1+M_2}(\mu^1+\mu^2,\nu^1+\nu^2)$ i.e. 
$$W_2^2(\mu^1+\mu^2,\nu^1+\nu^2)=\int_{\Om\times\Om}|x-y|^2\dd\gamma.$$ 
Similarly the quantities on the r.h.s. can be written with the help of some optimal plans $\gamma^i\in\Pi^{M_i}(\mu^i,\nu^i),$ $i=1,2,$ i.e.
$$W_2^2(\mu^i,\nu^i)=\int_{\Om\times\Om}|x-y|^2\dd\gamma^i.$$
Now set $\tilde\gamma:=\gamma^1+\gamma^2$. Clearly since $(\pi^x)_\#\tilde\gamma=\mu^1+\nu^1$ and $(\pi^y)_\#\tilde\gamma=\mu^2+\nu^2$ one has $\tilde\gamma\in\Pi^{M_1+M_2}(\mu^1+\mu^2,\nu^1+\nu^2)$. Hence
\begin{align*}
W_2^2(\mu^1+\mu^2,\nu^1+\nu^2)&=\int_{\Om\times\Om}|x-y|^2\dd\gamma\le\int_{\Om\times\Om}|x-y|^2\dd\tilde\gamma\\
&=\int_{\Om\times\Om}|x-y|^2\dd\gamma^1+\int_{\Om\times\Om}|x-y|^2\dd\gamma^2\\
&\le W_2^2(\mu^1,\nu^1)+W_2^2(\mu^2,\nu^2).
\end{align*}
Therefore, inequality \eqref{ineq:sum_1} follows.
\end{proof}

\section{A refined Aubin-Lions lemma}\label{sec:appendix_aubin-lions}

In \cite{RosSav} the authors present the following version of the classical Aubin-Lions lemma (see \cite{Aub}):

\begin{theorem}{\cite[Theorem 2]{RosSav}}\label{thm:aubin_refined}
Let $B$ be a Banach space and $\cU$ be a family of measurable $B$-valued function. Let us suppose that there exist a normal coercive integrand $\fF:(0,T)\times B\to [0,+\infty]$, meaning that 
\begin{itemize}
\item[(1)] $\fF$ is $\sB(0,T)\otimes\sB(B)$-measurable, where $\sB(0,T)$ and $\sB(B)$ denote the $\s$-algebgras of the Lebesgue measurable subsets of $(0,T)$ and of the Borel subsets of $B$ respectively;
\item[(2)] the maps $v\mapsto \fF_t(v):=\fF(t,v)$ are l.s.c. for a.e. $t\in(0,T)$;
\item[(3)] $\{v\in B:\fF_t(v)\le c\}$ are compact for any $c\ge 0$ and for a.e. $t\in(0,T),$
\end{itemize}
and a l.s.c. map $g:B\times B\to [0,+\infty]$ with the property
$$\left[u,v\in D(\fF_t),\ g(u,v)=0\right]\Rightarrow u=w,\ \text{for }\ae\ t\in(0,T).$$
If $$\sup_{u\in\cU}\int_0^T\fF(t,u(t))\dd t<+\infty\ \ \text{and}\ \ \lim_{h\da 0}\sup_{u\in\cU}\int_0^{T-h}g(u(t+h),u(t))\dd t=0,$$
then $\cU$ is relatively compact in $\sM(0,T; B).$
\end{theorem}

%
%Here we list some known results that we have used in our analysis. Let $\Om\subset\R^d$ be a compact set. 
%\begin{lemma}\label{lem:conv_power}
%Let $p\in(1,+\infty)$ and let $(f_n)_{n\ge0}$ be a non-negative sequence in $L^p(\Om)$ converging strongly to $f\in L^p(\Om).$ Let $0<q<p.$ Then $(f^q_n)_{\ge 0}$ converges strongly to $f^q$ in $L^{p/q}(\Om)$.
%\end{lemma}
%
%\begin{proof}
%Since $0<q<p,$ the function $\ds [0,+\infty)\ni x\mapsto x^{\frac{p}{q}}$ is convex, hence by Jensen's inequality one obtains
%$$|f^q_n-f^q|^{\frac{p}{q}}\le 2^{\frac{p}{q}}\left(\frac{|f_n|^q+|f|^q}{2}\right)^{\frac{p}{q}}\le 2^{\frac{p}{q}-1}\left(|f_n|^p+|f|^p\right).$$
%By the pointwise convergence of $|f_n|^p$ to $|f|^p$ a.e. and by Lebesgue's dominated convergence theorem we conclude.
%\end{proof}

\vspace{0.4cm}

{\it Conflict of Interest --} The authors declare that they have no conflict of interest.

\end{document}